\documentclass[11pt, reqno]{amsart}
\usepackage{graphicx, amssymb, amsmath, amsthm}
\usepackage[utf8]{inputenc}
\usepackage{epsfig}
\usepackage{hyperref}
\usepackage{comment}
\usepackage{mathrsfs}
\numberwithin{equation}{section}
\usepackage{amsthm}
\usepackage{subfigure}
\usepackage{tikz}
\usepackage{here}
\usepackage{float}
\usepackage{pifont}
\usepackage{enumitem}
\usetikzlibrary{matrix,arrows}
\usetikzlibrary{shapes}
\usetikzlibrary{calc}
\usetikzlibrary{arrows}
\usetikzlibrary{decorations.pathreplacing,decorations.markings}

\usepackage{tikz-cd}
\usepackage{ulem}

\usetikzlibrary{patterns}

\newtheorem{theorem}{Theorem}[section]

\newtheorem{lemma}[theorem]{Lemma}

\newtheorem{corollary}[theorem]{Corollary}
\newtheorem{conjecture}[theorem]{Conjecture}
\newtheorem{proposition}[theorem]{Proposition}

\theoremstyle{definition}
\newtheorem{definition}[theorem]{Definition}

\newtheorem*{rem}{Remark}

\makeatletter
\newcommand{\Extend}[5]{\ext@arrow0099{\arrowfill@#1#2#3}{#4}{#5}}
\makeatother

\DeclareMathOperator{\dist}{dist}

\DeclareMathOperator{\Ric}{Ric}
\DeclareMathOperator{\Div}{div}
\DeclareMathOperator{\area}{area}
\DeclareMathOperator{\osc}{osc}
\DeclareMathOperator{\Rm}{Rm}

\newcommand{\M}{\mathbf M}
\newcommand{\vol}{\operatorname{Vol}}
\newcommand{\R}{\mathbb R}
\newcommand{\bd}{\partial}
\newcommand{\bh}{\overline B}

\newcommand{\del}{\nabla}
\newcommand{\areah}{\overline{\operatorname{Area}}}
\newcommand{\volh}{\overline{\operatorname{Vol}}}
\newcommand{\eps}{\varepsilon}
\newcommand{\inv}{^{-1}}
\newcommand{\N}{\mathbb N}
\newcommand{\supp}{\operatorname{supp}}
\newcommand{\ind}{\operatorname{ind}}
\newcommand{\eval}{\bigg\vert}
\newcommand{\la}{\langle}
\newcommand{\ra}{\rangle}
\newcommand{\grad}{\nabla}

\begin{document}

\title[CMCs in Asymptotically Flat and Hyperbolic Manifolds]{Existence of Constant Mean Curvature Surfaces in Asymptotically Flat and Asymptotically Hyperbolic Manifolds}
\author{Liam Mazurowski}
\address[Liam Mazurowski]{Department of Mathematics, Lehigh University, 17 Memorial Dr E, Bethlehem, PA 18015, United States of America}
\email{lim624@lehigh.edu}

\author{Jintian Zhu}
\address[Jintian Zhu]{Institute for Theoretical Sciences, Westlake University, 600 Dunyu Road, Hangzhou, Zhejiang 310030, People's Republic of China}
\email{zhujintian@westlake.edu.cn}

\begin{abstract}
We prove the existence of compact surfaces with prescribed constant mean curvature in asymptotically flat and asymptotically hyperbolic manifolds. More precisely, let $(M^3,g)$ be an asymptotically flat manifold with scalar curvature $R\ge 0$. Then, for each constant $c>0$, there exists a compact, almost-embedded, free boundary constant mean curvature surface $\Sigma \subset M$ with mean curvature $c$. Likewise, let $(M^3,g)$ be an asymptotically hyperbolic manifold with scalar curvature $R\ge -6$. Then, for each constant $c > 2$, there exists a compact, almost-embedded, free boundary constant mean curvature surface $\Sigma \subset M$ with mean curvature $c$.  The proof combines min-max theory with the following fact about inverse mean curvature flow which is of independent interest: for any $T$ the inverse mean curvature flow emerging out of a point $p$ far enough out in an asymptotically flat (or asymptotically hyperbolic) end will remain smooth for all times $t\in (-\infty,T]$. 
\end{abstract}

\maketitle

\section{Introduction} 
In the early 1980s, Almgren \cite{almgren1965theory}, Pitts \cite{pitts2014existence}, and Schoen-Simon \cite{schoen1981regularity} developed a min-max theory for finding critical points of the area functional in a Riemannian manifold. Their combined work implies that any closed Riemannian manifold of dimension between 3 and 7 contains a closed minimal hypersurface. Given this, it is natural to ask whether closed manifolds also need to contain closed hypersurfaces with non-zero constant mean curvature (CMC). It is not hard to see that there should always exist closed CMC hypersurfaces with both very small and very large mean curvature in the presence of certain non-degeneracy conditions. Indeed, CMC hypersurfaces with mean curvature close to zero can be obtained by perturbing non-degenerate minimal hypersurfaces while, by work of R. Ye \cite{ye1991foliation}, CMC hypersurfaces with very large mean curvature can be obtained by perturbing geodesic spheres near non-degenerate critical points of the scalar curvature. 

It is a deep theorem of X. Zhou and J. Zhu \cite{zhou2019min} that closed manifolds $M$ of suitable dimension actually contain closed CMC hypersurfaces of any given prescribed mean curvature. Their proof is based on the fact that a smooth hypersurface $\Sigma = \bd \Omega$ in $M$ has constant mean curvature $c$ precisely when $\Omega$ is a critical point of the functional 
\[
A^c(\Omega) = \operatorname{Area}(\bd \Omega) - c\vol(\Omega). 
\]
They then develop a variant of the Almgren-Pitts min-max theory capable of finding smooth critical points of the $A^c$ functional. Finally, they use this min-max theory to show that a closed manifold $M$ contains a closed hypersurface of constant mean curvature $c$ for every choice of $c$. 

It is tempting to think that a similar theorem may be true for certain classes of complete, non-compact manifolds. In particular, X. Zhou \cite[Page 2711]{zhou2022mean} has conjectured the following: 
\begin{conjecture}[X. Zhou]\label{Conj: Zhou}
    Any asymptotically flat manifold of low dimension contains at least one closed CMC hypersurface for
any prescribed curvature.
\end{conjecture}
\begin{rem}
 We point out that it is necessary to assume the asymptotically flat manifold to have no boundary in X. Zhou's conjecture. Indeed, Brendle \cite{brendle2013constant} has shown that the coordinate spheres are the only closed CMC surfaces in the Schwarzschild manifolds, which implies that there do not exist closed CMC surfaces with very large mean curvature when $M$ is Schwarzschild.
\end{rem}
Heuristically speaking, a manifold $(M^3,g)$ is called asymptotically flat if $M$ has one end diffeomorphic to $\R^3$ minus a ball and the metric $g$ on $M$ approaches the Euclidean metric sufficiently rapidly at infinity. Asymptotically flat manifolds are important in the study of general relativity. Assuming the metric on $M$ approaches the Euclidean metric at a fast enough rate, Huisken and Yau \cite{huisken1996definition} proved that the end of $M$ can be foliated by constant mean curvature spheres by perturbing coordinate spheres around infinity when the total mass is positive.  CMC foliations near infinity in asymptotically flat manifolds have also been studied by Eichmair-Koerber \cite{eichmair2022foliations}, Huang \cite{huang2010foliations}, Metzger \cite{metzger2007foliations}, Nerz \cite{nerz2015foliations}, and others under weaker assumptions on the asymptotics of the metric. The existence of these foliations shows that asymptotically flat manifolds contain closed surfaces of constant mean curvature $c$ for every sufficiently small $c > 0$.  

Motivated in part by X. Zhou's conjecture on the existence of CMCs for any prescribed curvature, the first named author \cite{mazurowski2022prescribed} adapted the $A^c$ min-max theory to work in certain complete, non-compact manifolds. Let $(M^3,g)$ be an asymptotically flat manifold. Given a constant $c > 0$, define the {\it min-max value} 
\[
\omega_c(M) = \inf_{\{\Omega_t\}_{t\in [0,1]}} \left[\sup_{t\in [0,1]} A^c(\Omega_t)\right],
\]
where the infimum is taken over all {\it continuously varying} families of open sets $\{\Omega_t\}_{t\in[0,1]}$ that satisfy $\Omega_0 = \emptyset$ and $A^c(\Omega_1) < 0$.  By choosing $\{\Omega_t\}_{t\in [0,1]}$ to be a family of concentric balls far out in the end, it is easy to see that the inequality $$\omega_c(M) \le \omega_c(\R^3)$$ 
holds for any asymptotically flat manifold $M$.  

When this inequality is strict, the methods of \cite{mazurowski2022prescribed} give the existence of a closed CMC surface in $M$ with prescribed curvature $c$. 

\begin{theorem}[\cite{mazurowski2022prescribed}]\label{Thm: closed}
Let $(M^3,g)$ be a complete, asymptotically flat manifold with no boundary. Fix a constant $c > 0$ and assume that $\omega_c(M) < \omega_c(\R^3)$. Then there exists a closed, almost-embedded surface $\Sigma$ in $M$ with constant mean curvature $c$. 
\end{theorem}

\subsection{Main theorem and its proof}In this paper, we investigate the existence of compact CMC surfaces with prescribed mean curvature in asymptotically flat $3$-manifolds under the extra assumption that $M$ has non-negative scalar curvature. This is a very natural geometric assumption which, from the physical point of view, represents the fact that matter density should be non-negative. Our main theorem is the following:

\begin{theorem}
\label{main-theorem-flat}
Let $(M^3,g)$ be a complete, asymptotically flat manifold, possibly with non-empty boundary.   Assume that $M$ has non-negative scalar curvature. Then, for every constant $c > 0$, there exists a compact, almost-embedded, free boundary, constant mean curvature surface $\Sigma$ in $M$ with mean curvature $c$. 
\end{theorem}

When $M$ has no boundary, the CMC surface $\Sigma$ produced in Theorem \ref{main-theorem-flat} also has no boundary, so Theorem \ref{main-theorem-flat} gives an affirmative answer to X. Zhou's Conjecture \ref{Conj: Zhou} in dimension three under the assumption of non-negative scalar curvature. 
Since it is customary in the literature to work with asymptotically flat manifolds with outermost minimal boundary, we feel it is worth emphasizing that there are many examples of complete, boundary-free, asymptotically flat manifolds with non-negative scalar curvature, which do not contain any closed minimal surfaces. Such metrics can be constructed, for example, by taking a metric $\bar g$ on $\mathbb S^3$, whose Yamabe quotient satisfies $$Y(\mathbb S^3,[\bar g]) >  2^{-2/3}Y(\mathbb S^3,[g_{\text{round}}]),$$ and then letting 
$$(M,g) = (\mathbb S^3\setminus \{p\}, \Gamma^4 \bar g),$$ where $\Gamma$ is a Green's function for the conformal Laplacian with pole at $p$. It follows from the methods in \cite{bray2004classification} that such a metric cannot contain any closed minimal surfaces. 

Let us briefly discuss the proof of Theorem \ref{main-theorem-flat}. To handle the possible non-empty boundary, we use the free boundary min-max theory of A. Sun, Z. Wang, and X. Zhou \cite{sun2024multiplicity} and establish the following modified version of Theorem \ref{Thm: closed}.
\begin{theorem}
\label{theorem-min-max}
Let $(M^3,g)$ be a complete, asymptotically flat manifold, possibly with boundary. Fix a constant $c > 0$ and assume that $\omega_c(M) < \omega_c(\R^3)$. Then there exists a compact, almost-embedded, free boundary constant mean curvature surface $\Sigma$ in $M$ with mean curvature $c$. 
\end{theorem}

In light of Theorem \ref{theorem-min-max}, in order to prove Theorem \ref{main-theorem-flat}, it suffices to verify the hypothesis $\omega_c(M) < \omega_c(\R^3)$  always holds provided $M$ has non-negative scalar curvature and does not have a Euclidean end. We accomplish this using inverse mean curvature flow.  

Inverse mean curvature flow is the motion of surfaces with speed equal to the reciprocal of the mean curvature. When the ambient manifold is the Euclidean space,  it was shown by Gerhard \cite{gerhardt1990flow}, and independently Urbas \cite{urbas1990expansion}, that inverse mean curvature flow starting from a smooth, strictly mean convex, star-shaped surface remains smooth for all time and approaches a round sphere after renormalization as time goes to infinity. Huisken and Ilmanen \cite{huisken2008higher} later extended this result to $C^1$, weakly mean convex, star-shaped initial surfaces in Euclidean space.

In a general asymptotically flat manifold $M$, the inverse mean curvature flow starting from an initial surface $\Sigma$ may encounter singularities in finite time. To circumvent this, Huisken and Ilmanen \cite{huisken2001inverse} developed a theory of weak inverse mean curvature flow. The weak inverse mean curvature flow starting from any surface $\Sigma$ is guaranteed to exist for all time, but is allowed to instantaneously {\it jump over sets of positive volume}.  The regularity results of Huisken and Ilmanen \cite{huisken2008higher} imply that weak inverse mean curvature flow starting from any initial surface in Euclidean space becomes smooth after a long enough time. In the same spirit, Y. Shi and the second named author \cite{shi2021regularity} proved an analogous regularity result for weak inverse mean curvature flow in asymptotically hyperbolic manifolds. 

Importantly for us, Huisken and Ilmanen \cite{huisken2001inverse} also proved that given any point $p\in M$, there exists a weak inverse mean curvature flow emerging from $p$ at time $-\infty$. Denote this weak inverse mean curvature flow by $\Sigma_t = \bd \Omega_t$ where $t\in \R$.  Y. Shi \cite{shi2016isoperimetric} proved that if $M$ has non-negative scalar curvature and $M$ is not flat at $p$, then the sets $\Omega_t$ have a better-than-Euclidean isoperimetric ratio: 
\[
\frac{\operatorname{Area}(\Sigma_t)}{\vol(\Omega_t)^{2/3}} < (36\pi)^{1/3}.
\]
Using this, it is straightforward to show that 
\[
\sup_{t\in \R} A^c(\Omega_t) < \omega_c(\R^3). 
\]
Hence, if $\Omega_t$ was continuous  as a function of $t$, we could re-parameterize to obtain a family witnessing that $\omega_c(M) < \omega_c(\R^3)$. However, since weak inverse mean curvature flow may instantaneously  jump over sets of positive volume, the family of sets $\Omega_t$ will not vary continuously with $t$ in general. 

To get around this, we prove that if $p$ is very far out in the end of $M$, then the weak inverse mean curvature flow $\Sigma_t = \bd \Omega_t$ emerging from $p$ is actually smooth for $t\in (-\infty, T(p))$ where $T(p)\to \infty$ as $p$ goes to infinity. Actually, we have the following:
\begin{theorem}\label{main-theorem-IMCF}
Fix an asymptotically flat metric on $\R^3$ minus a ball. Given any constant $T$ there is a constant $\rho_1=\rho_1(T)>0$ such that for each point $p$ in $\mathbb R^3\setminus \bar D_{\rho_1}$ the weak solution $u$ for inverse mean curvature flow emerging from $p$ is smooth up to the moment $T$. That is, we can find a smooth map $\Phi:\mathbb S^2\times (-\infty,T]\to \mathbb R^3\setminus\{p\}$ such that
    \begin{itemize}
    \item the map $\Phi_t=\Phi(\cdot,t):\mathbb S^2\to \mathbb R^3\setminus\{p\}$ is a smooth embedding and we have $\Phi_t(\mathbb S^2)=\partial\{u<t\}$;
    \item the smooth map $\Phi$ satisfies
    $$
    \frac{\partial}{\partial t}\Phi(y,t)=H(y,t)^{-1}\nu(y,t)\mbox{ at all }(y,t)\in \mathbb S^2\times (-\infty,T],
    $$
    where $H(y,t)$ is the mean curvature of the surface $\Phi_t(\mathbb S^2)$ at the point $x=\Phi_t(y)$ with respect to unit normal vector $\nu(y,t)$.
    \end{itemize}
    Moreover, we can guarantee
    $$
   D_{c^{-1}e^{t/2}}(p)\setminus\{p\} \subset\Phi(\mathbb S^2\times(-\infty,T])\subset D_{ce^{t/2}}(p)\setminus\{p\}
    $$
    for some constant $c>1$ independent of $t\in (-\infty,T]$. 
 \end{theorem}
Now let us explain how to improve the regularity of the weak inverse mean curvature flow emerging from $p$. From the uniqueness of weak inverse mean curvature flow it suffices to construct smooth inverse mean curvature flow starting from any $\Sigma_t$ with $t\in(-\infty,T]$ for a definite amount of time. Since each surface $\Sigma_t$ is only $C^{1,1}$ and weakly mean convex,  as in \cite{huisken2001inverse} we construct a family of smooth mean convex approximation surfaces $\Sigma_{t,\epsilon}$ of $\Sigma_t$ by mean curvature flow and start a smooth inverse mean curvature flow from $\Sigma_{t,\epsilon}$. Our goal is to establish uniform a priori estimates for these smooth inverse mean curvature flows. 

In the case when the ambient manifold is  Euclidean space, Huisken and Ilmanen \cite{huisken2008higher} showed that, from a star-shaped initial surface, the starshape will be preserved along the flow, and this further guarantees that the mean curvature along the flow can be bounded from below by $c\min\{\sqrt t,1\}$. Using this mean curvature lower bound (and also the natural upper bound from the comparison principle), one can derive interior curvature estimates, and then all higher-order derivative estimates for curvature from Krylov's regularity theory.

In our case, things are more complicated since the ambient space is non-Euclidean. The first difficulty comes from the fact that the support function $\langle x-p,\nu\rangle$ does not have a well-behaved evolution equation along the flow in the presence of error terms, so the preservation of star-shape cannot be derived directly from the comparison principle. The second difficulty also appears in our way to obtain the lower bound for mean curvature, where we have to use the Michael-Simon-Sobolev inequality for the Stampacchia iteration procedure, which already needs curvature estimates a priori due to the non-Euclidean ambient manifold.

To overcome these difficulties, we use the {\it box argument} from the previous joint work \cite{shi2021regularity} of the second named author with Y. Shi. At the first singular time of the smooth inverse mean curvature flows with initial surface $\Sigma_{t,\epsilon}$, it follows from the work \cite{huisken2001inverse} that either the star-shape fails to be kept or the curvature blows up. Therefore, in order to obtain a smooth flow for a definite amount of time, it suffices to find a box region
$$\mathcal R=\{\Sigma:\inf\langle \nu,\partial_r\rangle\geq \iota\mbox{ and }\|\mathcal B\|\leq B\}$$
such that the smooth inverse mean curvature flows with initial surface $\Sigma_{t,\epsilon}$ stay in the box region $\mathcal R$ for a definite amount of time, where $\langle \nu,\partial_r\rangle$ denotes the radial length of the normal vector $\nu$ and $\|\mathcal B\|$ denotes the norm of the second fundamental form.

With the above idea in mind, let us go through the whole argument. According to \cite{huisken2001inverse} there is a weak inverse mean curvature flow $\Sigma_t$ emerging from any point $p$ (see also Lemma \ref{Lem: weak solution}). From area comparison arguments (see Lemma \ref{Lem: outer ball} and Lemma \ref{Lem: inner ball}) we can show
$$\Sigma_t\subset \{e^{t/2}/c\leq |x-p| \leq ce^{t/2}\}$$
for all $t\in (-\infty,T]$. Up to scaling it suffices to construct a smooth inverse mean curvature flow from $\Sigma_0$ for a definite amount of time where $\Sigma_0$ is the $0$-slice of a weak inverse mean curvature flow $\Sigma_t$ emerging from $p$ in an almost Euclidean ambient manifold satisfying $\Sigma_0\subset \{c^{-1}\leq |x|\leq c\}$. As was mentioned above, we construct smooth mean convex approximation surfaces $\Sigma_{0,\epsilon}$ using the mean curvature flow (see Lemma \ref{Lem: MCF}), and we start smooth inverse mean curvature flows $\Sigma_{0,\epsilon,s}$. 

\begin{figure}[htbp]
\centering
\includegraphics[width=10cm]{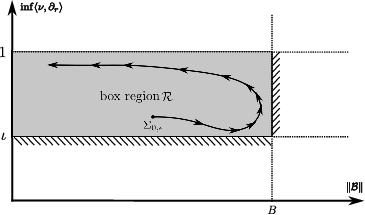}
\caption{The box argument}
\label{Fig: box}
\end{figure}
Given any constant $s_0>0$, we want to guarantee that $\Sigma_{0,\epsilon,s}$ stays in the box region $\mathcal R$ for all $s\in[0,s_0]$ after choosing $\iota$ and $B$ appropriately, as shown in Figure \ref{Fig: box}. 

For the initial surface $\Sigma_{0,\epsilon}$, which is a $C^{1,\beta}$- and $W^{2,p}$-approximation of $\Sigma_0$, we have a uniform area upper bound, a uniform almost non-negative Hawking mass lower bound (from Lemma \ref{Lem: Hawking mass limit} combined with the Geroch monotonicity), a uniform curvature estimate (from \cite{huisken2001inverse} and also \cite{heidusch2001regularitat}), and the annulus neighborhood estimate $\Sigma_{0,\epsilon}\subset \{c^{-1}\leq |x|\leq c\}$. As $p$ goes to infinity, these properties imply that $\Sigma_{0,\epsilon}$ is $C^{1,\beta}$-close to a round sphere and so $\Sigma_{0,\epsilon}$ satisfies the starshape estimate $\langle \nu,\partial_r\rangle\geq \iota(c)$ when $p$ is far out in the end (see Lemma \ref{Lem: star-shape}). We emphasize that only the annulus neighborhood estimate affects the lower bound for $\iota$, while other estimates determine how far out we need to put $p$ in the end.

Now we focus on the flow $\Sigma_{0,\epsilon,s}$ with $s\in [0,s_0]$. Clearly, all surfaces $\Sigma_{0,\epsilon,s}$ satisfy a uniform area upper bound, an almost non-negative Hawking mass lower bound, and the annulus neighborhood estimate. (Notice that we can assume that all surfaces $\Sigma_{0,\epsilon,s}$ satisfy the same estimates as those of $\Sigma_{0,\epsilon}$ a priori except the curvature estimate.)  According to the works \cite{huisken2001inverse} and \cite{heidusch2001regularitat},
if we have the starshape estimate
$$\langle \nu,\partial_r\rangle\geq \iota/4\mbox{ for all }s\in[0,s'],$$
then we have a unifrom curvature estimate $\|\mathcal B\|<B$ for all $s\in [0,s']$, where $B$ is a universal constant depending only on the value of $\iota/4$. Similar to $\Sigma_{0,\epsilon}$, this curvature estimate, combined with the area, Hawking mass, and the annulus neighborhood estimates, yields the improved starshape estimate $\langle \nu,\partial_r\rangle\geq \iota$ after we put $p$ further out in the end.

With above choice of $\iota$ and $B$ we set the box region $\mathcal R$. Notice that we already guarantee the initial surface $\Sigma_{0,\epsilon}$ to lie in the box region $\mathcal R$. We claim that the surface $\Sigma_{0,\epsilon,s}$ lies in the box region $\mathcal R$ up to the time $s_0$. Otherwise, the flow $\Sigma_{0,\epsilon,s}$ will pass through the $B$-edge or the $\iota$-edge of the box region $\mathcal R$ at some time $s<s_0$. If the flow passes through the $B$-edge, we obtain a contradiction to the previous curvature estimate. If the flow passes through the $\iota$-edge, we obtain a contradiction to the improved starshape estimate.

Based on the box argument, we obtain uniform starshape and curvature estimates for smooth inverse mean curvature flows $\Sigma_{0,\epsilon,s}$ for all $s\in [0,s_0]$. Then we can follow the work \cite{huisken2008higher} to obtain positive lower bound for mean curvature, and Krylov's regularity theory gives uniform estimates on all derivatives of curvature. Letting $\epsilon$ go to zero, then we obtain a smooth inverse mean curvature flow with initial surface $\Sigma_0$ for a definite amount of time. This further yields the smoothness of the weak inverse mean curvature flow from the uniqueness.

\subsection{Generalizations on asymptotically hyperbolic $3$-manifolds}
We can also prove an analogy of Theorem \ref{main-theorem-flat} on asymptotically hyperbolic manifolds. Here the natural geometric assumption is that the scalar curvature is at least $-6$. 

\begin{theorem}
\label{main-theorem-hyperbolic}
Let $(M^3,g)$ be a complete, asymptotically hyperbolic manifold, possibly with non-empty boundary. Assume that $M$ has scalar curvature at least $-6$. Then, for every constant $c > 2$, there exists a compact, almost-embedded, free boundary constant mean curvature surface $\Sigma$ in $M$ with mean curvature $c$. 
\end{theorem}

\begin{rem}
Note that in hyperbolic space there do not exist closed CMC surfaces with mean curvature less than or equal to $2$. Hence the restriction to constants $c > 2$ in Theorem \ref{main-theorem-hyperbolic} is necessary. 
\end{rem}

To prove Theorem \ref{main-theorem-hyperbolic} we follow the same line of reasoning as in the asymptotically flat case. Correspondingly, we can establish the following theorems:
\begin{theorem}
\label{asy-hyp-min-max}
Let $(M^3,g)$ be a complete, asymptotically hyperbolic manifold, possibly with boundary. Fix a constant $c > 2$ and assume that $\omega_c(M) < \omega_c(\mathbb H^3)$. Then there is a compact, almost-embedded, free boundary constant mean curvature surface $\Sigma$ in $M$ with mean curvature $c$. 
\end{theorem} 

\begin{theorem}\label{main-theorem-IMCF-AH}
     Given any constant $T$ there is a constant $\rho_1=\rho_1(T)>0$ such that for each point $p$ in $\mathbb R^3\setminus \bar D_{\rho_1}$ the weak solution $u$ for inverse mean curvature flow emerging from $p$ is smooth up to the moment $T$. That is, we can find a smooth map $\Phi:\mathbb S^2\times (-\infty,T]\to (\mathbb R^3\setminus\{p\},\hat g)$ such that
    \begin{itemize}
    \item the map $\Phi_t=\Phi(\cdot,t):\mathbb S^2\to \mathbb R^3\setminus\{p\}$ is a smooth embedding and we have $\Phi_t(\mathbb S^2)=\partial\{u<t\}$;
    \item the smooth map $\Phi$ satisfies
    $$
    \frac{\partial}{\partial t}\Phi(y,t)=H(y,t)^{-1}\nu(y,t)\mbox{ at all }(y,t)\in \mathbb S^2\times (-\infty,T],
    $$
    where $H(y,t)$ is the mean curvature of the surface $\Phi_t(\mathbb S^2)$ at the point $x=\Phi_t(y)$ with respect to unit normal vector $\nu(y,t)$.
    \end{itemize}
   Moreover, given any constant $r_0>0$ we can choose $T$ large enough such that $\Phi_t(T)$ encloses $D^h_{r_0}(p)$.
   \end{theorem}

\subsection{Organization of the rest of this paper} In Section 2, we apply the box method to prove Theorem \ref{main-theorem-IMCF}. Using this we construct a continuously varying family of open sets $\{U_t\}_{t\in[0,1]}$ in any non-Euclidean asymptotically flat manifolds with nonnegative scalar curvature such that the isoperimetric ratios of $U_t$ are less than $(36\pi)^{1/3}$, which is the main result of Proposition \ref{Prop: sweep-out}. In Section 3, we establish counterparts of above results in the setting of asymptotically hyperbolic manifolds, which are formulated by Theorem \ref{asy-hyp-min-max} and Proposition \ref{Prop: sweep-out AH}.
In Section 4, we generalize the min-max method in \cite{mazurowski2022prescribed} to asymptotically flat and hyperbolic manifolds with boundary, establishing Theorem \ref{theorem-min-max} and Theorem \ref{asy-hyp-min-max}. Combined with the use of previously constructed family $\{U_t\}_{t\in[0,1]}$, we finally prove Theorem \ref{main-theorem-flat} and Theorem \ref{main-theorem-hyperbolic}.

\subsection{Acknowledgments} The first-named author would like to thank Xin Zhou for introducing him to the problem and for many insightful discussions. He acknowledges the support of an AMS-Simons travel grant. The second-named author would like to thank Dr. Yuchen Bi for many helpful conversations. He was partially supported by National Key R\&D Program of China 2023YFA1009900 and NSFC grant 12401072 as well as the start-up fund from Westlake University.

\section{Sweep-out with sub-Euclidean isoperimetric ratio and smooth inverse mean curvature flows in an asymptotically flat end}\label{Sec: sweep-out AF}

In this section, a complete Riemannian $3$-manifold $(M^3,g)$ (possibly with compact boundary) is called an asymptotically flat manifold, if there is a compact subset $K$ of $M$ such that 
\begin{itemize}
    \item[(c1)] $M\setminus K$ is diffeomorphic to $\mathbb R^3\setminus \bar D_1$, where $$D_1=\{x\in \mathbb R^3: |x|<1\};$$
    \item[(c2)] the metric $g$ satisfies
$$\sum_{k=0}^3|x|^k|\partial^k\left(g_{ij}-\delta_{ij}\right)|=o(1)\mbox{ as }x\to\infty,$$
where $g_{ij}$ are components in the metric expression $$g=g_{ij}\mathrm dx_i\otimes \mathrm dx_j$$ with respect to the Euclidean coordinate chart $(x_i)$.
\end{itemize}
For convenience, we call $M\setminus K$ an asymptotically flat end.

Our goal of this section is to show the following

\begin{proposition}\label{Prop: sweep-out}
    Assume that $(M^3,g)$ is an asymptotically flat manifold with non-negative scalar curvature, possibly with compact boundary, which does not have Euclidean end. Then given any constant $v>1$, there exists a continuously varying family of open sets, denoted by $\{U_t \}_{t\in[0,1]}$, satisfying $\vol(U_0) = \emptyset$, $\vol(U_1 ) = v$, and also
$$
\area(\partial U_t)<(36\pi)^{\frac{1}{3}} \vol(U_t )^{\frac{2}{3}}$$
for all $t \in (0,1]$.
\end{proposition}

The proof is based on the construction of smooth inverse mean curvature flows in asymptotically flat ends, where we will adopt the existence of weak solutions from \cite{huisken2001inverse} and improve the regularity as in \cite{shi2021regularity}.

%%%%%%%%%%%%%%%%%

\subsection{Existence of weak solutions for inverse mean curvature flow from a point}
For our purpose, we always extend the asymptotically flat end $(M\setminus K,g)$ to a new complete Riemannian manifold $(\mathbb R^3,\hat g)$ such that 
$$\hat g=g\mbox{ outside }\bar D_1.$$
In the following, we always use $B_r(p)$ to denote the $\hat g$-geodesic $r$-ball centered at the point $p$, and we use $D_r(p)$ to denote the Euclidean $r$-ball centered at the point $p$. If $p$ is the origin, we omit the symbol $p$.

Let $\Omega$ be an open subset of $\mathbb R^3$. A locally Lipschitz function $u$ is called a weak solution for inverse mean curvature flow in $\Omega$, if for any locally Lipschitz function $v$ in $\Omega$ with $\{v\neq u\}$ precompact in $\Omega$, we always have $J_u^K(u)\leq J_u^K(v)$, where $K$ denotes the compact closure of $\{v\neq u\}$ in $\Omega$, and we use $J_u^K(\cdot)$ to denote the functional
$$
J_u^K(\phi)=\int_{K}|\nabla \phi|+\phi|\nabla u|\,\mathrm d\mu_M.
$$

\begin{lemma}\label{Lem: weak solution}
    For any point $p\in \mathbb R^3$ there is a locally Lipschitz function $u$ in $\mathbb R^3\setminus\{p\}$ such that
    \begin{itemize}
        \item $u(x)\to-\infty$ as $x\to p$ and $u(x)\to+\infty$ as $x\to \infty$;
        \item $u$ is a weak solution for inverse mean curvature flow.
    \end{itemize}
\end{lemma}
\begin{proof}
This follows from \cite[Lemma 8.1]{huisken2001inverse}, where the argument is somewhat brief, and so we include the details for completeness. 
The proof will be divided into three steps. \vspace{2mm}

    {\it Step 1. For each $\epsilon>0$ there is a locally Lipschitz function $u_\epsilon$ such that $\{u_\epsilon<0\}=B_{\epsilon}(p)$ and $u_\epsilon$ is a weak solution for inverse mean curvature flow in $\mathbb R^3\setminus \bar B_{\epsilon}(p)$ with $u_\epsilon\to+\infty$ as $|x|\to+\infty$.} \vspace{2mm}
    
    From \cite[Theorem 3.1]{huisken2001inverse} it suffices to find a locally Lipschitz function $w$ such that for some large positive constant $r_0$, we have the following:
    \begin{itemize}
    \item the function $w$ is smooth;
    \item $w(x)\to +\infty$ as $|x|\to+\infty$;
    \item $w$ is a weak subsolution for inverse mean curvature flow in $\mathbb R^3\setminus \bar D_{r_0}$. That is, we have $J^K_w(w)\leq J^K_w(v)$ for any locally Lipschitz function $v\leq w$ with $\{v\neq w\}$ precompact in $\mathbb R^3\setminus\bar D_{r_0}$.
    \end{itemize}
    
Take 
$$w=\max\{\log|x|,0\}.$$ 
It is easy to verify
$$
 \Div\left(\frac{\nabla w}{|\nabla w|}\right)\geq |\nabla w|
$$
and $|\nabla w|\neq 0$ outside ball $D_{r_0}$ with $r_0$ large enough. 
The first two properties are obvious, and we just need to verify the third one. Take a locally Lipschitz function $v\leq w$ as desired. Then we can compute
    \[
    \begin{split}
    J_{w}^K(v)-J_{w}^K(w)&=\int_K|\nabla v|-|\nabla w|+(v-w)|\nabla w|\,\mathrm d\mu\\
    &\geq \int_K|\nabla v|-|\nabla w|+(v-w)\Div\left(\frac{\nabla w}{|\nabla w|}\right)\,\mathrm d\mu\\
    &=\int_K|\nabla v|-\left\langle\nabla v,\frac{\nabla  w}{|\nabla w|}\right\rangle\,\mathrm d\mu\geq 0.
    \end{split}
    \]

    {\it Step 2. There are $\epsilon_i\to 0$ and $c_i\to +\infty$ as $i\to\infty$ such that the functions $u_{\epsilon_i}-c_i$ converge locally uniformly to a locally Lipschitz function $u$ in $\mathbb R^3\setminus\{p\}$.}\vspace{2mm}
    
From the gradient estimate \cite[Theorem 3.1]{huisken2001inverse} there exist a constant $\sigma_0>0$ and a dimensional constant $C$ such that we have
    \begin{equation}\label{Eq: gradient estimate}
    |\nabla u_\epsilon|(x)\leq \frac{C}{\dist(x,p)}
    \end{equation}
    at all points $x$ satisfying $2\epsilon\leq \dist(x,p)\leq 2\sigma_0$. Here and in the sequel, we always use $C$ to denote a dimensional constant, which may be different from line to line. For each $\epsilon>0$ we take 
    $$
   c_\epsilon= \min_{\partial B_{\sigma_0}(p)} u_\epsilon
    $$
 and define 
 \begin{equation}\label{Eq: mathring u epsilon}
 \mathring u_\epsilon=u_\epsilon-c_\epsilon.\end{equation} From the gradient estimate above and diameter bound of $\partial B_{\sigma_0}$, we have 
  \begin{equation}\label{Eq: mathring u bound}
      0\leq \mathring u_\epsilon\leq C\mbox{ on }\partial B_{\sigma_0}(p).
  \end{equation}
On each compact subset of $\mathbb R^3\setminus\{p\}$, the functions $\mathring u_\epsilon$ are uniformly bounded and equicontinuous  due to the the gradient estimate \cite[Theorem 3.1]{huisken2001inverse}. Passing to a subsequence, we can find $\epsilon_i\to 0$ as $i\to\infty$ such that the functions $\mathring u_{\epsilon_i}$ converge locally uniformly to a locally Lipschitz function $u$ in $\mathbb R^3\setminus\{p\}$. 

Denote $c_i=c_{\epsilon_i}$ for short. Now we show $c_i\to+\infty $ as $i\to\infty$. Denote
$$
N_{\epsilon,t}=\partial\{u_\epsilon<t\}\mbox{ and }N^+_{\epsilon,t}=\partial\{u_\epsilon\leq t\}.
$$
The exponential area growth for weak inverse mean curvature flow from \cite[Lemma 1.6 (i)]{huisken2001inverse} yields
    $$\area(N_{\epsilon,t})=A_\epsilon e^t\mbox{ for all }t>0,$$
where $A_\epsilon$ denotes the boundary area of the strictly minimizing hull of $B_\epsilon(p)$, and so we have 
$$A_\epsilon\leq \area(\partial B_\epsilon(p))=O(\epsilon^{2})\mbox{ as }\epsilon\to 0.$$
    On the other hand, we know from the proof of \cite[Lemma 4.2]{huisken2001inverse}  that  for any $t>0$ the sets $\{u_\epsilon>t\}$ and $\{u_\epsilon<t\}$ cannot have non-empty precompact component in $\mathbb R^3\setminus B_\epsilon(p)$. As a consequence, for any $r>\epsilon$ we have
\begin{equation}\label{Eq: interior upper bound}
u_\epsilon\leq \max_{\partial B_r(p)} u_\epsilon \mbox{ in }B_r(p)
\end{equation}   
and 
\begin{equation}\label{Eq: exterior lower bound}
u_\epsilon\geq \min_{\partial B_r(p)} u_\epsilon\mbox{ outside }B_r(p).
\end{equation}

Combining \eqref{Eq: interior upper bound} and \eqref{Eq: mathring u bound}, we obtain $u_\epsilon\leq c_\epsilon+C$ in $B_{\sigma_0}(p)$. In other words, we have the inclusion $B_{\sigma_0}(p) \subset \{u_\epsilon\leq c_\epsilon+C\}$. Since the metric $\hat g$ is equivalent to the Euclidean metric, namely we have $\hat C^{-1}g_{euc}\leq \hat g\leq \hat Cg_{euc}$ for some constant $\hat C>1$, it follows from the Euclidean isoperimetric inequality that we have 
\begin{equation}\label{Eq: area lower bound}
\area(N_{\epsilon,c_\epsilon+C})\geq \delta_0
\end{equation}
for some constant $\delta_0=\delta_0(\hat C,\sigma_0)>0$. As a consequence, we obtain
$$
c_\epsilon\geq \log\frac{\delta_0}{A_\epsilon}-C,
$$
which implies $c_\epsilon\to+\infty$ as $\epsilon\to 0$. \vspace{2mm}

    {\it Step 3. The function $u$ satisfies all the desired properties.} \vspace{2mm}
    
    First we prove $u(x)\to+\infty$ as $|x|\to +\infty$, which follows from a comparison argument. Let $w$ be the subsolution outside $D_{r_0}$, and $r_0$ the constant from Step 1. From \eqref{Eq: mathring u bound} and the gradient estimate \cite[Theorem 3.1]{huisken2001inverse} we know that $\mathring u_\epsilon$ is uniformly bounded in each compact subset of $\mathbb R^3\setminus\{p\}$. In particular, we must have $\mathring u_\epsilon> \Lambda$ on $\partial D_{r_0}$ for some constant $\Lambda$ independent of $\epsilon$. 
    
    Denote 
    $$w_{r_0}=\log\frac{|x|}{r_0}+\Lambda.$$
  Then we claim $\mathring u_\epsilon\geq w_{r_0}$ outside $\bar D_{r_0}$.  
 Suppose by contradiction that there is a point $x_0\in \mathbb R^3\setminus \bar D_{r_0}$ with $\mathring u_{\epsilon}(x_0)<w_{r_0}(x_0)$. Fix a constant $t_0>w_{r_0}(x_0)$. Then the function 
 $$w_{r_0,t_0}:=\min\{w_{r_0},t_0\}$$ is still a weak subsolution for inverse mean curvature flow in $\mathbb R^3\setminus \bar D_{r_0}$, and the set $\{w_{r_0,t_0}>\mathring u_{\epsilon}\}$ is precompact in $\mathbb R^3\setminus \bar D_{r_0}$. By \cite[Theorem 2.2]{huisken2001inverse} (where the comparison principal even holds for a supersolution $u$ and a subsolution $v$) we obtain $\mathring u_\epsilon\geq w_{r_0,t_0}$ in $\mathbb R^3\setminus \bar D_{r_0}$, which leads to a contradiction. 

 Passing to the limit, we obtain $u\geq w_{r_0}$ in $\mathbb R^3\setminus  \bar D_{r_0}$, and in particular we have $u(x)\to+\infty$ as $|x|\to+\infty$.

    Next we prove $u(x)\to-\infty$ as $x\to p$. Denote
    \begin{equation}\label{Eq: mathring N epsilon}
        \mathring N_{\epsilon,t}=\partial\{\mathring u_\epsilon<t\}\mbox{ and }\mathring N_{\epsilon,t}=\partial\{\mathring u_\epsilon\leq t\}.
    \end{equation}
    For our purpose, we take $T$ to be a large positive constant and consider the surface 
    $\mathring N^+_{\epsilon,-T}$.
     It follows from the exponential area growth and \eqref{Eq: area lower bound} that 
    $$\area(\mathring N^+_{\epsilon,-T})=\area(\mathring N_{\epsilon,C})\cdot e^{-T-C}\geq \delta_0\cdot e^{-T-C}.$$
    At the very beginning, we may pick  $\sigma_0$ small enough a priori such that $\area(\partial B_{\rho}(p))\leq 5 \pi\rho^2$ for each $\rho\in (0,\sigma_0)$. Since the surface  $\mathring N^+_{\epsilon,-T}$ is strictly outer-minimizing, it cannot be contained in ball $B_{\rho_0}(p)$ with
    $$
    \rho_0=\sqrt\frac{\delta_0\cdot e^{-T-C}}{5\pi}.
    $$
    In particular, we obtain 
    $$\min_{\partial B_{\rho_0}(p)} \mathring u_\epsilon\leq -T,$$
    and so we can derive the following estimate
   $$
    \max_{\partial B_{\rho_0}(p)} \mathring u_\epsilon\leq C-T.
  $$
    From \eqref{Eq: interior upper bound} we see
     \begin{equation}\label{Eq: interior upper bound from area}
    \mathring u_{\epsilon}\leq C-T \mbox{ in } B_{\rho_0}(p)\setminus B_\epsilon(p).
      \end{equation}
       Passing to the limit and letting $T\to+\infty$, we see $u(x)\to-\infty$ as $x\to p$.
     
    The second property is a direct consequence of \cite[Theorem 2.1]{huisken2001inverse}, and we omit the details.
\end{proof}

In the following, the function $u$ from Lemma \ref{Lem: weak solution} will be called the weak solution for inverse mean curvature flow emerging from $p$. For convenience, we shall always assume the weak solution for inverse mean curvature flow emerging from $p$ to satisfy the normalization condition
\begin{equation}\label{Eq: normalization}
\min_{\partial D_1(p)} u=0.
\end{equation}
Moreover, we denote
$$\Sigma_t=\partial\{u<t\}\mbox{ and }\Sigma_t^+=\partial\{u\leq t\}.$$

\begin{lemma}\label{Lem: Hawking mass limit}
Denote
\begin{equation}\label{Eq: Hawking mass}
m_h(\Sigma_t)=\frac{\area(\Sigma_t)^{\frac{1}{2}}}{(16\pi)^{\frac{3}{2}}}\left(16\pi-\int_{\Sigma_t}H^2\right).
\end{equation}
Then the limit
\begin{equation*}
\lim_{t\to-\infty} m_h(\Sigma_t)
\end{equation*}
exists and is non-negative.
\end{lemma}
\begin{proof}
Since $\{u<0\}$ is contained in $D_1(p)\setminus \{p\}$, for each $t<0$ the surface $\Sigma_t$ must be connected, and so its Euler number satisfies $\chi(\Sigma_t)\leq 2$. Then it follows from Geroch monotonicity \cite[Formula 5.8]{huisken2001inverse} that for any $s<t<0$ we have
$$
m_h(\Sigma_t)\geq m_h(\Sigma_s)+\frac{1}{(16\pi)^{\frac{3}{2}}}\int_s^t\left[\area(\Sigma_\tau)^{\frac{1}{2}}\int_{\Sigma_\tau}R(\hat g)\,\mathrm d\sigma_\tau\right]\,\mathrm d\tau.
$$
Notice that the scalar curvature $R(\hat g)$ is uniformly bounded, i.e. $|R(\hat g)|\leq \Lambda$ in $\mathbb R^3$ for some constant $\Lambda>0$. From the exponential area growth we see $\area(\Sigma_\tau)=A_0e^\tau$ for some constant $A_0>0$. Therefore, we have
$$
m_h(\Sigma_t)\geq m_h(\Sigma_s)-\frac{2\Lambda A_0^\frac{3}{2}}{3\cdot(16\pi)^{\frac{3}{2}}} e^{\frac{3}{2}t}.
$$
First taking limit superior of RHS as $s\to -\infty$, and then taking limit inferior on both sides as $t\to-\infty$, we conclude that the Hawking mass limit exists.

Next we show the nonnegativity of this limit. From the simple fact that every Riemannian manifold is infinitesimally Euclidean, we have
$$
\lim_{\epsilon\to 0}m_h(\partial B_\epsilon(p))=0.
$$
Denote $\mathring u_\epsilon$ to be the function in \eqref{Eq: mathring u epsilon} and use the notation in \eqref{Eq: mathring N epsilon}. Clearly, it follows from \eqref{Eq: exterior lower bound} that we have $\{\mathring u_\epsilon< 0\} \subset B_{\sigma_0}(p)$. Then the strictly outer-minimizing property of $\mathring N^+_{\epsilon,-t}$ yields 
$$\area(\mathring N_{\epsilon,-t})=\area(\mathring N^+_{\epsilon,-t})\leq \area(\partial B_{\sigma_0}(p))\mbox{ for all } t>0.$$ From the Geroch monotonicity again we can derive
\[
\begin{split}
m_h(\mathring N_{\epsilon,-t})&\geq -(16\pi)^{-\frac{3}{2}}\int_0^{c_{\epsilon}-t}\left[\area(N_{\epsilon,\tau})^{\frac{1}{2}}\int_{N_{\epsilon,\tau}}R(\hat g)\,\mathrm d\sigma_\tau\right]\,\mathrm d\tau\\
&\geq -\frac{2\Lambda}{3}(16\pi)^{-\frac{3}{2}}\area(\mathring N_{\epsilon,-t/2})^{\frac{3}{2}}e^{-\frac{3t}{4}}\\
&\geq O(e^{-\frac{3t}{4}})\mbox{ as }t\to +\infty.
\end{split}
\]
From \eqref{Eq: interior upper bound from area} we know that, for each $t>0$, the distances of surfaces $\mathring N_{\epsilon,-t}$ to the point $p$ are uniformly bounded from below by a positive constant independent of $\epsilon$. Then
it follows from the gradient estimate \eqref{Eq: gradient estimate} as well as \cite[Theorem 1.3]{huisken2001inverse} that the surfaces $\mathring N_{\epsilon,-t}$ satisfy uniform $C^{1,\beta}$-estimates and $W^{2,p}$-estimates independent of $\epsilon$. As a result, for almost every $t>0$, the surfaces $\mathring N_{\epsilon,-t}$ converge to the surface $N_{-t}$ in $C^{1,\beta'}$-sense and $W^{2,p'}$-sense up to a subsequence. From the upper semi-continuity of the Hawking mass (see \cite{huisken2001inverse}) we obtain $m_h(N_{-t})\geq O(e^{-\frac{3t}{4}})$ as $t\to+\infty$. The proof is completed by letting $t\to+\infty$.
\end{proof}

\subsection{Regularity theory}

The goal of this subsection is to show Theorem \ref{main-theorem-IMCF}. The proof will be presented at the end of this subsection. Before that, we have to establish a series of preliminary estimates.

\subsubsection{Annulus neighborhood estimate} In the following, we fix a constant $\hat C=\hat C(\mathbb R^3,\hat g)>1$ such that
$$
\hat C^{-1}g_{euc}\leq \hat g\leq \hat Cg_{euc}\mbox{ in }\mathbb R^3.
$$

\begin{lemma}\label{Lem: outer ball}
There is a constant $\bar\alpha=\bar\alpha(\hat C)>0$ such that when $\rho_1$ is chosen large enough, we have $\{u\leq t\}\subset D_{\bar\alpha\hat Ce^ {t/2}}(p)$ for all $t\in(-\infty, T]$, where 
$$D_{\bar\alpha\hat Ce^ {t/2}}(p)=\{x\in \mathbb R^3:|x-p|<\bar\alpha\hat C e^{t/2}\}.$$
\end{lemma}
\begin{proof}
Since $(\mathbb R^3,\hat g)$ is asymptotically flat, according to \cite[Theorem 3.1]{huisken2001inverse}, no matter which $\bar \alpha>0$ is chosen, we can take $\rho_1$ large enough such that the gradient estimate $|\nabla u|(x)\leq C\dist(x,p)^{-1}$ holds in $ D_{\bar\alpha\hat Ce^ {T/2}}(p)$ for a dimensional constant $C$. As a consequence, there is a constant $C_0$, depending only on $\hat C$, such that we have
$$\osc_{\partial D_\rho(p)}u\leq C_0 \mbox{ for all } \rho\in(0,\bar\alpha\hat Ce^{T/2}].$$ 

In the following, we show that if $\bar\alpha$ is chosen to be large enough, then we have $\{u\leq t\}\subset D_{\bar\alpha\hat Ce^{t/2}}(p)$ for all $t\in (-\infty,T]$. Due to the normalization condition \eqref{Eq: normalization}, the estimate \eqref{Eq: exterior lower bound}, the exponential area growth of $\Sigma_t^+$, and the strictly outer-minimizing property of $\Sigma_t^+$, we conclude that
$$\area(\Sigma_{0}^+)=\lim_{t\to 0^-}\area(\Sigma_{t}^+)\leq \area(\partial D_1(p))\leq 4\pi \hat C.$$
Moreover, the exponential area growth yields 
$$\area (\Sigma_{t+C_0}^+)\leq 4\pi \hat C e^{t+C_0}\mbox{ for all }t\in (-\infty,T].$$
 Suppose by contradiction that the set $\Sigma_{t}^+\setminus D_{\bar\alpha\hat Ce^ {t/2}}(p)$ is non-empty for some $t\in(-\infty,T]$. Then there exists a point $q\in \partial D_{\bar\alpha\hat Ce^ {t/2}}(p)$ with $u(q)\leq t$, and so we have 
$$\max_{\partial D_{\bar\alpha\hat Ce^ {t/2}}(p)}u\leq t+C_0.$$ It follows from \eqref{Eq: interior upper bound} that we have 
$$D_{\bar\alpha\hat Ce^ {t/2}}(p)\subset \{u\leq t+C_0\}.$$ 
Then the Euclidean isoperimetric inequality yields 
$$\area(\Sigma_{t+C_0}^+)\geq \hat C^{-1}4\pi \left(\bar\alpha\hat Ce^{t/2}\right)^2 ,$$ 
which cannot happen when $\bar\alpha>e^{C_0/2}$. 

To sum up, if we take $\bar\alpha=e^{C_0/2}+1$ and also take $\rho_1$ large enough, then we can guarantee $\{u\leq t\}\subset D_{\bar\alpha\hat Ce^{t/2}}(p)$ for all $t\in(-\infty,T]$.
\end{proof}

\begin{lemma}\label{Lem: inner ball}
There is a constant $\underline\alpha=\underline\alpha(\hat C)>0$ such that when $\rho_1$ is chosen large enough we have $\{u\leq t\}\supset D_{\underline\alpha\hat C^{-1}e^ {t/2}}(p)\setminus\{p\}$ for all $t\in(-\infty,T]$.
\end{lemma}
\begin{proof}
As before, no matter which $\underline\alpha>0$ is chosen, we can always guarantee $\osc_{\partial D_\rho(p)}u\leq C_0$ for all $\rho\in(0,\underline\alpha\hat C^{-1}e^{T/2}]$ by taking $\rho_1$ large enough. From \eqref{Eq: interior upper bound} and the Euclidean isoperimetric inequality, we can derive $\area(\Sigma_{C_0}^+)\geq 4\pi\hat C^{-1}$, and so we have
 $$\area(\Sigma^+_{t-C_0}) \geq 4\pi\hat C^{-1}e^{t-2C_0}\mbox{ for all }t\in(-\infty,T].$$ 
Denote $r=\hat C^{-1}e^{t/2-C_0}$. From the fact $\area(\partial D_r(p))\leq 4\pi\hat C^{-1}e^{t-2C_0}$
and also the strictly outer-minimizing property of $\Sigma^+_{t-C_0}$, we know that there is a point $q\in \partial D_r(p)$ such that $u(q)\leq t-C_0$, and so  we have $\max_{\partial D_r(p)}u\leq t$. This implies 
$D_{r}(p)\setminus\{p\}\subset \{u\leq t\}$ whenever $t\in(-\infty,T]$. The proof is now completed by taking $\underline \alpha=e^{-C_0}$.
\end{proof}
\begin{corollary}\label{Cor: annulus}
    There is a universal constant $c$ such that we have
    \begin{equation}\label{Eq: annulus control}
D_{c^{-1}e^{t/2}}(p)\setminus\{p\} \subset\{u\leq t\}\subset D_{ce^{t/2}}(p)\setminus\{p\}
\end{equation}
when $\rho_1$ is taken large enough.
\end{corollary}
\begin{proof}
Take
$
c=\max\{\bar\alpha \hat C,\underline\alpha^{-1}\hat C\}
$ and use Lemma \ref{Lem: outer ball} and Lemma \ref{Lem: inner ball}. 
\end{proof}

\subsubsection{Star-shapeness estimate for a class of almost-round surfaces} In the following, we define the eccentricity with respect to $p$ for a closed surface $\Sigma$ enclosing the point $p$ by
$$
\theta(\Sigma,p)=\frac{\bar r(\Sigma,p)}{\underline r(\Sigma,p)},
$$
where
\begin{equation}\label{Eq: outer radius}
\bar r(\Sigma,p)=\inf\{r>0:\Sigma\mbox{ is contained in } D_r(p)\}
\end{equation}
and
\begin{equation}\label{Eq: inner radius}
\underline r(\Sigma,p)=\sup\{r>0:\Sigma\mbox{ encloses }D_r(p)\}.
\end{equation}

\begin{lemma}\label{Lem: star-shape}
Given constants $\theta_0>1$, $A_0>0$ and $B_0>0$ we can find constants $m_0<0$ and $\rho_1>0$ such that for each point $p\in \mathbb R^3\setminus \bar D_{\rho_1}$ if $\Sigma$ is a connected closed surface satisfying the following properties:
\begin{itemize}
\item $\Sigma$ encloses the point $p$ and the eccentricity of $\Sigma$ with respect to $p$ satisfies $\theta(\Sigma,p)\leq \theta_0$;
\item $\Sigma$ satisfies $\area(\Sigma)\leq A_0$;
\item the second fundamental form of $\Sigma$ satisfies 
$$ \|\mathcal B_\Sigma\|_{\mathcal L^\infty}\leq B_0\cdot \area(\Sigma)^{-\frac{1}{2}};$$
\item the Hawking mass of $\Sigma$ satisfies $m_h(\Sigma)\geq m_0\cdot \area(\Sigma)^\frac{1}{2}$,
\end{itemize}
then we have
\begin{equation}\label{Eq: star}
\hat g(x-p,\nu(x))\geq \left[1+\left(\theta_0-1\right)^2\right]^{-\frac{1}{2}}\cdot |x-p|_{\hat g}
\end{equation}
for all $x\in \Sigma$.
\end{lemma}
\begin{proof}
We argue by contradiction. Suppose that, for any pair of constants $(m_0,\rho_1)$, there exists an exceptional surface $\Sigma_{m_0,\rho_1}$ enclosing some point $p\in \mathbb R^3\setminus \bar D_{\rho_1}$, which satisfies all properties but \eqref{Eq: star}. Namely we have
$$
\hat g(q-p,\nu(q))< \ \left[1+\left(\theta_0-1\right)^2\right]^{-\frac{1}{2}}\cdot |q-p|_{\hat g}
$$
at some point $q\in \Sigma$. For later use, we take a sequence $(m^i_0,\rho_1^i)$ such that $m^i_0\to 0^-$ and $\rho^i_1\to+\infty$ as $i\to\infty$, and we denote the corresponding surface $\Sigma_{m_0^i,\rho_1^i}$ by $\Sigma_i$, which encloses some point $p_i\in \mathbb R^3\setminus\bar D_{\rho^i_1}$ and satisfies
$$
\hat g(q_i-p_i,\nu(q_i))<  \left[1+\left(\theta_0-1\right)^2\right]^{-\frac{1}{2}}\cdot |q_i-p_i|_{\hat g}\mbox{ at some }q_i\in \Sigma_i.$$ 

For convenience, let us define the area radius $\sigma_i$ by 
$$\sigma_i=\sqrt{\frac{\area(\Sigma_i)}{4\pi}}.$$
We shall consider the convergence of the pointed manifolds $(\mathbb R^3,p_i,\sigma_i^{-2}\hat g,\Sigma_i)$ as $i\to\infty$. In the following, we always use the normalized coordinate chart
$$
\phi_i:\mathbb R^3\to \mathbb R^3,\,x\mapsto\sigma_ix+p_i.
$$
It is easy to check that the metric
$\tilde g_i:=\phi_i^*(\sigma_i^{-2}\hat g)$ can be written as
$$\tilde g_i=\tilde g_{i,kl}\mathrm dx_k\otimes \mathrm dx_l,\mbox{ where }\tilde g_{i,kl}(x)=\hat g_{kl}(\sigma_i x+p_i).$$
Recall that we have $\area(\Sigma_i)\leq A_0$. The area radii $\sigma_i$ are uniformly bounded from above, and so the metrics $g_i$ converge to the Euclidean metric locally in $C^2$-sense. Denote $\tilde \Sigma_i=\sigma_i^{-1}(\Sigma_i-p_i)$. Clearly, we have $\area_{\tilde g_i}(\tilde \Sigma_i)=4\pi$ and also
$$\|\mathcal B_{\tilde \Sigma_i}\|_{\mathcal L^\infty}=\sigma_i \|\mathcal B_{\Sigma_i}\|_{\mathcal L^\infty}\leq \frac{B_0}{\sqrt{4\pi}}.$$
In particular, surfaces $\tilde\Sigma_i$ converge to a limit surface $\tilde \Sigma_\infty$ in $C^{1,\beta}$-sense up to a subsequence. Denote $q$ to be the limit of the points $q_i$  in $\tilde\Sigma_\infty$ up to a subsequence. Then the $C^{1,\beta}$-convergence of $\tilde \Sigma_i$ yields
\begin{equation}\label{Eq: non star-shape}
g_{euc}(q,\nu(q))\leq  \left[1+\left(\theta_0-1\right)^2\right]^{-\frac{1}{2}} |q|.
\end{equation}

On the other hand, we know from the Hawking mass condition that
$$
\int_{\tilde \Sigma_i}\tilde H_i^2\,\mathrm d\tilde\sigma_i\leq 16\pi-(16\pi)^{\frac{3}{2}} m_{0,i}=16\pi+o(1)\mbox{ as }i\to 0.
$$
Denote $\bar H_i$ and $\mathrm d\bar\sigma_i$ to be the mean curvature and the area element of $\tilde \Sigma_i$ with respect to the Euclidean metric. From the facts 
$$|\bar H_i-\tilde H_i|=(1+\|\mathcal B_{\tilde \Sigma_i}\|)\cdot o(1)\mbox{ and }\mathrm d\bar\sigma_i=(1+o(1))\mathrm d\tilde \sigma_i,$$
we know
$$
\int_{\tilde\Sigma_i}\bar H_i^2\,\mathrm d\bar\sigma_i\leq 16\pi+o(1)\mbox{ as }i\to\infty.
$$
Combined with the resolution of the Willmore conjecture (see \cite[Theorem A]{Marques2014Willmore}), we see that each surface $\tilde\Sigma_i$ is a topological sphere for large $i$. Using the Gauss equation, we obtain
$$
\int_{\tilde \Sigma_i}\| \mathring{ \bar {\mathcal B}}_{\tilde \Sigma_i}\|^2\,\mathrm d\bar \sigma_i=\frac{1}{2}\left(\int_{\tilde \Sigma_i}\bar H_i^2\,\mathrm d\bar\sigma_i-4\pi\chi(\tilde\Sigma_i)\right)=o(1)\mbox{ as }i\to\infty.
$$
It follows from \cite[Theorem 1.1]{de2005optimal} that surfaces $\tilde \Sigma_i$ converge to a unit sphere in $C^0$-sense up to a subsequence. In particular, we conclude that the limit surface $\tilde \Sigma_\infty$ is a unit sphere in $\mathbb R^3$. Since the $C^0$-convergence guarantees the convergence of eccentricity, we obtain
$\theta(\tilde \Sigma_\infty)\leq \theta_0$. Then we can derive from the elementary plane geometry that
$$
g_{euc}(x,\nu(x))\geq \left[1+\left(\frac{\theta_0-1}{2}\right)^2\right]^{-\frac{1}{2}} |x|\mbox{ for all }x\in \tilde \Sigma_{\infty}.
$$
This leads to a contradiction to \eqref{Eq: non star-shape}.
\end{proof}

\subsubsection{Smoothing surfaces with mean curvature flow} In the following, we use $N$ to denote the diffeomorphism type of $\Sigma_t^+$.
\begin{lemma}\label{Lem: MCF}
The constant $\rho_1$ can be taken large enough such that for each $N_t^+$ with $t\in(-\infty,T]$ there is a mean curvature flow $\Psi:N\times (0,\epsilon_0)\to \mathbb R^3$ such that $\Sigma_{t,\epsilon}:=\Psi(N,\epsilon)$ converge to $\Sigma_t^+$ as $\epsilon\to0^+$ in $C^{1,\beta}$-sense and $W^{2,p}$-sense for all $\beta\in (0,1)$ and all $p\geq 1$. Moreover, $\Sigma_{t,\epsilon}$ is mean-convex and the second fundamental form $\mathcal B_\epsilon$ of $\Sigma_{t,\epsilon}$ satisfies
\begin{equation}\label{Eq: curvature bound approximation}
\|\mathcal B_\epsilon\|_{\mathcal L^\infty}\leq \|\mathcal B_{0}\|_{\mathcal L^\infty}\cdot e^{-C\epsilon-\frac{c_0^2}{\beta}\epsilon^{\beta}}
\end{equation}
for some constants $C$ and $c_0$ independent of $\epsilon$, where $\mathcal B_{0}$ denotes the second fundamental form of $\Sigma_t^+$. In particular, $\mathcal B_\epsilon$ are uniformly bounded.
\end{lemma}
\begin{proof}
The proof here follows the argument in \cite[Lemma 2.6]{Huisken2008HigherRegularity} with slight modifications due to the non-Euclidean ambient space. 

Recall that $\Sigma_t^+$ is $C^1$ with bounded mean curvature. Then the standard regularity results of Allard and Calderon-Zygmund yield that $\Sigma_t^+$ is of $C^{1,\beta}$ and $W^{2,p}$ for all $\beta\in (0,1)$ and $p\geq 1$. Through mollification we can find a sequence of smooth surfaces $\Sigma_i$ converging to $\Sigma_t^+$ as $i\to\infty$ in $C^{1,\beta}$-sense and $W^{2,p}$-sense for any $\beta\in (0,1)$ and $p\geq 1$. Due to the pseudolocality theorem \cite[Theorem 7.3]{Chen2007Pseudolocality}, the mean curvature flow $\Psi_i:N\times [0,\epsilon_0)\to (\mathbb R^3,\hat g)$ with initial data $\Sigma_i$ exists on a uniform time interval, and for each $\alpha>0$ there is a constant $0<\epsilon_\alpha<\epsilon_0$ such that the surface $\Psi_i(N,\epsilon)$ satisfies
$\|\mathcal B_\epsilon^i\|^2\leq\alpha\epsilon^{-1}$ when $\epsilon<\epsilon_\alpha$. In particular, this implies \begin{equation}\label{Eq: Hausdorff distance}
    d_{\mathcal H}(\Psi_i(N,\epsilon),\Sigma_i)\leq 2\sqrt\alpha \epsilon^{\frac{1}{2}}.
\end{equation}
 Using the argument from \cite[Section 9]{Ilmanen2019PlanarNetFlow}, we conclude that the flows $\Psi_i$ can be written as graphs with uniformly bounded gradient in Gaussian adapted coordinates. Then it follows from the interior parabolic Schauder regularity theory that we have the improved curvature estimates
\begin{equation}\label{Eq: curvature estimate MCF}
\|\nabla^k \mathcal B^i_{\epsilon}\|\leq c_k\epsilon^{\frac{\beta-k-1}{2}},\,k\in \mathbb N,
\end{equation}
where each $c_k$ is a universal constant independent of $\epsilon$ and $i$. As a result, the flows $\Psi_i$ converge to a mean curvature flow $\Psi:N\times (0,\epsilon_0)\to\mathbb R^3$, where the slice $\Sigma_{t,\epsilon}:=\Psi(N,\epsilon)$ satisfies the curvature estimate \eqref{Eq: curvature estimate MCF} as well.

Next we establish uniform estimates for $\Sigma_{t,\epsilon}$ independent of $\epsilon$. Using the evolution equation (vi) from Lemma \ref{Lem: evolution equation MCF} for approximation mean curvature flows $\Psi_i$, we have
$$
\frac{\partial}{\partial \epsilon}\|\mathcal B^i_\epsilon\|^2\leq \Delta\|\mathcal B^i_\epsilon\|^2-2\|\nabla \mathcal B^i_\epsilon\|^2+2\|\mathcal B^i_\epsilon\|^4+C(\|\mathcal B^i_\epsilon\|+\|\mathcal B^i_\epsilon\|^2),
$$
where the last two terms come from the contractions of ambient curvatures and $\mathcal B^i_\epsilon$. Here and in the sequel, we always use $C$ to denote a universal constant independent of $i$ and $\epsilon$. From integration by parts, for any $p> 2$ we have
$$
\frac{\partial}{\partial t}\int_{\Sigma_{i,\epsilon}}\|\mathcal B^i_\epsilon\|^p\,\mathrm d\sigma^i_\epsilon\leq p\int_{\Sigma_{i,\epsilon}}\|\mathcal B_\epsilon^i\|^{p+2}\,\mathrm d\sigma^i_\epsilon+Cp\int_{\Sigma_{i,\epsilon}}\left(\|\mathcal B_\epsilon^i\|^{p-1}+\|\mathcal B_\epsilon^i\|^{p}\right)\,\mathrm d\sigma^i_\epsilon,
$$
where we denote $\Sigma_{i,\epsilon}:=\Psi_i(N,\epsilon)$ for short. Using the previous improved curvature estimate $\|\mathcal B_\epsilon^i\|\leq c_0^2\epsilon^{\beta-1}$ and applying the H\"older inequality to the last second term, we can derive
\begin{equation}\label{Eq: Lp curvature estimate}
\|\mathcal B^i_\epsilon\|_{\mathcal L^p}\leq \left(\frac{1}{p}(Cp)^{\frac{1}{p}}\epsilon+\|\mathcal B^i_0\|_{\mathcal L^p}\right)\cdot e^{C\epsilon+\frac{c_0^2}{\beta}\epsilon^{\beta}}.
\end{equation}
Since surfaces $\Sigma_i$ converge to $\Sigma_t^+$ in $W^{2,p}$-sense, by letting $i\to\infty$ we see that $\Sigma_{t,\epsilon}$ has uniform $W^{2,p}$-bounds for all $p>2$. Denote the $W^{2,p}$-limit of some sequence $\Sigma_{t,\epsilon_l}$ with $\epsilon_l\to 0$ by $\Sigma$.
After passing \eqref{Eq: Hausdorff distance} to the limit as $i\to\infty$ and using the $C^{1,\beta}$-convergence of $\Sigma_i$, we obtain $\Sigma=\Sigma_t^+$ from the uniqueness of Hausdorff limit. This yields that $\Sigma_{t,\epsilon}$ converge to $\Sigma_t^+$ in $W^{2,p}$-sense and also in $C^{1,\beta}$-sense.
To see \eqref{Eq: curvature bound approximation}, we let $i\to\infty$ and $p\to+\infty$ in $\mathcal L^p$-curvature estimate \eqref{Eq: Lp curvature estimate}.

It remains to show the mean-convexity of $\Sigma_{t,\epsilon}$. For this purpose, we use the evolution equation of the mean curvature (see evolution equation (v) from Lemma \ref{Lem: evolution equation MCF}), i.e.
$$
\frac{\partial}{\partial\epsilon} H^i_\epsilon=\Delta H^i_\epsilon+(\Ric(\nu)+\|\mathcal B^i_\epsilon\|^2) H^i_\epsilon.
$$
Denote $H^i_{\epsilon,-}=\min\{H^i_\epsilon,0\}$. From integration by parts, we have
$$
\frac{\partial}{\partial \epsilon}\int_{\Sigma_{i,\epsilon}}|H^i_{\epsilon,-}|^2\,\mathrm d\sigma^i_\epsilon\leq 2\int_{\Sigma_{i,\epsilon}}(C+\|\mathcal B^i_\epsilon\|^2)|H^i_{\epsilon,-}|^2\,\mathrm d\sigma^i_\epsilon.
$$
Similar to the previous discussion, we can derive
$$
\|H^i_{\epsilon,-}\|_{\mathcal L^2}\leq \|H^i_{0,-}\|_{\mathcal L^2}\cdot e^{C\epsilon+\frac{c_0^2}{\beta}\epsilon^{\beta}}.
$$
Since the mean curvature $H$ of $N_t^+$ satisfies $H\geq 0$, after taking $i\to\infty$ we conclude that the mean curvature $H_\epsilon$ of $\Sigma_{t,\epsilon}$ satisfies $H_\epsilon\geq 0$. Since the constant $\rho_1$ can be taken large enough to guarantee that the region $D_{ce^{T/2}}(p)\setminus\{p\}$ is foliated by mean-convex concentric spheres, it follows that $\Sigma_{t,\epsilon}$ cannot be minimal for any $\epsilon$, and thus each surface $\Sigma_{t,\epsilon}$ has to be mean-convex.
\end{proof}

\subsubsection{Smooth inverse mean curvature flow from $\Sigma_t^+$}
As a preparation, we collect estimates for surface $N_t^+$ which will be used later. It follows from the annulus neighborhood estimate \eqref{Eq: annulus control} that $\{u\leq t\}$ is contained in $D_{2ce^{T/2}}(p)\setminus\{p\}$ for each $t<T$. Therefore, $\Sigma_t^+$ must be a connected surface enclosing point $p$ (recall that $\{u\leq t\}$ cannot have precompact components). Furthermore, the annulus neighborhood estimate \eqref{Eq: annulus control} yields
\begin{equation}\label{Eq: eccentricity}
\Sigma_t^+\subset \bar D_{ce^{t/2}}(p)\setminus  D_{c^{-1}e^{t/2}}(p).
\end{equation}
The outer-minimizing property of $\Sigma_t^+$ then yields 
\begin{equation} \label{Eq: area upper bound}
\area(\Sigma_t^+)\leq 4\pi c^2 e^t(1+o(1)).
\end{equation}
From the Geroch monotonicity and Lemma \ref{Lem: Hawking mass limit} we see 
$$
m_h(\Sigma_t)\geq \area(\Sigma_t)^{\frac{3}{2}}\cdot o(1).
$$
Using the facts $\area(\Sigma_t^+)=\area(\Sigma_t)$ and $m_h(\Sigma_t^+)\geq m_h(\Sigma_t)$, we can derive
\begin{equation}\label{Eq: Hawking mass lower bound}
m_h(\Sigma_t^+)\geq \area(\Sigma_t^+)^{\frac{3}{2}} \cdot o(1).
\end{equation}
Up to scaling, it follows from \cite[Theorem 3.1 and Theorem 1.3]{huisken2001inverse}, \cite[Theorem 5.5]{heidusch2001regularitat}, and the annulus neighborhood estimate \eqref{Eq: annulus control} that there exists a constant $\Lambda$ depending only on the dimension and $c$ such that
\begin{equation}\label{Eq: curvature bound}
\|\mathcal B_{\Sigma_t^+}\|_{\mathcal L^\infty}\leq \Lambda\cdot \area(\Sigma_t^+)^{-\frac{1}{2}}
\end{equation}
when $\rho_1$ is taken large enough.

\begin{lemma}\label{Lem: smooth IMCF}
For any constant $s_0>0$ the constant $\rho_1$ can be taken large enough such that for each $\Sigma_t^+$ with $t\in(-\infty,T]$ there is a smooth inverse mean curvature flow $\Phi^+: N\times (0,s_0 ]\to \mathbb R^3$ such that
$$d_{\mathcal H}(\Phi^+(N,s),\Sigma_t^+)\to 0\mbox{ as }s\to 0^+,$$
and that $\Phi^+(N,s)$ encloses $D_{c^{-1}e^{\frac{t}{2}+\frac{s}{3}}}(p)$.
\end{lemma}

The proof will be given shortly after we establish several necessary a priori estimates.
Denote $\Psi:N\times (0,\epsilon_0)\to \mathbb R^3$ to be the mean curvature flow with initial data $\Sigma_t^+$ constructed in Lemma \ref{Lem: MCF} and use the notation $\Sigma_{t,\epsilon}=\Psi(N,\epsilon)$.
\begin{lemma}\label{Lem: initial estimate}
The surfaces $\Sigma_{t,\epsilon}$ can satisfy all the estimates \eqref{Eq: eccentricity}-\eqref{Eq: curvature bound} uniformly for small $\epsilon$ after slightly adjusting the values of constants. 
\end{lemma} 
\begin{proof}
    This follows from the $C^{1,\alpha}$- and $W^{2,p}$-convergence as well as Lemma \ref{Lem: MCF}.
\end{proof}

Since $\Sigma_{t,\epsilon}$ is smooth and mean-convex, we can guarantee the existence of a smooth inverse mean curvature flow $\Phi_\epsilon:N\times [0,\tau)\to\mathbb R^3$ with initial data $\Sigma_{t,\epsilon}$, where $\tau$ denotes the maximal existence time. For covenience, we denote $\Sigma_{t,\epsilon,s}:=\Phi_\epsilon(N,s)$.
We are going to show that the maximal existence time of the flow $\Phi_\epsilon$ is always beyond $s_0$ when $\rho_1$ is taken large enough.

Since the smooth inverse mean curvature flow is scaling-invariant, we do scaling $\tilde g=e^{-t}\hat g$ and work in the coordinate chart 
$$\phi:\mathbb R^3\to \mathbb R^3,\,x\mapsto e^{t/2}x+p. $$
Our goal is to establish uniform a priori estimates for the inverse mean curvature flows $\Phi_\epsilon$ under the assumption $\tau\leq s_0$ with respect to the scaled metric $\tilde g$.
\begin{lemma}\label{Lem: area bound scaled}
The $\tilde g$-area of $\Sigma_{t,\epsilon,s}$ is uniformly bounded, i.e. $\area_{\tilde g}(\Sigma_{t,\epsilon,s})\leq \tilde A_0$ for a universal constant $\tilde A_0$ independent of $p$, $t$, $\epsilon$ and $s$.
\end{lemma}
\begin{proof}
From Lemma \ref{Lem: initial estimate} and \eqref{Eq: area upper bound} as well as the exponential area growth of $\Sigma_{t,\epsilon,s}$, we have
$$\area_{\tilde g}(\Sigma_{t,\epsilon,s})=4\pi c^2 (1+o(1))e^s \mbox{ as }\rho_1\to +\infty.$$
The uniform $\tilde g$-area bound comes from the assumption $s\leq s_0$. 
\end{proof}
\begin{lemma}\label{Lem: eccentricity scaled}
    The eccentricity of $\Sigma_{t,\epsilon,s}$ with respect to the origin $o$ satisfies $\tilde \theta(\Sigma_{t,\epsilon,s},o)\leq \tilde\theta_0$ for some universal constant $\tilde \theta_0$ independent of $p$, $t$ ,$\epsilon$ and $s$.
\end{lemma}
\begin{proof}
    In the new coordinate chart, the scaled metric $\tilde g$ has the expression $\tilde g=\tilde g_{kl}\mathrm dx_k\otimes \mathrm dx_l$, where $\tilde g_{kl}(x)=\hat g_{kl}(e^{t/2}x+p)$. From the fact $t\leq T$ and the asymptotically-flat property of $g$, we see 
\begin{equation}\label{Eq: scaled metric}
|\tilde g_{kl}|+|\partial\tilde g_{kl}|+|\partial^2\tilde g_{kl}|+|\partial^3\tilde g_{kl}|=o(1)\mbox{ as }\rho_1\to+\infty
\end{equation}
in any fixed compact subset. In particular, $\rho_1$ can be taken large enough such that functions $w_{sub}=\log |x|$ and $w_{sup}=3\log |x|$ are a weak subsolution and supersolution for inverse mean curvature flow in $D_{2e^{s_0}}$, respectively. Then it follows from Lemma \ref{Lem: initial estimate}, the annulus neighborhood estimate \eqref{Eq: eccentricity}, and a comparison argument that the surface $\Sigma_{t,\epsilon,s}$ satisfies
\begin{equation}\label{Eq: annulus estimate scaled}
\Sigma_{t,\epsilon,s}\subset \bar D_{ce^{s}}\setminus D_{c^{-1}e^{s/3}}.
\end{equation}
This yields $\tilde\theta(\Sigma_{t,\epsilon,s},o)\leq \tilde\theta_0:= c^2 e^{2s_0/3}$, and we complete the proof.
\end{proof}
\begin{lemma}\label{Lem: Hawking mass scaled}
    The Hawking mass of $\Sigma_{t,\epsilon,s}$ satisfies
    $$\tilde m_h(\Sigma_{t,\epsilon,s})\geq o(1)\cdot \area_{\tilde g}(\Sigma_{t,\epsilon,s})^{\frac{1}{2}}\mbox{ as }\rho_1\to+\infty.$$
\end{lemma}
\begin{proof}
    From Lemma \ref{Lem: initial estimate} and estimates \eqref{Eq: area upper bound}-\eqref{Eq: Hawking mass lower bound},  we see
    $$\area_{\tilde g}(\Sigma_{t,\epsilon,0})^{-\frac{1}{2}}\tilde m_h(\Sigma_{t,\epsilon,0})\geq o(1).$$
    From the Geroch monotonicity we have
    $$\tilde m_h(\Sigma_{t,\epsilon,s})\geq \tilde m_h(\Sigma_{t,\epsilon,0})+\area_{\tilde g}(\Sigma_{t,\epsilon,s})^{\frac{3}{2}}\cdot o(1).$$
    Therefore, we obtain
    $$\area_{\tilde g}(\Sigma_{t,\epsilon,s})^{-\frac{1}{2}}\tilde m_h(\Sigma_{t,\epsilon,s})\geq e^{-s/2}\cdot o(1)+\area_{\tilde g}(\Sigma_{t,\epsilon,s})\cdot o(1).$$
    Now the desired estimate follows from the uniform $\tilde g$-area bound of $\Sigma_{t,\epsilon,s}$ from Lemma \ref{Lem: area bound scaled}.
\end{proof}

Denote
$$
\tilde X=\frac{x}{|x|_{\tilde g}}.
$$
In the following, we collect those estimates depending on the a priori star-shapeness assumption.
\begin{lemma}\label{Lem: curvature estimate scaled}
Given a positive constant $\iota_0$, we can take $\rho_1$ large enough and a universal constant $\tilde B_0$ independent of $p$, $t$, $\epsilon$ and $s$ such that if $\Sigma_{t,\epsilon,s}$ satisfies 
\begin{equation}\label{Eq: starshape assumption}
    \tilde g( \tilde X,\tilde \nu(x))\geq \iota_0>0 
\end{equation}
for all $s\in [0,\tau^*]$ with $\tau^*<\tau$, then the second fundamental form $\tilde{\mathcal B}_s$ of $\Sigma_{t,\epsilon,s}$ satisfies
$$
\|\tilde{\mathcal B}_s\|\leq \tilde B_0\cdot \area_{\tilde g}(\Sigma_{t,\epsilon,s})^{-\frac{1}{2}}
$$
for all $s\in [0,\tau^*]$.
\end{lemma}
\begin{proof}

From Lemma \ref{Lem: initial estimate} and the initial curvature estimate \eqref{Eq: curvature bound}, we know that the second fundamental form $\tilde {\mathcal B}_0$ of $\Sigma_{t,\epsilon,0}$ satisfies 
$\|\tilde {\mathcal B}_0\|\leq \Lambda \cdot \area_{\tilde g}(\Sigma_{t,\epsilon,0})$. Using estimates \eqref{Eq: scaled metric} and \eqref{Eq: annulus estimate scaled},
we can conclude from \cite[Theorem 3.1]{huisken2001inverse} and \cite[Theorem 5.1]{heidusch2001regularitat} that there is a positive constant $c_1$, independent of $p$, $t$, $\epsilon$ and $s$, such that after taking $\rho_1$ large enough we have
\begin{equation}\label{Eq: curvature bound scaled}
\|\tilde{\mathcal B}_s\|\leq c_1(1+\|\tilde {\mathcal B}_0\|).
\end{equation}
From the uniform $\tilde g$-area bound of $\Sigma_{t,\epsilon,s}$ in Lemma \ref{Lem: area bound scaled} and the exponential area growth, we have
$$
\area_{\tilde g}(\Sigma_{t,\epsilon,s})^{\frac{1}{2}}\cdot \|\tilde {\mathcal B}_s\|\leq c_1\tilde A_0^{1/2}+c_1\Lambda\cdot e^{s/2}.
$$
The proof is completed from the assumption $s\leq s_0$.
\end{proof}
\begin{lemma}\label{Lem: Sobolev}
Under the starshape assumption \eqref{Eq: starshape assumption}, the constant $\rho_1$ can be taken large enough such that we have the following uniform Sobolev inequality $$
\left(\int_{\Sigma_{t,\epsilon,s}}\zeta^{2}\,\mathrm d\tilde \sigma\right)^{\frac{1}{2}}\leq c_{S}\int_{\Sigma_{t,\epsilon,s}} |\tilde \nabla \zeta|+|\tilde H||\zeta|\,\mathrm d\tilde \sigma
$$
on $\Sigma_{t,\epsilon,s}$, where $c_S$ is a constant independent of $p$, $t$, $\epsilon$ and $s$.
\end{lemma}
\begin{proof}
From \eqref{Eq: curvature bound scaled} we know that $\tilde{\mathcal B}_s$ are uniformly bounded. We use $\bar H$ to denote the mean curvature of $\Sigma_{t,\epsilon,s}$ with respect to the Euclidean metric. Then we have $|\tilde H-\bar H|=|\tilde{\mathcal B}_s|\cdot o(1)$ as $\rho_1\to +\infty$. It is well-known that we have the following Michael-Simon-Sobolev inequality
\begin{equation}\label{Eq: MS}
\left(\int_{\Sigma_{t,\epsilon,s}}\zeta^{2}\,\mathrm d\bar \sigma\right)^{\frac{1}{2}}\leq \bar c_{S}\int_{\Sigma_{t,\epsilon,s}} |\bar \nabla \zeta|+|\bar H||\zeta|\,\mathrm d\bar \sigma,
\end{equation}
where every quantity with a bar is computed with respect to the Euclidean metric, and $\bar c_S$ is a dimensional constant. It is clear that RHS cannot exceed
\begin{equation}\label{Eq: error term}
\bar c_S(1+o(1))\int_{\Sigma_{t,\epsilon,s}}|\tilde\nabla \zeta|+|\tilde H||\zeta|\,\mathrm d\tilde \sigma+o(1)\int_{\Sigma_{t,\epsilon,s}} |\zeta|\,\mathrm d\tilde\sigma,
\end{equation}
where the last term can be handled by
$$
o(1)\int_{\Sigma_{t,\epsilon,s}} |\zeta|\,\mathrm d\tilde\sigma\leq o(1)\cdot \area_{\tilde g}(\Sigma_{t,\epsilon,s})^{\frac{1}{2}}\left(\int_{\Sigma_{t,\epsilon,s}}\zeta^2\,\mathrm d\bar \sigma\right)^{\frac{1}{2}}.
$$
It follows from Lemma \ref{Lem: area bound scaled} that, after taking $\rho_1$ large enough, the last term in \eqref{Eq: error term} can be finally absorbed in the LHS of \eqref{Eq: MS}, and we are done.
\end{proof}
\begin{corollary}\label{Cor: L2 Sobolev}
For any $q>1$ we have the following Sobolev inequality
$$
\left(\int_{\Sigma_{t,\epsilon,s}}|\zeta|^{2q}\,\mathrm d\tilde \sigma\right)^{\frac{1}{q}}\leq c_{S}^2q^2\int_{\Sigma_{t,\epsilon,s}} |\tilde \nabla \zeta|^2+\tilde H^2\zeta^2\,\mathrm d\tilde \sigma
$$
where $c_{S}$ is a positive constant independent of $p$, $t$, $\epsilon$ and $s$ as well as $q$.
\end{corollary}
\begin{proof}
From Lemma \ref{Lem: Sobolev} we have
\[
\begin{split}
\left(\int_{\Sigma_{t,\epsilon,s}}|\zeta|^{2q}\,\mathrm d\tilde \sigma\right)^{\frac{1}{2}}&\leq c_S \int_{\Sigma_{t,\epsilon,s}} q|\zeta|^{q-1}|\tilde \nabla \zeta|+|\tilde H||\zeta|^q\,\mathrm d\tilde \sigma\\
&\leq 2c_Sq\left(\int_{\Sigma_{t,\epsilon,s}} |\tilde \nabla \zeta|^2+\tilde H^2\zeta^2\,\mathrm d\tilde \sigma
\right)^{\frac{1}{2}}\left(\int_{\Sigma_{t,\epsilon,s}}|\zeta|^{2q-2}\mathrm d\tilde \sigma\right)^{\frac{1}{2}}.
\end{split}
\]
The desired inequality comes from the H\"older inequality
$$
\int_{\Sigma_{t,\epsilon,s}}|\zeta|^{2q-2}\mathrm d\tilde \sigma\leq \left(\int_{\Sigma_{t,\epsilon,s}}|\zeta|^{2q}\mathrm d\tilde\sigma\right)^{\frac{q-1}{q}}\area_{\tilde g}(\Sigma_{t,\epsilon,s})^{\frac{1}{q}}
$$
as well as Lemma \ref{Lem: area bound scaled}.
\end{proof}
\begin{corollary}\label{Cor: Poincare}
For any $q>1$ we have the following Poincar\'e inequality
$$
\int_{\Sigma_{t,\epsilon,s}}\zeta^{2}\,\mathrm d\tilde \sigma\leq c_P\int_{\Sigma_{t,\epsilon,s}}|\tilde \nabla \zeta|^2+\tilde H^2\zeta^2\,\mathrm d\tilde \sigma,
$$
where $c_P$ is a positive constant independent of $p$, $t$, $\epsilon$ and $s$.
\end{corollary}
\begin{proof}
It follows from the H\"older inequality that
$$
\int_{\Sigma_{t,\epsilon,s}} |\tilde \nabla \zeta|+|\tilde H||\zeta|\,\mathrm d\tilde \sigma\leq \sqrt 2\left(\int_{\Sigma_{t,\epsilon,s}}|\tilde \nabla \zeta|^2+\tilde H^2\zeta^2\,\mathrm d\tilde \sigma\right)^{\frac{1}{2}}\area_{\tilde g}(\Sigma_{t,\epsilon,s})^{\frac{1}{2}}.
$$
The desired Poincar\'e inequality comes from Lemma \ref{Lem: Sobolev} as well as Lemma \ref{Lem: area bound scaled}.
\end{proof}

\begin{lemma}\label{Lem: star-shape to curvature}
 Given any positive constant $\iota_0$, we can take $\rho_1$ large enough and a universal constant $\tilde c$, independent of $p$, $t$, $\epsilon$ and $s$, such that under the starshape assumption \eqref{Eq: starshape assumption} we have for all $s\in [0,\tau^*]$ that
\begin{equation}\label{Eq: mean curvature lower bound}
\tilde H_s\geq \tilde c\min\{1,s^{\frac{1}{2}}\}.
\end{equation}
\end{lemma}
\begin{proof}
We follow the argument from \cite{Huisken2008HigherRegularity} with slight modifications due to the non-Euclidean ambient manifold. Take $
v=\tilde g(x,\tilde\nu(x))
$ to be the support function and we compute the evolution equation of $v$. Clearly we have
$$\frac{\partial}{\partial s}v=\langle\tilde\nabla_{\partial_s}x,\tilde\nu\rangle+\langle x,\tilde\nabla_{\partial_s}\tilde\nu\rangle=\tilde H^{-1}(1+o(|x|))-\langle x,\nabla \tilde H^{-1}\rangle.$$
Take $\{e_i\}$ to be an orthonormal basis on $\Sigma_{t,\epsilon,s}$. Then we can compute
$$\Delta v=\langle \tilde \nabla_{e_i}\tilde \nabla_{e_i}x,\tilde \nu\rangle+2\langle\tilde\nabla_{e_i}x,\tilde \nabla_{e_i}\tilde\nu\rangle+\langle x,\tilde \nabla_{e_i}\tilde \nabla_{e_i}\tilde\nu\rangle.$$
From \eqref{Eq: scaled metric} we have
$$\langle \tilde \nabla_{e_i}\tilde \nabla_{e_i}x,\tilde\nu\rangle=-\tilde H+o(|x|)+o(|x|\|\tilde {\mathcal B}\|)+o(1),$$
$$\langle\tilde\nabla_{e_i}x,\tilde \nabla_{e_i}\tilde\nu\rangle=\tilde H+o(|x|\|\tilde{\mathcal B}\|),$$
and
$$\langle x,\tilde \nabla_{e_i}\tilde \nabla_{e_i}\tilde\nu\rangle=\widetilde{\Ric}(x^t,\tilde \nu)+\langle x,\nabla\tilde H\rangle-\|\tilde{\mathcal B}\|^2v.$$
Therefore, we have
$$
\left(\frac{\partial}{\partial s}-\tilde H^{-2}\Delta\right) v=\frac{\|\tilde {\mathcal B}\|^2}{\tilde H^2} v+\tilde H^{-2}\left(o(1)+o(|x|)+o(|x|\|\tilde{\mathcal B}\|)\right).
$$
From the assumption \eqref{Eq: starshape assumption} we know that $v$ is equivalent to $|x|$. Moreover, we have $|x|\geq c^{-1}$ from the annulus neighborhood estimate \eqref{Eq: eccentricity}, and we have $\widetilde{\Ric}(x^t,\nu)=o(1)+o(|x|)$. Therefore, we can write
\begin{equation}\label{Eq: evolution support function}
\left(\frac{\partial}{\partial s}-\tilde H^{-2}\Delta\right) v=\frac{\|\tilde {\mathcal B}\|^2}{\tilde H^2} v+\frac{1+\|\tilde{\mathcal B}\|}{\tilde H^2}v\cdot o(1),
\end{equation}
On the other hand, according to the evolution equation (iv) in Lemma \ref{Lem: evolution equation IMCF} we have
$$
\left(\frac{\partial}{\partial s}-\tilde H^{-2}\Delta\right)\tilde H^{-1}=\frac{\|\tilde{\mathcal B}\|^2}{\tilde H^2}\tilde H^{-1}+\frac{\widetilde\Ric(\tilde\nu)}{\tilde H^2}\tilde H^{-1}.
$$
Define $$w=(\tilde Hv)^{-1}.$$ Since $\tilde{\mathcal B}$ is uniformly bounded and $\widetilde\Ric(\tilde\nu)=o(1)$ as $\rho_1\to+\infty$, we obtain through a direct computation that
$$
\frac{\partial}{\partial s}w=\Div\left(\tilde H^{-2}\nabla w\right)-2\tilde H^{-2}w^{-1}|\nabla w|^2+\tilde H^{-2}w\cdot o(1).
$$

The rest arguments will be divided into two steps.

\vspace{2mm}{\it Step 1. The function $w$ satisfies the following $\mathcal L^q$-estimate for any $q>2$
$$
\|w\|_{\mathcal L^q(\Sigma_{t,\epsilon,s})}\leq \sqrt{\frac{5}{2\tilde c}}e^{\frac{5s}{2q}}\left(e^{\frac{5s}{q}}-1\right)^{-\frac{1}{2}},
$$
where $\tilde c$ is a universal constant independent of $p$, $t$, $\epsilon$ and $s$.
}

\vspace{2mm}For any $q>2$ we have
\begin{equation}\label{Eq: integral evolution}
\begin{split}
\frac{\partial}{\partial s}\int_{\Sigma_{t,\epsilon,s}}w^q\,\mathrm d\tilde \sigma\leq&-q(q+1)\int_{\Sigma_{t,\epsilon,s}}\tilde H^{-2}w^{q-2}|\tilde \nabla w|^2\,\mathrm d\tilde \sigma\\
&\qquad+\int_{\Sigma_{t,\epsilon,s}}w^q\,\mathrm d\tilde \sigma+o(1)\int_{\Sigma_{t,\epsilon,s}}v^2 w^{q+2}\,\mathrm d\tilde \sigma.
\end{split}
\end{equation}
From the assumption \eqref{Eq: starshape assumption} we conclude $ \iota_0 |x|_{\tilde g}\leq v\leq |x|_{\tilde g}$. Using the fact $\tilde g=(1+o(1))g_{euc}$ as well as \eqref{Eq: annulus estimate scaled}, we can take $\rho_1$ large enough such that
$$
\frac{\iota_0}{2c}\leq v\leq 2ce^{s_0},
$$
where $c$ is the constant from the annulus neighborhood estimate \eqref{Eq: eccentricity}. As a result, we have
$$
\int_{\Sigma_{t,\epsilon,s}}v^2 w^{q+2}\,\mathrm d\tilde \sigma\leq 4c^2e^{2s_0}\int_{\Sigma_{t,\epsilon,s}} w^{q+2}\,\mathrm d\tilde \sigma
$$
and
\begin{equation*}
\begin{split}
\int_{\Sigma_{t,\epsilon,s}}\tilde H^{-2}w^{q-2}|\tilde \nabla w|^2\,\mathrm d\tilde \sigma&=\int_{\Sigma_{t,\epsilon,s}}v^2w^{q}|\tilde \nabla w|^2\,\mathrm d\tilde \sigma\\
&\geq \frac{\iota_0^2}{c^2(q+2)^2}\int_{\Sigma_{t,\epsilon,s}}|\tilde \nabla w^{\frac{q+2}{2}}|^2\,\mathrm d\tilde \sigma.
\end{split}
\end{equation*}
The Poincar\'e inequality from Corollary \ref{Cor: Poincare} yields 
\begin{equation*}
\begin{split}
\int_{\Sigma_{t,\epsilon,s}}|\tilde \nabla w^{\frac{q+2}{2}}|^2\,\mathrm d\tilde \sigma &\geq c_P^{-1} \int_{\Sigma_{t,\epsilon,s}}w^{q+2}\,\mathrm d\tilde \sigma- \int_{\Sigma_{t,\epsilon,s}}\tilde H^2 w^{q+2}\,\mathrm d\tilde \sigma\\
&=c_P^{-1} \int_{\Sigma_{t,\epsilon,s}}w^{q+2}\,\mathrm d\tilde \sigma- \int_{\Sigma_{t,\epsilon,s}}v^{-2} w^{q}\,\mathrm d\tilde \sigma\\
&\geq c_P^{-1} \int_{\Sigma_{t,\epsilon,s}}w^{q+2}\,\mathrm d\tilde \sigma-\frac{4c^2}{\iota_0^2}\int_{\Sigma_{t,\epsilon,s}} w^q\,\mathrm d\tilde \sigma.
\end{split}
\end{equation*}
Using the H\"older inequality we also have
$$
\int_{\Sigma_{t,\epsilon,s}}|\tilde \nabla w^{\frac{q+2}{2}}|^2\,\mathrm d\tilde \sigma\geq c_P^{-1}\left(\int_{\Sigma_{t,\epsilon,s}}w^{q}\,\mathrm d\tilde \sigma\right)^{\frac{q+2}{q}}\area_{\tilde g}(\Sigma_{t,\epsilon,s})^{-\frac{2}{q}}-\frac{4c^2}{\iota_0^2}\int_{\Sigma_{t,\epsilon,s}} w^q\,\mathrm d\tilde \sigma.
$$
Notice that we have
$$
\frac{q(q+1)}{(q+2)^2}>\frac{3}{8}\mbox{ when }q>2.
$$
From these calculations, Lemma \ref{Lem: area bound scaled}, and  \eqref{Eq: integral evolution}, we conclude that when $\rho_1$ is taken large enough (depending on $c_P$ but not on $q$), there holds
$$
\frac{\partial}{\partial s}\int_{\Sigma_{t,\epsilon,s}}w^q\,\mathrm d\tilde \sigma\leq -\tilde c\left(\int_{\Sigma_{t,\epsilon,s}}w^{q}\,\mathrm d\tilde \sigma\right)^{\frac{q+2}{q}}+\frac{5}{2}\int_{\Sigma_{t,\epsilon,s}} w^q\,\mathrm d\tilde \sigma,
$$
 where $\tilde c$ is a universal constant independent of $p$, $t$, $\epsilon$ and $s$. From this we can derive
$$
\|w\|_{\mathcal L^q(\Sigma_{t,\epsilon,s})}\leq \sqrt{\frac{5}{2\tilde c}}e^{\frac{5s}{2q}}\left(e^{\frac{5s}{q}}-1\right)^{-\frac{1}{2}}.
$$

\vspace{2mm}{\it Step 2. We have the desired estimate \eqref{Eq: mean curvature lower bound}.}

\vspace{2mm}Fix any two constants $s_1, s_2$ such that $0<s_1<s_2\leq \tau^*$ and $s_2-s_1\leq 1$. We let $s$ vary in $[s_1,s_2]$. Define
$$
\hat w=(s-s_1)^{1/4}w.
$$
Then we can compute
\[
\begin{split}
\frac{\partial}{\partial s}\hat w&=\Div\left(\tilde H^{-2}\nabla \hat w\right)-2\tilde H^{-2}\hat w^{-1}|\nabla \hat w|^2\\
&\qquad\qquad\qquad+\tilde H^{-2}\hat w\cdot o(1)+\frac{\hat w}{4(s-s_1)}.
\end{split}
\]
Denote $$\hat w_k=\max\{\hat w-k,0\}$$ and $$\Omega_k=\Sigma_{t,\epsilon,s}\cap \{\hat w\geq k\}\mbox{ for }k>0.$$ 
Multiplying both sides by $\hat w_k$ and using integration by parts, we have
\begin{equation}\label{Eq: integral evolution parabolic}
\begin{split}
\frac{\partial}{\partial s}\int_{\Sigma_{t,\epsilon,s}}\hat w_k^2\,\mathrm d\tilde\sigma \leq& -2\int_{\Omega_k}\tilde H^{-2}|\tilde\nabla \hat w|^2\,\mathrm d\tilde \sigma\\
&+\int_{\Sigma_{t,\epsilon,s}}\tilde H^{-2}\hat w_k\hat w\cdot o(1)\,\mathrm d\tilde \sigma\\
&\quad+\frac{1}{2(s-s_1)}\int_{\Omega_k}\hat w^2\,\mathrm d\tilde \sigma+\int_{\Sigma_{t,\epsilon,s}}\hat w_k^2\,\mathrm d\tilde \sigma.
\end{split}
\end{equation}
Using the Poincar\'e inequality from Corollary \ref{Cor: Poincare}, we have
\[
\begin{split}
\int_{\Omega_k}\tilde H^{-2}|\tilde\nabla \hat w|^2\,\mathrm d\tilde \sigma&\geq  \frac{\iota_0^2}{16c^2}(s-s_1)^{-\frac{1}{2}}\int_{\Sigma_{t,\epsilon,s}}|\tilde \nabla(\hat w_k+k)^2|\,\mathrm d\tilde\sigma\\
&\geq \frac{\iota_0^2}{16c^2c_P}(s-s_1)^{-\frac{1}{2}}\int_{\Omega_k}\hat w^4\,\mathrm d\tilde \sigma-\frac{1}{4}\int_{\Omega_k}\hat w^2\,\mathrm d\tilde \sigma.
\end{split}
\]
On the other hand, we have
$$
\left|\int_{\Sigma_{t,\epsilon,s}}\tilde H^{-2}\hat w_k\hat w\cdot o(1)\,\mathrm d\tilde \sigma\right|\leq o(1)(s-s_1)^{-\frac{1}{2}}\int_{\Omega_k}\hat w^4\,\mathrm d\tilde \sigma.
$$
So we can handle the second term on RHS of \eqref{Eq: integral evolution parabolic} by the first term therein. Furthermore, we can apply the Sobolev inequality from Corollary \ref{Cor: L2 Sobolev} for some fixed $q$ to the first term on RHS of \eqref{Eq: integral evolution parabolic} as following
\[
\begin{split}
\int_{\Omega_k}\tilde H^{-2}|\tilde\nabla \hat w|^2\,\mathrm d\tilde \sigma
\geq& \frac{\iota_0^2 k^2}{4c^2} (s-s_1)^{-\frac{1}{2}}\int_{\Sigma_{t,\epsilon,s}}|\tilde \nabla\hat w_k^2|\,\mathrm d\tilde\sigma\\
\geq& \frac{\iota_0^2k^2}{4c^2c_S^2q^2}(s-s_1)^{-\frac{1}{2}}\left(\int_{\Sigma_{t,\epsilon,s}}\hat w_k^{2q}\,\mathrm d\tilde \sigma\right)^{\frac{1}{q}}-\int_{\Sigma_{t,\epsilon,s}}\hat w_k^2\,\mathrm d\tilde \sigma.
\end{split}
\]
Then we obtain
\[
\begin{split}
\frac{\partial}{\partial s}\int_{\Sigma_{t,\epsilon,s}}\hat w_k^2\,\mathrm d\tilde\sigma \leq&-\tilde c^{-1} k^2q^{-2}(s-s_1)^{-\frac{1}{2}}\left(\int_{\Sigma_{t,\epsilon,s}}\hat w_k^{2q}\,\mathrm d\tilde \sigma\right)^{\frac{1}{q}}\\
&\qquad+ \frac{1}{s-s_1}\int_{\Omega_k}\hat w^2\,\mathrm d\tilde \sigma+2\int_{\Sigma_{t,\epsilon,s}}\hat w_k^2\,\mathrm d\tilde \sigma,
\end{split}
\]
where  $\tilde c$ is denoted to be a universal constant independent of $p$, $t$, $\epsilon$ and $s$, and we use the simple fact
$$
\int_{\Omega_k}\hat w^2\,\mathrm d\tilde \sigma\leq \frac{1}{s-s_1}\int_{\Omega_k}\hat w^2\,\mathrm d\tilde \sigma\mbox{ when }s_1<s\leq s_1+1.
$$

Now we arrive at exactly the same situation as in \cite[(1.7)]{Huisken2008HigherRegularity}. Through a Stampacchia iteration argument based on the previous $\mathcal L^q$-estimate, we can obtain 
$$
\|\hat w\|_{\mathcal L^\infty(\Sigma_{t,\epsilon,s_2})}\leq\tilde c\left( \gamma+\max\{1,s_1^{-1}\}^{\frac{1}{2}}(s_2-s_1)^{\frac{1-q}{4-8q}} \gamma^\frac{q}{1-2q}\right)
$$
for any fixed $s_2$ satisfying $s_1<s_2\leq s_1+1$, where $\gamma$ can be chosen to be any constant no less than $(s_2-s_1)^{1/4}$,
and $\tilde c$ to be a universal constant independent of $p$, $t$, $\epsilon$, $s_1$ and $s_2$. 

We make the following choices for $s_1$ and $s_2$ according to the value of $s_2$. 
\begin{itemize}
\item When $s_2<2$, we choose $s_1=s_2/2$ and $\gamma=(s_2-s_1)^{-1/4}$, then it follows
$$
\| w\|_{\mathcal L^\infty(\Sigma_{t,\epsilon,s_2})}\leq\tilde c s_2^{-1/2}.
$$
\item 
When $s_2\geq 2$, we take $s_1=s_2-1$ and $\gamma=1$, and so we obtain
$$
\| w\|_{\mathcal L^\infty(\Sigma_{t,\epsilon,s_2})}\leq\tilde c.
$$
\end{itemize}
To sum up, we have $$\|w\|_{\mathcal L^\infty(\Sigma_{t,\epsilon,s})} \leq \tilde c\min\{1,s\}^{-1/2}$$ for some universal constant $\tilde c$ independent of $p$, $t$, $\epsilon$ and $s$. In particular, we have 
$$\tilde H\geq \tilde c^{-1} v^{-1}\min\{1,s\}^{\frac{1}{2}}\geq \frac{1}{2c\tilde ce^{s_0}}\min\{1,s\}^{\frac{1}{2}}.$$
This completes the proof.
 \end{proof}
 
 Now we are ready to prove the following lemma with the box argument.

 \begin{lemma}\label{Lem: star-shape all time}
    There are two positive constants $\iota$ and $B$, independent of $p$, $t$, $\epsilon$ and $s$, such that the constant $\rho_1$ can be taken large enough to ensure that $\Sigma_{t,\epsilon,s}$ satisfies $$\tilde g(\tilde X,\tilde \nu(x))\geq \iota\mbox{ and }\|\tilde {\mathcal B}\|< B\area_{\tilde g}(\Sigma_{t,\epsilon,s})^{-\frac{1}{2}}$$
    for all $s\in [0,s_0]$.
 \end{lemma}
\begin{proof}
As preparation, we can take $\rho_1$ large enough to guarantee the validity of Lemma \ref{Lem: area bound scaled}, Lemma \ref{Lem: eccentricity scaled} and Lemma \ref{Lem: Hawking mass scaled}.
Then it follows from \eqref{Eq: eccentricity}, Lemma \ref{Lem: initial estimate} and Lemma \ref{Lem: eccentricity scaled} that we can pick a constant $\theta$, independent of $p$, $t$, $\epsilon$ and $s$, such that we have
$$\theta(\Sigma,p)\leq \theta\mbox{ for }\Sigma=\Sigma_t^+,\,\Sigma_{t,\epsilon}\mbox{ and }\Sigma_{t,\epsilon,s}.$$ 
Moreover, it follows from \eqref{Eq: area upper bound}, Lemma \ref{Lem: initial estimate} and Lemma \ref{Lem: area bound scaled} that we can pick a constant $A$, independent of $p$, $t$, $\epsilon$ and $s$, such that
$$\area(\Sigma)\leq A\mbox{ for }\Sigma=\Sigma_t^+,\,\Sigma_{t,\epsilon}\mbox{ and }\Sigma_{t,\epsilon,s}.$$
Let us fix
$$
\iota=\left[1+\left(\theta-1\right)^2\right]^{-\frac{1}{2}},$$
which is the starshape constant from Lemma \ref{Lem: star-shape} with the choice $\theta_0=\theta$.
By taking $\rho_1$ large enough we can further guarantee that Lemma \ref{Lem: curvature estimate scaled} holds under the star-shaped assumption \eqref{Eq: starshape assumption} with $$\iota_0=\iota/4.$$  
Then it follows from \eqref{Eq: curvature bound}, Lemma \ref{Lem: initial estimate} and Lemma \ref{Lem: curvature estimate scaled} that we can pick a positive constant $B$ such that, under the starshape assumption \eqref{Eq: starshape assumption} with $\iota_0=\iota/4$, we have 
\begin{equation}\label{Eq: curvature bound for all}
    \|{\mathcal B}_{\Sigma}\|<B\area(\Sigma)^{-\frac{1}{2}}\mbox{ or equivalently }\|\tilde{\mathcal B}_{\Sigma}\|<B\area_{\tilde g}(\Sigma)^{-\frac{1}{2}}
\end{equation}
for $\Sigma=\Sigma_t^+$, $\Sigma_{t,\epsilon}$ and $\Sigma_{t,\epsilon,s}$ with $s\leq \tau^*$.

Set
$$\theta_0=\theta,\, A_0=A\mbox{ and }B_0=B.$$
Then we can take $\rho_1$ large enough and a constant $m_0<0$ such that Lemma \ref{Lem: star-shape} holds with the above choice for $\theta_0$, $A_0$ and $B_0$. 

Define the box region
$$\mathcal R=\{\Sigma:\tilde g(\tilde X,\tilde \nu(x))\geq \iota\mbox{ and }\|\tilde {\mathcal B}_\Sigma\|\leq B \area_{\tilde g}(\Sigma)^{-\frac{1}{2}}\}.$$
We are going to verify that all the surfaces $\Sigma_{t,\epsilon,s}$ with $s\leq s_0$ stay in the box region $\mathcal R$. For the initial surface $\Sigma_{t,\epsilon}$, it follows from \eqref{Eq: area upper bound}, \eqref{Eq: Hawking mass lower bound} and Lemma \ref{Lem: initial estimate} that we can guarantee $$ m_h(\Sigma_{t,\epsilon})\geq m_0\area(\Sigma_{t,\epsilon})^{\frac{1}{2}}$$
after taking $\rho_1$ large enough. Using Lemma \ref{Lem: star-shape} we have
$$\tilde g(\tilde X,\tilde \nu(x))\geq \iota\mbox{ for }\Sigma_{t,\epsilon},$$
and so the initial surface $\Sigma_{t,\epsilon}$ lies in the box region $\mathcal R$. Similarly, it follows from  Lemma \ref{Lem: Hawking mass scaled} and Lemma \ref{Lem: star-shape} that we can take $\rho_1$ large enough such that we have
$$m_h(\Sigma_{t,\epsilon,s})\geq  m_0\area(\Sigma_{t,\epsilon,s})^{\frac{1}{2}},$$ 
and so under the starshape assumption \eqref{Eq: starshape assumption} with $\iota_0=\iota/4$ we have
\begin{equation}\label{Eq: improved starshape for all}
    \tilde g(\tilde X,\tilde \nu(x))\geq \iota\mbox{ for all }\Sigma_{t,\epsilon,s}\mbox{ with }s\leq \tau^*.
\end{equation}
In the following, we consider the first moment $s^*$ when the flow $\Sigma_{t,\epsilon,s}$ passes through the boundary of the box region as shown in Figure \ref{Fig: 2}. 
It suffices to show $s^*\geq s_0$ and let us derive contradictions when $s^*<s_0$: 
\begin{itemize}
\item If $\Sigma_{t,\epsilon,s^*}$ is on the $B$-edge, then we have the starshape assumption \eqref{Eq: starshape assumption} with $\iota_0=\iota/4$ and $\tau^*=s^*$. Recall from \eqref{Eq: curvature bound for all} we have 
$$\|\tilde {\mathcal B}_{\Sigma_{t,\epsilon,s^*}}\|<B\area_{\tilde g}(\Sigma_{t,\epsilon,s^*})^{-\frac{1}{2}},$$
which leads to a contradiction.
\item If $\Sigma_{t,\epsilon,s^*}$ is on the $\iota$-edge, according to Lemma \ref{Lem: star-shape to curvature} and Krylov's regularity theory, we know that $s^*$ is strictly less than the maximal existence time $\tau$. Since the flow $\Sigma_{t,\epsilon,s}$ passes through the boundary of the box region at the moment $s^*$, we can find a nearby moment $s^{**}$ such that $\inf \tilde g(\tilde X,\tilde \nu)<\iota$ for $\Sigma_{t,\epsilon,s^{**}}$. From the continuity we can also guarantee the starshape assumption \eqref{Eq: starshape assumption} with $\iota_0=\iota/4$ and $\tau^*=s^{**}$. Recall from \eqref{Eq: improved starshape for all} we have $\inf \tilde g(\tilde X,\tilde \nu)\geq \iota$ for $\Sigma_{t,\epsilon,s^{**}}$, which leads to a contradiction.
\end{itemize}
We complete the proof.
\end{proof}
\begin{figure}[htbp]
\centering
\includegraphics[width=10cm]{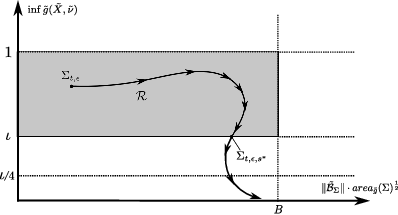}
\caption{The first pass-through moment $s^*$}
\label{Fig: 2}
\end{figure}

\begin{proof}[Proof of Lemma \ref{Lem: smooth IMCF}]
Notice that all the surfaces $\Sigma_{t,\epsilon,s}$ enclose a ball with fixed size, so we have $\area_{\tilde g}(\Sigma_{t,\epsilon,s})\geq c$ for some universal positive constant $c$ independent of $p$, $t$, $\epsilon$ and $s$. Combined with Lemma \ref{Lem: star-shape all time} we see
$$\tilde g(\tilde X,\tilde \nu)\geq \iota\mbox{ and }\|\tilde {\mathcal B}\|\leq Bc^{-\frac{1}{2}}$$
for all $\Sigma_{t,\epsilon,s}$ with $s\in [0,s_0]$.
    Using Lemma \ref{Lem: star-shape to curvature} and also Krylov's regularity theory \cite{Krylov1987NonlineaPDE}, 
    we can derive uniform interior estimates for all derivatives of $\tilde {\mathcal B}$, independent of $p$, $t$, $\epsilon$ but depending on $s$. In particular, all the smooth inverse mean curvature flows $\Phi_\epsilon$ exist on the uniform time interval $[0,s_0]$, and they converge to a limit flow $\Phi^+:N\times (0,s_0]\to \mathbb R^3$. From the uniformly bounded curvature we conclude
    $$d_{\mathcal H}(\Phi^+(N,s),\Sigma_t^+)\to 0 \mbox{ as }s\to 0.$$
    From \eqref{Eq: annulus estimate scaled} we can derive that $\Phi^+(N,s)$ encloses $D_{c^{-1}e^{\frac{t}{2}+\frac{s}{3}}}(p)$ when $\rho_1$ is taken to be large enough.
\end{proof}

\subsubsection{Regularity of weak solution for inverse mean curvature flow}

\begin{lemma}\label{Lem: smooth solution}
The constant $\rho_1$ can be taken large enough such that for any point $p\in \mathbb R^3\setminus \bar D_{\rho_1}$ there is a smooth inverse mean curvature flow $\Phi: N\times (-\infty,T]\to \mathbb R^3$ such that $\Phi(N,t)=\Sigma_t=\Sigma_t^+$.
\end{lemma}
\begin{proof}
Since we have $\Sigma_t^+\subset D_{ce^{t/2}}(p)$ and $\Phi^+(N,s)$ encloses $D_{c^{-1}e^{\frac{t}{2}+\frac{s}{3}}}(p)$, we can fix $s_0$ large enough such that $\Phi^+(N,s_0/2)$ encloses $\Sigma_{t+1}^+$. Then it follows from the comparison principle \cite[Theorem 2.2 (ii)]{huisken2001inverse} that $\Sigma_{t+s}$ is enclosed by $\Phi^+(N,s)$ for each $s\in (0,1]$. Using the comparison principle in the opposite direction, we know that $\Phi^+(N,s)$ is enclosed by $\Sigma_{t+s}$ for all $s\in (0,s_0)$. In particular, we see 
$$\Sigma_{t+s}=\Sigma^+_{t+s}=\Phi^+(N,s)\mbox{ when }s\in(0,1],$$ 
and so the weak solution for inverse mean curvature is smooth in the region $\{t<u<t+1\}$. The proof is completed by letting $t$ vary in $(-\infty,T]$.
\end{proof}

\begin{proof}[Proof of Theorem \ref{main-theorem-IMCF}]
    This follows from Lemma \ref{Lem: smooth solution} and the fact that $\Sigma_t^+$ with $t\leq T$ are all topological spheres as well as the annulus neighborhood estimate \eqref{Eq: annulus control}.
\end{proof}

\subsection{The sweep-out with sub-Euclidean isoperimetric ratio}
In the following, we work with a complete asymptotically flat manifold $(M^3,g)$ with non-negative scalar curvature and possibly with a compact minimal surface as boundary.
\begin{lemma}\label{Lem: sub-Euclidean ratio}
   Let $\Phi:\mathbb S^2\times (-\infty,T]\to(M,g)$ be the smooth inverse mean curvature flow from a point $p$ constructed in Theorem \ref{main-theorem-IMCF}. If $(M,g)$ is not flat at the point $p$, then we have
    \begin{equation}\label{Eq: isoperimetric ratio}
    \frac{\area(\partial U_t)}{\vol(U_t)^{\frac{2}{3}}}<(36\pi)^{\frac{1}{3}},\end{equation}
    where $U_t$ denotes the region enclosed by $\Sigma_t$.
\end{lemma}
\begin{proof}
We do the same computation as in \cite{shi2016isoperimetric}. Denote
$$A(t)=\area(\partial U_t),\,V(t)=\vol(U_t)\mbox{ and }m(t)=m_h(\partial U_t).$$
From Lemma \ref{Lem: Hawking mass limit} and the discussion in \cite[Page 423]{huisken2001inverse} concerning the equality case of the Geroch monotonicity, we know that $m(t)>0$ for all $t\in(-\infty,T]$, and in particular, we have
$$\int_{N_t}H^2\,\mathrm d\sigma =16\pi-(16\pi)^{\frac{3}{2}}A^{-\frac{1}{2}}(t)m(t)<16\pi.$$ 
Through a straight-forward computation we have
$$A'(t)=A(t)\mbox{ and }V'(t)=\int_{N_t}H^{-1}\,\mathrm d\sigma.$$
The H\"older inequality yields
$$A(t)\leq \left(\int_{N_t}H^2\,\mathrm d\sigma\right)^{\frac{1}{3}} \left(\int_{N_t}H^{-1}\,\mathrm d\sigma\right)^{\frac{2}{3}},$$
and then we can derive
$$(A^{\frac{3}{2}})'(t)\leq \frac{3}{2} V'(t) \left(\int_{N_t}H^2\,\mathrm d\sigma\right)^{\frac{1}{2}}<V'(t)(36\pi)^{\frac{1}{2}}.$$
Integrating this inequality from $-\infty$ to $t$, we obtain $A^{\frac{3}{2}}(t)<(36\pi)^{\frac{1}{2}} V(t)$, and this yields the desired inequality \eqref{Eq: isoperimetric ratio}.
\end{proof}
\begin{proof}[Proof of Proposition \ref{Prop: sweep-out}]
Take the asymptotically flat end $M\setminus K$. Since $(M,g)$ does not have Euclidean end, we can pick a non-flat point $p\in M\setminus K$ outside any compact subset. Notice that the metric $g$ is equivalent to $g_{euc}$, so we can take $T$ large and $p$ sufficiently near infinity such that the smooth inverse mean curvature flow $\Phi:\mathbb S^2\times (-\infty,T]\to (M,g)$ in Theorem \ref{main-theorem-IMCF} satisfies
$$\vol(U_T)\geq \vol(D_{c^{-1}e^{T/2}}(p))\geq v,$$
where $U_t$ denotes the region enclosed by $\Phi(\mathbb S^2,t)$. The proof is completed by taking a sub-family with reparametrization as well as Lemma \ref{Lem: sub-Euclidean ratio}.
\end{proof}

\section{Sweep-out with sub-hyperbolic isoperimetric ratio in asymptotically hyperbolic manifolds}
In this section, a complete Riemannian manifold $(M^3,g)$ is said to be asymptotically hyperbolic, if there exist a compact set $K\subset M$ and a diffeomorphism $M\setminus K \equiv \R^3 \setminus B_1(0)$ such that, in the polar coordinates induced by this diffeomorphism, the metric $g$ takes the form 
\[
g = dr^2 + \sinh^2 r\, g_{\mathbb S^2} + Q
\]
where $Q$ satisfies 
\begin{equation}\label{Eq: C2 decay AH}
\vert Q\vert_g + \vert \del_h Q\vert_g + \vert \del_h^2 Q\vert_g+ |\nabla_h^3 Q|_g \le C e^{-3r}
\end{equation}
for some constant $C$, with respect to the hyperbolic covariant derivative $\nabla_h$.

The goal of this section is to prove an analogy of Proposition \ref{Prop: sweep-out} in the case of asymptotically hyperbolic manifolds. Namely, we are going to show
\begin{proposition}\label{Prop: sweep-out AH}
    Let $(M^3,g)$ be an asymptotically hyperbolic manifold with scalar curvature
$R\geq -6$, possibly with compact boundary. Assume that $M$ does not have hyperbolic end. Then given any constant $v>1$, there exists a continuously varying family of open sets $\{U_t\}_{t\in[0,1]}$ with $\vol_g(U_0)=\emptyset$ and $\vol(U_T)=v$ satisfying
$$\area_g(\partial U_t)< \mathcal I_{hyp}(\vol_g(U_t))$$
for all $t\in (0,1]$, where $\mathcal I_{hyp}$ is the hyperbolic isoperimetric profile.
\end{proposition}\label{Prop: smooth IMCF AH}
 As before, we extend $(M\setminus K,g)$ to be a complete Riemannian manifold $(\mathbb R^3,\hat g)$. We start with the following existence of weak solutions for inverse mean curvature flow.
 \begin{lemma}\label{Lem: weak solution AH}
    For any point $p\in \mathbb R^3$ there is a locally Lipschitz function $u$ in $\mathbb R^3\setminus\{p\}$ such that
    \begin{itemize}
        \item $u(x)\to-\infty$ as $x\to p$ and $u(x)\to+\infty$ as $x\to\infty$;
        \item $u$ is a weak solution for inverse mean curvature flow.
    \end{itemize}
\end{lemma}
\begin{proof}
    Take the subsolution to be
    $$w=\max\{|x|,1\}.$$
    Then we can just repeat the proof of Lemma \ref{Lem: weak solution} except that here we use the hyperbolic isoperimetric inequality instead of the Euclidean isoperimetric inequality.
\end{proof}
In the following, we use $ g_{hyp}$ to denote the hyperbolic metric on $\mathbb R^3$ given by 
$$g_{hyp}=\mathrm dr^2+\sinh^2 rg_{S^2},$$
and we always use $D^h_r(p)$ to denote the hyperbolic $r$-ball centered at the point $p$.
For any point $p\in \mathbb R^3$ we use $u$ to denote the weak solution for inverse mean curvature flow from Lemma \ref{Lem: weak solution AH} with the normalization
\begin{equation}\label{Eq: normalization AH}
    \min_{\partial D^h_1(p)}u=0.
\end{equation}

As before, our goal is to show Theorem \ref{main-theorem-IMCF-AH} --- the analogy of Theorem \ref{main-theorem-IMCF} in the case of asymptotically hyperbolic manifolds. In the following, we still use the notation
$$\Sigma_t=\partial\{u<t\}\mbox{ and }\Sigma_t^+=\partial\{u\leq t\}.$$
We point out that Proposition \ref{Prop: sweep-out AH} follows immediately from Theorem \ref{main-theorem-IMCF-AH} with the help of the following lemma.
\begin{lemma}\label{Lem: Hawking mass limit AH}
Denote
$$
m^*_{h}(\Sigma_t)=\frac{\area(\Sigma_t)^{\frac{1}{2}}}{(16\pi)^{\frac{3}{2}}}\left(16\pi-\int_{\Sigma_t}\left(H^2-4\right)\right).
$$
Then the limit
\begin{equation*}
\lim_{t\to-\infty}  m_{h}^*(\Sigma_t)
\end{equation*}
exists and is non-negative.
\end{lemma}
\begin{proof}
    Notice that
    $$ m_{h}^*(\Sigma_t)=m_h(\Sigma_t)+4\left(\frac{\area(\Sigma_t)}{16\pi}\right)^{\frac{3}{2}},$$
    where $m_h(\Sigma_t)$ is the Hawking mass defined in \eqref{Eq: Hawking mass}.
  Then the desired consequence follows from  the fact $
     \area(\Sigma_t)\to 0$ as $t\to-\infty$
    and Lemma \ref{Lem: Hawking mass limit}.
\end{proof}

\begin{proof}[Proof of Proposition \ref{Prop: sweep-out AH} using Theorem \ref{main-theorem-IMCF-AH}]
     Using the assumption that the manifold $(M,g)$ does not have hyperbolic end, we can find a non-hyperbolic point $p$ outside any compact subset. Let us take 
     $$\Phi:\mathbb S^2\times (-\infty,T]\to (M,g)$$ to be the smooth inverse mean curvature flow in Theorem \ref{main-theorem-IMCF-AH}, and we use $U_t$ to denote the region enclosed by the surface $\Phi(\mathbb S^2,t)$. Clearly, we can take $T$ large enough such that $\vol(U_T)\geq v$. Moreover, it follows from the Geroch monotonicity and its rigidity for equality that we have 
    $$ m_{h}^*(\partial U_t)>0\mbox{ for all }t\in(-\infty,T].$$
Denote $A(t)=\area(\partial U_t)$ and $V(t)=\vol(U_t)$. The same computation as in the proof of Lemma \ref{Lem: sub-Euclidean ratio} yields
$$(A^{\frac{3}{2}})'(t)<\frac{3}{2}V'(t)(16\pi+4A(t))^{\frac{1}{2}},$$
or equivalently
$$A'(V)<\left(\frac{16\pi+4A}{A}\right)^{\frac{1}{2}}.$$
The hyperbolic isoperimetric profile satisfies $$\mathcal I_H'=\left(\frac{{16\pi+4\mathcal I_H}}{{\mathcal I_H}}\right)^{\frac{1}{2}}.$$ From a comparison argument we can show $\area(\partial U_t)<\mathcal I_H(\vol(U_t))$ for all $t\in(-\infty,T]$. The proof is then completed by taking a sub-family and doing reparametrization.
\end{proof}

\subsection{Smoothness of weak solution in small scale}
First, we point out that the discussion in Section \ref{Sec: sweep-out AF} actually leads to the following
\begin{proposition}\label{Prop: smoothness approximate Euclidean}
Let $(M_i,g_i,p_i)$ be any sequence of non-compact pointed Riemannian manifolds and $u_i:M_i\setminus\{p_i\}\to (-\infty,+\infty)$ be locally Lipschitz functions. If the following statements are true:
\begin{itemize}
\item $u_i$ is a weak solution for inverse mean curvature flow in $M_i\setminus\{p_i\}$ such that $u_i\to-\infty$ as $x\to p_i$, $u_i\to +\infty$ as $x\to\infty$, and also $$\min_{\partial B_1^{g_i}(p_i)}u_i=0;$$
\item $(M_i,g_i,p_i)$ converge to $(\mathbb R^3,g_{euc},O)$ in $C^3$-sense as $i\to\infty$,
\end{itemize}
then for any constant $T$ there is a constant $i_0\in \mathbb N_+$ such that $u_i$ is smooth up to the moment $T$ for all $i\geq i_0$.
\end{proposition}
\begin{proof}
In Section \ref{Sec: sweep-out AF} we deal with those pointed Riemannian manifolds given by $(\mathbb R^3,\hat g,p,u)$, where $(\mathbb R^3,\hat g)$ is asymptotically flat. It is not difficult to check that all the arguments still work here, and so the proposition follows.
\end{proof}

In particular, we have the following corollary.
\begin{corollary}\label{Cor: smoothness small scale}
There is a constant $T_s$ such that at any point $p\in \mathbb R^3\setminus D_2$ with $D_2=\{x\in\mathbb R^3:|x|<2\}$ the weak solution $u$ for inverse mean curvature flow is smooth up to the moment $T_s$.
\end{corollary}
\begin{proof}
Suppose by contradiction that there is a sequence of triples $(p_i,u_i,T_i)$ with $T_i\to-\infty$ as $i\to\infty$ such that the weak solution $u_i$ cannot be smooth up to the moment $T_i$. Denote
$$r^h_i=\inf\{r>0:u_i> T_i\mbox{ in } \mathbb R^3\setminus D^h_r(p_i)\}.$$

We claim that $r^h_i\to 0$ as $i\to\infty$. Notice that the disks $(D^h_1(p_i),\hat g)$ have uniformly bounded geometry. From \cite[Theorem 3.1]{huisken2001inverse} we have the uniform gradient estimate
$$|\nabla u_i|(x)\leq C\dist(x,p_i)^{-1}\mbox{ in }D^h_1(p_i)\setminus\{p_i\},$$
where $C$ denotes a universal constant independent of $i$. Combined with \eqref{Eq: normalization AH}, for any fixed constant $\rho\in (0,1)$, we have $u_i\geq -\Lambda$ on $\partial D^h_\rho(p_i)$ for some universal constant $\Lambda=\Lambda(\rho)$. From \eqref{Eq: exterior lower bound} we see $u_i\geq -\Lambda$ in $\mathbb R^3\setminus D^h_\rho(p_i)$, and so we have $r^h_i\leq \rho$ when $T_i\leq -\Lambda$, which means $\limsup_{i\to\infty} r^h_i\leq \rho$. Let $\rho\to 0$ and then we obtain $r^h_i\to 0$ as $i\to \infty$.

Denote
$$r_i=\inf\{r>0:u_i> T_i\mbox{ in } \mathbb R^3\setminus B_r(p_i)\},$$
where $B_r(p_i)$ denotes the $\hat g$-geodesic $r$-ball centered at the point $p_i$. In a similar way,
using the uniformly bounded geometry of $B_1(p_i)$, we can derive $r_i\to 0$ as $i\to \infty$ as well. After taking $(M_i,g_i,p_i,u_i)$ to be $$(\mathbb R^3,r_i^{-2}\hat g,p_i,u_i+T_i)$$ in Proposition \ref{Prop: smoothness approximate Euclidean}, we conclude that, for $i$ large enough, $u_i$ is smooth up to the moment $T_i$, which leads to a contradiction.
\end{proof}

\subsection{Smoothness of weak solution in large scale}
The proof follows the line of those arguments in Section \ref{Sec: sweep-out AF}. First, let us introduce some preliminary lemmas.
\begin{lemma}\label{Lem: standard solution hyperbolic}
    Let $u_O:\mathbb R^3\setminus\{O\}\to (-\infty,+\infty)$ be a weak solution for inverse mean curvature flow in the hyperbolic space $\mathbb H^3=(\mathbb R^3,g_{hyp})$ such that $u_O(x)\to -\infty$ as $x\to O$, $u_O(x)\to +\infty$ as $x\to \infty$, and $\min u_O=0$ on $\partial D^h_1(O)$. Then we have $$u_O=2\ln \sinh |x|-2\ln \sinh 1.$$
\end{lemma}
\begin{proof}
    We apply a similar eccentricity argument as in the work \cite{huisken2001inverse}. For any connected closed surface $\Sigma$ in $\mathbb R^3$ enclosing the origin $O$, we define the $(\ln\sinh)$-eccentricity to be
    $$\theta(\Sigma)=\frac{\ln \sinh \bar r(\Sigma)}{\ln \sinh \underline r(\Sigma)}\in[1,+\infty),$$
    where  $\bar r(\Sigma)$ and $\underline r(\Sigma)$ are outer radius and inner radius of $\Sigma$ defined in \eqref{Eq: outer radius} and \eqref{Eq: inner radius}, respectively. Denote $$N_t=\partial\{u_O<t\}.$$ It is clear that $N_t$ is a connected closed surface enclosing the origin $O$. Through a comparison of $u_O$ with the standard solution $2\ln\sinh |x|$, it is easy to derive
    $$\theta(N_{t+s})\leq \frac{2\ln\log \bar r(N_t)+s}{2\ln\log \underline r(N_t)+s}\leq \theta(N_t)\mbox{ for any }s>0.$$
    As in the proof of Corollary \ref{Cor: smoothness small scale} we can show $\bar r(N_t)\to 0$ as $t\to-\infty$. Through a blow-up analysis we also know $\bar r(N_t)/\underline r(N_t)\to 1$ as $t\to-\infty$. From these facts we conclude
    $$\theta(N_t)\to 1\mbox{ as }t\to-\infty.$$
    Combined with the monotonicity of $\theta(N_t)$, we know that $N_t$ are all round spheres in $\mathbb R^3$ centered at the origin. So $u_O$ is the standard solution.
\end{proof}

\begin{lemma}\label{Lem: radius upper bound AH}
    For any constant $T$ there is a constant $\Lambda=\Lambda(T)$ such that $\{u\leq T\}$ is contained in $D^h_{\Lambda}(p)$.
\end{lemma}
\begin{proof}
    Recall that $\{u\leq t\}$ cannot have precompact components, so we see that $\Sigma_t^+$ is connected for all $t$. Therefore, it suffices to show that there exists a constant $\Lambda$ such that $\Sigma_t^+\subset D_\Lambda^h(p)$ for all $t\leq T$. Clearly, we have uniform area estimate $\area(\Sigma_t^+)\leq A_0$ for all $t\leq T$. Using the fact that $(\mathbb R^3,\hat g)$ has bounded geometry, from \cite{huisken2001inverse} and \cite{heidusch2001regularitat}  we know that $\Sigma_t^+$ has uniformly bounded second fundamental form in $\mathbb R^3\setminus D^h_1(p)$. Combined with the area estimate we conclude that $\Sigma_t^+$ has uniformly bounded diameter, and so we have $\Sigma_t^+\subset D^h_\Lambda(p)$ for some constant $\Lambda$.
\end{proof}
\begin{lemma}\label{Lem: radius lower bound AH}
    For any constant $T$ there is a constant $\lambda=\lambda(T)>0$ such that $\Sigma_t$ encloses $D^h_\lambda(p)$ for all $t\geq T$.
\end{lemma}
\begin{proof}
    In small neighborhoods of $p$ the geometry is Euclidean-like, so we can establish similar estimates as \eqref{Eq: interior upper bound from area}. Then the lemma follows easily.
\end{proof}

In the following, we shall continue our discussion in the Poincar\'e model. At each point $p\in \mathbb R^3$, let us take a coordinate chart 
$$\mathcal X_p:(\mathbb B^3,O)\to (\mathbb R^3,p)$$
such that 
$$\mathcal X_p^*g_{hyp}=g_{poi}:=\frac{4}{(1-|x|^2)^2}g_{euc}.$$
From the asymptotically hyperbolic assumption, we know that $(\mathbb B^3,O,\mathcal X_p^*\hat g)$ converges to the hyperbolic space in $C^3$-sense as $p\to\infty$.

Denote
$$\mathfrak r(t)=\frac{e^{\frac{t}{2}}\sinh 1 }{\sqrt{1+e^t\sinh^21}+1}.$$
We remark that the inverse function of $\mathfrak r(t)$ gives the standard solution of the inverse mean curvature flow in the Poincar\'e model, and so the exponential area growth gives
\begin{equation}\label{Eq: transfer to ball model}
\area_{g_{poi}}(\partial \mathbb B_{\mathfrak r(t)})=4\pi e^t\sinh^2 1,
\end{equation}
where  $\mathbb B_{(\cdot)}$ denotes the Euclidean ball in $\mathbb B^3$ centered at the origin.
\begin{lemma}
    There is a constant $\rho_1>0$ such that, for any point $p$ in $\mathbb R^3\setminus \bar D_{\rho_1}$, the weak solution $u$ satisfies
    \begin{equation}\label{Eq: annulus estimate AH}
         \mathbb B_{\mathfrak r(t-1)}\subset \mathcal X_p^{-1}(\{u<t\})\subset \mathcal X_p^{-1}(\{u\leq t\})\subset \mathbb B_{\mathfrak r(t+1)}
    \end{equation}
    for any $t\in [T_s, T]$.
\end{lemma}
\begin{proof}
    It suffices to show that, for all $t\in [T_s,T]$, the surfaces $\Sigma_t$ and $\Sigma_t^+$ lie in the annulus region $\mathbb B_{\mathfrak r(t+1)}\setminus\bar{ \mathbb B}_{\mathfrak r(t-1)}$. 
    
    First, we show that, as $p$ goes to the infinity, the weak solutions $u_p\circ \mathcal X_p$ converge to the standard solution $u_O\circ \mathcal X_O$  up to a subsequence, where $u_O$ denoted the solution from Lemma \ref{Lem: standard solution hyperbolic}. It is clear that manifolds $(\mathbb B^3,O,\mathcal X_p^*\hat g)$ converge to $(\mathbb B^3,O,g_{poi})$ in locally $C^2$-sense as $p$ goes to the infinity. From \eqref{Eq: normalization AH} and local gradient estimate, we know that the weak solutions $u_p\circ\mathcal X_p$ converge to a weak solution for inverse mean curvature flow in $\mathbb B^3\setminus\{O\}$ up to a subsequence, denoted by $u_\infty$. Using Lemma \ref{Lem: radius lower bound AH} and Lemma \ref{Lem: radius upper bound AH}, we conclude that $u_\infty\to -\infty$ as $x\to O$ and $u_\infty\to +\infty$ as $x\to \infty$. From Lemma \ref{Lem: standard solution hyperbolic} we see $u_\infty=u_O\circ \mathcal X_O$.

    For all $t\in[T_s,T]$, the surfaces $\Sigma_t$ and $\Sigma_t^+$ have uniformly bounded area independent of $p$. Also, $\Sigma_t$ and $\Sigma_t^+$ have a definite distance to $p$ due to Lemma \ref{Lem: radius lower bound AH}. Therefore, it follows from the gradient estimate in \cite{huisken2001inverse} and the curvature estimate in \cite{heidusch2001regularitat} that the surfaces $\Sigma_t$ and $\Sigma_t^+$ have uniformly bounded second fundamental form. Using a standard contradiction-limiting argument, we can finally obtain the desired annulus neighborhood estimate \eqref{Eq: annulus estimate AH}.
\end{proof}

\begin{lemma}
    Given constants $\lambda_0>0$, $\theta_0>0$, $A_0>0$ and $B_0>0$ we can find constants $m_0<0$ and $\rho_1>0$ such that for each point $p\in \mathbb R^3\setminus \bar D_{\rho_1}$ if $\Sigma$ is a connected closed surface satisfying the following properties:
    \begin{itemize}
        \item $\Sigma$ encloses $D^h_{\lambda_0}$ and is contained in $D^h_{{\lambda_0^{-1}}}(p)$;
        \item $\Sigma$ satisfies $\area(\Sigma)\leq A_0$;
        \item $\mathcal X_p^{-1}(\Sigma)$ lies in the annulus region $\mathbb B_{\mathfrak r(t+\theta_0)}\setminus \bar{\mathbb B}_{\mathfrak r(t-\theta_0)}$ with
        $t=\ln\area(\Sigma)$;
        \item the second fundamental form of $\Sigma$ satisfies $\|\mathcal B_\Sigma\|\leq B_0$;
        \item the Hawking mass of $\Sigma$ satisfies
        $m_{h}^*(\Sigma)\geq m_0$,
    \end{itemize}
    then we have 
    \begin{equation}\label{Eq: star AH}
    \hat g(x-p,\nu(x))\geq {\iota_0}\cdot |x-p|_{\hat g}
    \end{equation}
    for all $x\in \Sigma$, where $\iota_0$ is a constant determined by $\lambda_0$.
\end{lemma}
\begin{proof}
    We argue by contradiction. Suppose that there is a sequence of pairs $(m_i,\rho_i)$ with $(m_i,\rho_i)\to(0,+\infty)$ as $i\to\infty$ such that there is a closed surface, denoted by $\Sigma_{i}$, enclosing some point $p_i$ in $\mathbb R^3\setminus\bar D_{\rho_i}$, which satisfies all the assumptions but \eqref{Eq: star AH}. Our goal is to show that if we take appropriate value for $\iota_0$ determined by $\lambda_0$, then we can obtain a contradiction.
    
    With $\iota_0$ to be determined later, we shall investigate the convergence of  $$(\mathbb B^3,O,\mathcal X^*_{p_i}\hat g,\mathcal X_{p_i}^{-1}(\Sigma_i)).$$ Recall that $(\mathbb R^3,\hat g)$ is asymptotically hyperbolic, so the ambient manifolds $(\mathbb B^3,O,\mathcal X_{p_i}^*\hat g)$ converge to the Poincar\'e model $(\mathbb B^3,O,g_{poi})$ as $i\to\infty$. Recall that each $\Sigma_i$ is contained in $D^h_{\lambda_0^{-1}}(p_i)$, so all surfaces $\mathcal X_{p_i}^{-1}(\Sigma_i)$ are contained in a fixed compact subset. In particular, we conclude that the surfaces $\mathcal X_{p_i}^{-1}(\Sigma_i)$ have uniformly bounded second fundamental forms and areas in the Poincar\'e model. Therefore, $\mathcal X_{p_i}^{-1}(\Sigma_i)$ must converge to a closed surface $\mathcal S$ in $C^{1,\alpha}$-sense, up to a subsequence. From the $C^{1,\alpha}$-convergence we have
    \begin{equation}\label{Eq: not star shape}
    g_{hyp}(x_0,\nu(x_0))\leq \iota_0\cdot|x_0|_{hyp}\mbox{ at some }x_0\in \mathcal X_O(\mathcal S).
    \end{equation}
    
    In the following discussion, we view $\mathcal X_{p_i}^{-1}(\Sigma_i)$ as surfaces in the Euclidean ball $(\mathbb B^3,g_{euc})$, and do some estimates. Similarly as before, we can derive the estimate
    $$m_{h,hyp}^*(\Sigma_i)=m_h^*(\Sigma_i)+o(1)\mbox{ as }i\to\infty,$$
    where we use $m_{h,hyp}^*(\Sigma_i)$ to denote the Hawking mass of $\Sigma_i$ in the hyperbolic space $(\mathbb R^3,g_{hyp})$. Since $\Sigma_i$ encloses $D^h_{\lambda_0}(p_i)$, $\area(\Sigma_i)$ are bounded from below by a positive constant, which implies
    $$\int_{\Sigma_i}((H_i^h)^2-4)\mathrm d\sigma_{hyp}=16\pi+o(1),$$
    where $H_i^h$ denotes the mean curvature of $\Sigma_i$ in the hyperbolic space. Denote $\bar H_i$ to be the mean curvature of $\mathcal X_{p_i}^{-1}(\Sigma_i)$ in the Euclidean ball. This further implies
    $$\int_{\mathcal X_{p_i}^{-1}(\Sigma_i)}H_i^2\,\mathrm d\sigma_{euc}=16\pi+o(1).$$
    Combined with the resolution of the Willmore conjecture (see \cite[Theorem A]{Marques2014Willmore}), we see that each surface $\mathcal X_{p_i}^{-1}(\Sigma_i)$ is a topological sphere for large $i$. Using the Gauss equation, we can derive
    $$\int_{\mathcal X_{p_i}^{-1}(\Sigma_i)}\|\mathring{\bar{\mathcal B_i}}\|^2_{g_{euc}}\,\mathrm d\sigma_{euc}=o(1),$$
    where $\mathring{\bar{\mathcal B_i}}$ denotes the trace-free part of the second fundamental form of $\mathcal X_{p_i}^{-1}(\Sigma_i)$ in the Euclidean ball. On the other hand, we can obtain Euclidean-area estimates for $\mathcal X_{p_i}^{-1}(\Sigma_i)$. To see this, we compute in the Poincar\'e model that
    $$\area_{g_{hyp}}(\Sigma_i)=\int_{\mathcal X_{p_i}^{-1}(\Sigma_i)}\frac{4}{(1-|x|^2)^2}\,\mathrm d\sigma_{euc}.$$
    Using the first property, the third properties, and also \eqref{Eq: transfer to ball model}, we have
    $$e^{t-\theta_0}\sinh^21\cdot\frac{1}{\tanh^2\lambda_0^{-1}}\leq \frac{4}{(1-|x|^2)^2}\leq e^{t+\theta_0}\sinh^21\cdot \frac{1}{\tanh^2\lambda_0}.$$
    Recall that we have $t=\ln\area(\Sigma_i)$ and
    $$\area_{g_{hyp}}(\Sigma_i)=(1+o(1))\area(\Sigma_i).$$
    We conclude that the Euclidean-areas of $\mathcal X_{p_i}^{-1}(\Sigma_i)$ are uniformly bounded from above and below. It follows from \cite[Theorem 1.1]{de2005optimal} that the surfaces $\mathcal X_{p_i}^{-1}(\tilde \Sigma_i)$ converge to a geodesic sphere in $C^0$-sense, up to a subsequence. In particular, we know that the $C^{1,\alpha}$-limit $\mathcal S$ is a geodesic sphere enclosing $\mathbb B_{\tanh\lambda_0}$ and contained in $\mathbb B_{\tanh \lambda_0^{-1}}$.
    
    From the compactness of those geodesic spheres enclosing $\mathbb B_{\tanh\lambda_0}$ and contained in $\mathbb B_{\tanh \lambda_0^{-1}}$ at the same time, there exists a constant $c$ depending only on $\lambda_0$ such that for any geodesic sphere $\Sigma$ we have
      $$g_{hyp}(x,\nu(x))\geq c\cdot|x|_{hyp}\mbox{ for all }x\in \mathcal X_O(\Sigma).$$   
      Take $\iota_0=c/2$. Then we can obtain a contradiction to \eqref{Eq: not star shape}. This completes the proof.
\end{proof}

It is easy to check that the smoothing argument with mean curvature flow still works in asymptotically hyperbolic case. Namely, we have
\begin{lemma}\label{Lem: MCF AH}
   The statement of Lemma \ref{Lem: MCF} is true when $(M,g)$ is replaced by an asymptotically hyperbolic manifold.
\end{lemma}

In the following discussion, we are going to construct smooth inverse mean curvature flows from $\Sigma_t^+$ with $t\in [T_s,T]$, where $T_s$ is the constant in Corollary \ref{Cor: smoothness small scale} and $T$ is the prescribed constant in Theorem \ref{main-theorem-IMCF-AH}. First, we list the properties of $\Sigma_t^+$ to be used as follows. From Lemma \ref{Lem: radius lower bound AH} there is a constant $\lambda_0$ such that 
\begin{equation}\label{Eq: c1 AH}
\Sigma_t^+ \mbox{ encloses } D^h_{\lambda_0}(p_i)\mbox{ for all }t\geq T_s.
\end{equation}
From Lemma \ref{Lem: standard solution hyperbolic}, Lemma \ref{Lem: radius upper bound AH}, Lemma \ref{Lem: radius lower bound AH}, and a limiting argument, we can guarantee by taking $\rho_1$ large enough that
\begin{equation}\label{Eq: c2 AH}
\mathcal X_p^{-1}(\Sigma_t^+)\mbox{ lies in the annulus region }\mathbb B_{\mathfrak r(t+1)}\setminus \bar{\mathbb B}_{\mathfrak r(t-1)}
\end{equation}
for all $t\in [T_s,T]$. From Lemma \ref{Lem: radius upper bound AH} and the outer-minimizing property of $\Sigma_t^+$, we have
\begin{equation}\label{Eq: c3 AH}
\area(\Sigma_t^+)\leq A_0 \mbox{ for all }t\in (-\infty,T],
\end{equation}
where $A_0$ is a constant independent of $t$. From the Geroch monotonicity as well as Lemma \ref{Lem: Hawking mass limit AH}, we have
\begin{equation}\label{Eq: c4 AH}
m_h^*(\Sigma_t^+)\geq o(1)\mbox{ for all }t\in (-\infty,T],
\end{equation}
as $\rho_1\to \infty$. It follows from \cite[Theorem 3.1 and Theorem 1.3]{huisken2001inverse}, \cite[Theorem 5.5]{heidusch2001regularitat}, and the annulus neighborhood estimate \eqref{Eq: c2 AH} that, when $\rho_1$ is taken large enough, we have
\begin{equation}\label{Eq: c5 AH}
\|\mathcal B_{\Sigma_t^+}\|_{\mathcal L^\infty}\leq \Lambda\mbox{ for all }t\in[T_s,T],
\end{equation}
where $\Lambda$ denotes a universal constant independent of $p$ and $t$.

From Lemma \ref{Lem: MCF AH} we can take
$$\Psi:N\times(0,\tau)\to \mathbb R^3$$ to be the mean curvature flow in Lemma \ref{Lem: MCF} starting from the surface $\Sigma_t^+$. Recall that $N$ is the diffeomorphism type of $\Sigma_t^+$. Similarly as before, we use the notation $\Sigma_{t,\epsilon}:=\Psi(N,\epsilon)$ and we have
\begin{lemma}\label{Lem: N_t^+ properties AH}
The surfaces $\Sigma_{t,\epsilon}$ satisfy all the estimates \eqref{Eq: c1 AH}-\eqref{Eq: c5 AH} uniformly for small $\epsilon$ after slightly adjusting the constants therein.
\end{lemma}
As before, we consider the smooth inverse mean curvature flow 
$$\Phi_\epsilon:N\times [0,\tau)\to \mathbb R^3$$
with the initial surface $\Sigma_{t,\epsilon}$, and denote $\Sigma_{t,\epsilon,s}=\Phi_\epsilon(N,s)$ for short. Our goal is to show
\begin{lemma}\label{Lem: smooth IMCF AH}
For any constant $s_0>0$, the constant $\rho_1$ can be taken large enough such that for each $\Sigma_t^+$ with $t\in[T_s,T]$ there is a smooth inverse mean curvature flow $\Phi^+:N\times (0,s_0]\to \mathbb R^3$ such that
$$d_{\mathcal H}(\Phi^+(N,s),\Sigma_t^+)\to 0\mbox{ as }s\to 0^+,$$
and that $\mathcal X_p^{-1}(\Phi^+(N,s))$ encloses $\mathbb B_{\mathfrak r(t-1+s/2)}$.
\end{lemma}

Clearly, we have the following lemmas for surfaces $\Sigma_{t,\epsilon,s}$ with $s\leq s_0$.
\begin{lemma}\label{Lem: Sigma_s area bound AH}
The areas of surfaces $\Sigma_{t,\epsilon,s}$ are uniformly bounded from above by a universal constant $\tilde A_0$ independent of $p$, $t$, $\epsilon$ and $s$.
\end{lemma}
\begin{proof}
This follows from the exponential area growth, Lemma \ref{Lem: N_t^+ properties AH} and also \eqref{Eq: c3 AH}.
\end{proof}
\begin{lemma}\label{Lem: Sigma_s annulus}
For each surface $\Sigma_{t,\epsilon,s}$, $\mathcal X_p^{-1}(\Sigma_{t,\epsilon,s})$ lies in the annulus region $\mathbb B_{\mathfrak r(t+1+2s)}\setminus\bar{\mathbb B}_{\mathfrak r(t-1+s/2)}$.
\end{lemma}
\begin{proof}
This follows from Lemma \ref{Lem: N_t^+ properties AH}, the annulus neighborhood estimate \eqref{Eq: c2 AH}, and also the fact that functions $\mathfrak r(t/2)$ and $\mathfrak r(2t)$ induce a subsolution and a supersolution for inverse mean curvature flow around the infinity in the Poincar\'e model $(\mathbb B^3,\mathcal X_p^*\hat g)$, when $\rho_1$ is taken large enough.
\end{proof}
\begin{lemma}
The Hawking mass of $\Sigma_{t,\epsilon,s}$ satisfies
$$m^*_h(\Sigma_{t,\epsilon,s})\geq o(1)\mbox{ as }\rho_1\to+\infty.$$
\end{lemma}
\begin{proof}
This follows from Lemma \ref{Lem: N_t^+ properties AH}, the Hawking mass lower bound \eqref{Eq: c4 AH}, the Geroch monotonicity, and Lemma \ref{Lem: Sigma_s area bound AH}.
\end{proof}
Define
$$\tilde X=\frac{x}{|x|_{g_{poi}}}\mbox{ with }x\in \mathbb B_1.$$
\begin{lemma}
Given a positive constant $\iota_0$, we can take $\rho_1$ large enough and a universal constant $\tilde B_0$ independent of $p$, $t$, $\epsilon$ and $s$ such that if $\mathcal X_p^{-1}(\Sigma_{t,\epsilon,s})$ satisfies
\begin{equation}\label{Eq: starshape assumption AH}
\mathcal X_p^*\hat g(\tilde X,\nu(x))\geq \iota_0>0
\end{equation}
for all $s\in [0,\tau^*]$ with $\tau^*<\tau$, then the second fundamental form $\mathcal B_s$ of $\Sigma_{t,\epsilon,s}$ satisfies
$$\|\mathcal B_s\|\leq \tilde B_0$$
for all $s\in[0,\tau^*]$.
\end{lemma}
\begin{proof}
From Lemma \ref{Lem: Sigma_s annulus} we know that all surfaces $\mathcal X_p^{-1}(\Sigma_{t,\epsilon,s})$ are contained in a fixed compact subset. Then we can conclude from \cite[Theorem 3.1]{huisken2001inverse} and \cite[Theorem 5.1]{heidusch2001regularitat} that there is a positive constant $c_1$, which is independent of $p$, $t$, $\epsilon$ and $s$, such that after taking $\rho_1$ large enough we have
$$\|\mathcal B_s\|\leq c_1(1+\|\mathcal B_0\|).$$
The conclusion follows from Lemma \ref{Lem: N_t^+ properties AH} and also \eqref{Eq: c5 AH}.
\end{proof}
\begin{lemma}\label{Lem: star-shape to curvature AH}
 Given any positive constant $\iota_0$, we can take $\rho_1$ large enough and a universal constant $c$, independent of $p$, $t$, $\epsilon$ and $s$, such that under the assumption \eqref{Eq: starshape assumption AH} we have for all $s\in [0,\tau^*]$ that
\begin{equation}\label{Eq: mean curvature lower bound AH}
H_s\geq c\min\{1,s^{\frac{1}{2}}\}.
\end{equation}
\end{lemma}
\begin{proof}
According to the evolution equation (iv) in Lemma \ref{Lem: evolution equation IMCF}, we have
$$
\left(\frac{\partial}{\partial s}- H^{-2}\Delta\right) H^{-1}=\frac{\|{\mathcal B}\|^2}{ H^2} H^{-1}+\frac{\Ric(\nu)}{ H^2} H^{-1}.
$$
Define
$$v=\langle X,\nu\rangle\mbox{ where }X=\sinh|x-p|\cdot\frac{x-p}{|x-p|}.$$
As in the proof of Lemma \ref{Lem: star-shape to curvature} we can compute
$$
\left(\frac{\partial}{\partial s}- H^{-2}\Delta\right) v=\frac{\|{\mathcal B}\|^2}{ H^2} v+ H^{-2}\left(o(\cosh|x-p|)+o(|X|)+o(|X|\|{\mathcal B}\|)\right).
$$
From Lemma \ref{Lem: Sigma_s annulus} and the assumption \eqref{Eq: starshape assumption AH} we can write
$$\left(\frac{\partial}{\partial s}- H^{-2}\Delta\right) v=\frac{\|\mathcal B\|^2}{H^2}v+\frac{1+\|{\mathcal B}\|}{ H^2}v\cdot o(1).$$
Define
$$w=(Hv)^{-1}.$$
Then $w$ satisfies the evolution equation
$$\frac{\partial}{\partial s}w=\Div\left(H^{-2}\nabla w\right)-2H^{-2}w^{-1}|\nabla w|^2+H^{-2}w\left(\Ric(\nu)+o(1+\|{\mathcal B}\|)\right).$$
Since $\mathcal B$ is uniformly bounded and $\Ric(\nu)=-2+o(1)$, we have
$$\frac{\partial}{\partial s}w\leq\Div\left(H^{-2}\nabla w\right)-2H^{-2}w^{-1}|\nabla w|^2-H^{-2}w$$
when $\rho_1$ is taken to be large enough. 

Fix any two constants $s_1, s_2$ such that $0<s_1<s_2\leq \tau^*$ and $s_2-s_1\leq 1$. We let $s$ vary in $[s_1,s_2]$. Define
$$\hat w=(s-s_1)^{\frac{1}{2}}w\mbox{ and }\hat w_k=\max\{w-k,0\}.$$
Then $\hat w$ satisfies
$$
\frac{\partial}{\partial s}\hat w\leq\Div\left(H^{-2}\nabla \hat w\right)-2H^{-2}\hat w^{-1}|\nabla \hat w|^2-H^{-2}\hat w+\frac{1}{2}(s-s_1)^{-1} \hat w.
$$
From this we can compute
\[
\begin{split}
\frac{\partial}{\partial s}\int_{\Sigma_{t,\epsilon,s}}\hat w_k^2\,\mathrm d\sigma \leq& -2\int_{\Sigma_{t,\epsilon,s}} H^{-2}\hat w\cdot\hat w_k\,\mathrm d \sigma\\
&\quad+\frac{1}{(s-s_1)}\int_{\Sigma_{t,\epsilon,s}}\hat w\cdot\hat w_k\,\mathrm d \sigma+\int_{\Sigma_{t,\epsilon,s}}\hat w_k^2\,\mathrm d \sigma.
\end{split}
\]
From \eqref{Eq: c1 AH} and the assumption \eqref{Eq: starshape assumption AH} we have
\[
\begin{split}
\int_{\Sigma_{t,\epsilon,s}}H^{-2}\hat w\cdot\hat w_k\,\mathrm d \sigma&=(s-s_1)^{-1}\int_{\Sigma_{t,\epsilon,s}}v^2\hat w^3\cdot \hat w_k\,\mathrm d\sigma\\
&\geq (s-s_1)^{-1}\sinh^2\lambda_0 \iota_0^2k^2\int_{\Sigma_{t,\epsilon,s}}\hat w\cdot\hat w_k\,\mathrm d\sigma.
\end{split}\]
Take
$$k=\frac{1}{\sinh\lambda_0\iota_0}.$$
Then we obtain
$$\frac{\partial}{\partial s}\int_{\Sigma_{t,\epsilon,s}}\hat w_k^2\,\mathrm d\sigma \leq\int_{\Sigma_{t,\epsilon,s}}\hat w_k^2\,\mathrm d\sigma.$$
Notice that $\hat w=0$ when $s=s_1$, so we have
$$\int_{\Sigma_{t,\epsilon,s_1}}\hat w_k^2\,\mathrm d\sigma=0,\mbox{ which further yields }\int_{\Sigma_{t,\epsilon,s}}\hat w_k^2\,\mathrm d\sigma=0.$$
In particular, we have $\hat w\leq k$ on $\Sigma_{t,\epsilon,s_2}$, which implies
$$H_{s_2}\geq \sinh\lambda_0\iota_0v^{-1}(s_2-s_1).$$
From Lemma \ref{Lem: Sigma_s annulus} we know that all $\Sigma_{t,\epsilon,s}$ are contained in a compact subset, so $v$ has a uniform upper bound. With the same choice of $s_1$ as in the proof of Lemma \ref{Lem: star-shape to curvature}, we obtain the desired estimate.
\end{proof}
 \begin{lemma}\label{Lem: star-shape all time AH}
    There are two positive constants $\iota$ and $B$, independent of $p$, $t$, $\epsilon$ and $s$, such that the constant $\rho_1$ can be taken large enough to ensure that $\Sigma_{t,\epsilon,s}$ satisfies 
    $$\mathcal X_p^*\hat g(\tilde X,\nu(x))\geq \iota\mbox{ and }\|\mathcal B\|<B$$
 for all $s\in [0,s_0]$.
 \end{lemma}
 \begin{proof}
 The proof is similar to that of Lemma \ref{Lem: star-shape all time} and here we just omit the details.
 \end{proof}
 \begin{proof}[Proof of Lemma \ref{Lem: smooth IMCF AH}]
 The proof is similar to that of Lemma \ref{Lem: smooth IMCF}.
 \end{proof}
 \begin{lemma}\label{Lem: smooth solution AH}
The constant $\rho_1$ can be taken large enough such that for any point $p\in \mathbb R^3\setminus \bar D_{\rho_1}$ there is a smooth inverse mean curvature flow $\Phi: N\times (-\infty,T]\to \mathbb R^3$ such that $\Phi(N,t)=\Sigma_t=\Sigma_t^+$.
\end{lemma}
\begin{proof}
The proof is similar to that of Lemma \ref{Lem: smooth solution}.
\end{proof}
\begin{proof}[Proof of Theorem \ref{main-theorem-IMCF-AH}]
This follows directly from Lemma \ref{Lem: smooth solution AH} and the fact that $\Sigma_t^+$ with $-\infty<t\leq T$ are all topological spheres as well as  the annulus neighborhood estimate \eqref{Eq: annulus estimate AH}.
\end{proof}

 \section{Min-Max Theory}

\label{section-min-max}

In \cite{mazurowski2022prescribed}, the first named author proved the existence of closed constant mean curvature surfaces in asymptotically flat manifolds with no boundary which satisfy a certain technical condition. To begin this section, we recall the result of \cite{mazurowski2022prescribed}, and then we give the proof of our first main result Theorem \ref{main-theorem-flat}. Then we discuss the adaptation to the case of asymptotically hyperbolic manifolds.   

\subsection{Asymptotically Flat Manifolds} Let $(M^3,g)$ be an asymptotically flat manifold with no boundary. Let $\mathcal C(M)$ denote the space of Caccioppoli sets in $M$ with compact support. Given a constant $c > 0$, the functional $A^c\colon \mathcal C(M)\to \R$ is defined by 
\[
A^c(\Omega) = \M(\bd \Omega) - c\vol(\Omega). 
\]
A continuous family of Caccioppoli sets $\{\Omega_t\}_{t\in[0,1]}$ is called a {\it mountain pass path} for $A^c$ provided $\Omega_0 = \emptyset$ and $A^c(\Omega_1)<0$. Finally, we define the one-parameter min-max value 
\[
\omega_c(M) = \inf_{\{\Omega_t\}_{t\in [0,1]}} \left[\sup_{t\in [0,1]} A^c(\Omega_t)\right],
\]
where the infimum is taken over all mountain pass paths. Then we have the following result. 

\begin{theorem}[Theorem 1.2 in \cite{mazurowski2022prescribed}] 
\label{asy-flat-min-max} Let $(M^3,g)$ be an asymptotically flat manifold with no boundary. Fix a constant $c > 0$ and assume that $\omega_c(M) < \omega_c(\R^3)$. Then $M$ contains a closed, almost-embedded surface $\Sigma$ with constant mean curvature $c$. 
\end{theorem}

As noted in \cite[Remark 4.3]{mazurowski2022prescribed}, the proof can be adapted to handle the case of asymptotically flat manifolds with boundary by using the free boundary min-max theory of \cite{sun2024multiplicity}. We will explain carefully how this adaptation works below in the case of asymptotically hyperbolic manifolds and leave the details of the  asymptotically flat case to the reader. Thus we have the following theorem. 

\begin{theorem}
\label{min-max-flat-boundary}
Let $(M^3,g)$ be a complete, asymptotically flat manifold, possibly with boundary. Fix a constant $c > 0$ and assume that $\omega_c(M) < \omega_c(\R^3)$. Then there exists a compact, almost-embedded, free boundary hypersurface $\Sigma$ in $M$ with constant mean curvature $c$.    
\end{theorem}

As observed in \cite{mazurowski2022prescribed}, it is possible to verify the assumption $\omega_c(M) < \omega_c(\R^3)$ by constructing a continuous family of open sets with favorable isoperimetric ratio. We have shown above that this is always possible when $M$ has non-negative scalar curvature and has a non-Euclidean end. Hence we can now give the proof of our first main theorem. 

\begin{theorem}
Let $(M^3,g)$ be a complete, asymptotically flat manifold, possibly with non-empty boundary.   Assume that $M$ has non-negative scalar curvature. Then, for every constant $c > 0$, there exists a compact, almost-embedded, free boundary, constant mean curvature surface $\Sigma$ in $M$ with mean curvature $c$.   
\end{theorem}

\begin{proof}
    Let $(M^3,g)$ be as in the statement of the theorem and fix a constant $c > 0$. If the end of $M$ is Euclidean, then the result follows by taking $\Sigma$ to be a ball of radius $2/c$ in the Euclidean end. Therefore, we can assume that $M$ has a non-Euclidean end.  In this case, it follows from Proposition \ref{Prop: sweep-out} that for any $v > 0$ we can find a continuous family $\{\Omega_t\}_{t\in [0,1]}$ with $\Omega_0 = \emptyset$ and $\vol(\Omega_1) = v$ and 
    \[
    \frac{\operatorname{Area}(\bd \Omega_t)}{\vol(\Omega_t)^{2/3}} < (36\pi)^{1/3}
    \]
    for all $t\in [0,1]$. Assuming $v$ is large enough, it follows from Proposition 4.29 in \cite{mazurowski2022prescribed} that this family is a mountain pass path with 
    \[
    \sup_{t\in [0,1]} A^c(\Omega_t) < \frac{4\pi}{3}\left(\frac 2 c\right)^2 = \omega_c(\R^3). 
    \]
    Therefore we have $\omega_c(M)<\omega_c(\R^3)$ and the result now follows by Theorem \ref{min-max-flat-boundary}.
\end{proof}

\subsection{Asymptotically Hyperbolic Manifolds} In the remainder of this section, we will explain how to prove the analog of Theorems \ref{asy-flat-min-max} and \ref{min-max-flat-boundary} in the asympotically hyperbolic setting. The main idea is the same as in \cite{mazurowski2022prescribed}. We first define a family of cut-off min-max problems which admit barriers near infinity. Then we show that the solution to these cut-off problems cannot drift to infinity. In this step, we need to argue somewhat differently from \cite{mazurowski2022prescribed}, as it does not seem straightforward to prove a uniform upper bound on the area of the approximate solutions. 

\subsubsection{Definitions}

We now establish the basic set-up for the min-max problem in an asymptotically hyperbolic manifold. 

\begin{definition}
A Riemannian manifold $(M^3,g)$ is called asymptotically hyperbolic if there is a compact set $K\subset M$, and a diffeomorphism $M\setminus K \equiv \R^3 \setminus B_1(0)$, and in the polar coordinates induced by this diffeomorphism the metric $g$ takes the form 
\[
g = dr^2 + \sinh^2 r\, g_{S^2} + Q
\]
where $\vert Q\vert + \vert \del Q\vert + \vert \del^2 Q\vert \le C e^{-3r}$.  Here $g_{S^2}$ denotes the round metric on $S^2$ and the norms and derivatives of $Q$ are taken with respect to the hyperbolic metric $\overline g = dr^2 + \sinh^2 r\, g_{S^2}$.  We further require that $Q\to 0$ smoothly as $\vert x\vert \to \infty$, although with no assumption on the rate. 
\end{definition}

Fix an asymptotically hyperbolic manifold $(M^3,g)$ and a constant $c > 2$. 

\begin{definition}
Let $\mathcal C(M)$ denote the set of Caccioppoli sets in $M$ with compact support. The flat distance on $\mathcal C(M)$ is defined by $\mathcal F(\Omega_1,\Omega_2) = \vol(\Omega_1 \Delta \Omega_2)$.  We will also need the $\mathbf F$-topology on $\mathcal C(M)$ defined in \cite{zhou2020existence}.  Given $\Omega\in \mathcal C(M)$, let $\bd \Omega$ denote the reduced boundary of $\Omega$ restricted to the interior of $M$. In particular, we have $\bd \Omega \llcorner \bd M = \emptyset$. 
\end{definition}

\begin{definition}
Let $h\colon M\to \R$ be a smooth function.  The $A^h$ functional on $\mathcal C(M)$ is 
\[
A^h(\Omega) = \M(\bd \Omega) -  \int_\Omega h.
\] 
In the case where $h\equiv c$ is constant, this reduces to $A^c(\Omega) = \M(\bd \Omega) - c \vol(\Omega)$. 
\end{definition}

\begin{definition} 
\label{min-max-value}
A mountain pass path for $A^c$ is an $\mathbf F$-continuous family $\{\Omega_t\}_{t\in [0,1]}$ in $\mathcal C(M)$ such that $\Omega_0 = \emptyset$ and $A^c(\Omega_1) < 0$.  Let $\mathcal P_c$ denote the set of all mountain pass paths. The $A^c$ min-max value is defined by
\[
\omega_c(M) = \inf_{\{\Omega_t\}\in \mathcal P_c}\left(\sup_{t\in [0,1]} A^c(\Omega_t)\right).
\]
\end{definition} 

To continue, we will require some elementary facts about hyperbolic geometry.  In what follows, we adopt the convention that objects decorated with a bar are associated to hyperbolic space. Thus $\bh_r$ denotes a hyperbolic ball of radius $r$, $\overline g$ denotes the hyperbolic metric, and so on.  

\begin{proposition}
The area, enclosed volume, and mean curvature of $\bd \bh_r$ are given by 
\begin{gather*}
\areah(\bd \bh_r) = 4\pi \sinh^2 r, \\ \volh(\bh_r) = \pi(\sinh(2r)-2r), \\ \overline H(\bd \bh_r) = 2\coth r. 
\end{gather*}
\end{proposition}

\begin{proof}
This is an elementary calculation. 
\end{proof} 

\begin{proposition}
Fix a constant $c > 2$.  The function $r\mapsto \overline{A^c}(\bh_r)$ is concave on $[0,\infty)$.  It attains a unique maximum $\alpha_c = \overline{A^c}(\bh_{\operatorname{arctanh(2/c)}})$. 
\end{proposition}

\begin{proof}
This follows easily from the previous proposition. 
\end{proof} 

\begin{proposition}
The min-max value of hyperbolic space is given by $\omega_c(\mathbb H^3) = \alpha_c$. 
\end{proposition} 

\begin{proof}
Consider $(\mathbb H^3,\overline g)$. Fix a very large number $R$. Then the path $t \in [0,1] \mapsto \overline B_{Rt}$ is a mountain pass path and 
\[
\sup_{t\in [0,1]} \overline {A^c}(\overline B_{Rt}) = \alpha_c
\]
by the previous proposition. This proves that $\omega_c(\mathbb H^3) \le \alpha_c$. On the other hand, it is well-known that geodesic balls in $\mathbb H^3$ are isoperimetric. Therefore, given any Caccioppoli set $\Omega$ in $(\mathbb H^3,\overline g)$, one has $\overline{A^c}(\Omega) \ge \overline{A^c}(\overline B)$ where $\overline B$ is a hyperbolic ball with the same volume as $\Omega$. This implies that $\omega_c(\mathbb H^3) \ge \alpha_c$. 
\end{proof}

\begin{proposition}
Let $(M^3,g)$ be an asymptotically hyperbolic manifold and fix a constant $c > 2$. The min-max value $\omega_c(M)$ satisfies $0 < \omega_c(M) \le \omega_c(\mathbb H^3)$. 
\end{proposition}

\begin{proof}
The fact that $\omega_c(M) > 0$ is a consequence of the isoperimetric inequality.  Since $M$ is asymptotically hyperbolic, there are constants $C>0$ and $v > 0$ such that 
\[
\area(\bd \Omega) \ge C \vol(\Omega)^{2/3}
\]
for all $\Omega\in \mathcal C(M)$ with $\vol(\Omega) \le v$.  Let $\{\Omega_t\}_{t\in[0,1]}$ be a mountain pass path.  Choose a small number $v_0 \le v$.  Since $\Omega_t$ is  continuous, there is a $t_0$ such that $\vol(\Omega_{t_0}) = v_0$. Thus 
\[
A^c(\Omega_{t_0}) = \area(\bd \Omega_{t_0}) - c \vol(\Omega_{t_0}) \ge C v_0^{2/3} - cv_0.
\]
This is strictly positive provided $v_0$ is chosen small enough. This proves that $\omega_c(M) > 0$. 

To show $\omega_c(M) \le \omega_c(\mathbb H^3)$, we construct explicit mountain pass paths near infinity.  Since $M$ is asymptotically hyperbolic, we can identify $M\setminus K$ with $\R^3 \setminus B$. Given a point $p\in \R^3\setminus B$, let $B_r(p)$ denote the ball of radius $r$ centered at $p$.  Assuming $\|p\|$ and $R$ are chosen sufficiently large, the path $t\in [0,1]\mapsto  B_{Rt}(p)$ is a mountain pass path. Moreover, since $g$ is asymptotically hyperbolic, it follows that 
\[
\sup_{t\in [0,1]} A^c(B_{Rt}(p)) \to \omega_c(\mathbb H^3)
\]
as $\|p\|\to \infty$. This proves that $\omega_c \le \omega_c(\mathbb H^3)$. 
\end{proof}

We now want to prove Theorem \ref{asy-hyp-min-max}, which we restate below for convenience. 

\begin{theorem}
\label{min-max-hyp}
    Let $(M^3,g)$ be a complete, asymptotically hyperbolic manifold, possibly with boundary. Fix a constant $c > 2$ and assume that $\omega_c(M) < \omega_c(\mathbb H^3)$. Then there is a compact, almost-embedded, free boundary hypersurface $\Sigma$ in $M$ with constant mean curvature $c$.
\end{theorem}

The crucial assumption required for Theorem \ref{min-max-hyp} is that $\omega_c(M) < \omega_c(\mathbb H^3)$. We will assume this throughout the remainder of this section. 

\subsubsection{The Approximate Problems} 

We now define and solve a sequence of approximate problems that admit barriers near infinity. First we need to study the mean curvature of large coordinate spheres in asymptotically hyperbolic manifolds.  Fix an asymptotically hyperbolic manifold $(M,g)$. Let $\R^3\setminus B$ be the asymptotically hyperbolic end of $M$ and let $S_r$ be the sphere of Euclidean radius $r$ centered at the origin in $\R^3\setminus B$. 

\begin{proposition}
\label{mcs} 
The mean curvature of $S_r$ satisfies $H = 2 + 4e^{-2r} + O(e^{-3r})$ as $r\to \infty$. In particular, the mean curvature of $S_r$ is at least $2(1+e^{-r})$ when $r$ is large enough. 
\end{proposition}

\begin{proof}
This is a straightforward calculation. 
\end{proof}

\begin{definition}
Fix a constant $c > 2$. Let $\zeta\colon \R\to \R$ be a smooth, decreasing function such that $\zeta(r) \equiv c$ for $r \le 0$ and $\zeta(r)\equiv 2$ for $r \ge 1$. 
\end{definition}

\begin{definition}
Let $B_R$ denote the compact piece of $M$ enclosed by $S_R$.  For $R$ sufficiently large, define a function $\zeta_R\colon M\to \R$ by setting $\zeta_R \equiv c$ on $B_R$ and setting $\zeta_R(x) = \zeta( \|x\|-R)$ on the asymptotically hyperbolic end $\R^3 \setminus B$. 
\end{definition}

Next we prove that the spheres $S_r$ serve as a barrier to the $A^{\zeta_R}$ functional when $r$ is sufficiently large. This allows us to confine the min-max problem for $A^{\zeta_R}$ to the compact region $B_{R+1}$.  

\begin{proposition}
Let $N$ be the $g$-unit normal vector to the spheres $S_r$.  Then $\Div(N) \ge 2(1+e^{-r})$ when $r$ is sufficiently large.
\end{proposition}

\begin{proof}
The divergence of $N$ is equal to the mean curvature of $S_r$. Hence this follows from Proposition \ref{mcs}. 
\end{proof}

\begin{definition}
For each sufficiently large $R > 0$ select a number $\eps_R$ so that $\zeta(1-6\eps_R) \le 2+2e^{-R-1}$.  This choice ensures that $\zeta_R \le \operatorname{div}(N)$ on $M\setminus B_{R+1-6\eps_R}$.  
\end{definition}

\begin{proposition}
\label{barrier} 
Fix a large number $R > 0$ and let $\Omega \in \mathcal C(M)$. Then $\Omega \cap B_{R+1-6\eps_R}$ also belongs to $\mathcal C(M)$ and $A^{\zeta_R}(\Omega\cap B_{R+1-6\eps_R})\le A^{\zeta_R}(\Omega)$. 
\end{proposition} 

\begin{proof}
The intersection of two Caccioppoli sets is a Caccioppoli set.  Thus $\Omega\cap B_{R+1-6\eps_R}$ and $\Omega \setminus B_{R+1-6\eps_R}$ are both Caccioppoli sets.  Note that $\operatorname{div}(N) \ge  \zeta_R$ in the set $M\setminus B_{R+1-6\eps_R}$. Hence the result follows by the divergence theorem. 
\end{proof} 

Next we set up the approximate problems.  Consider a large number $R > 0$. 

\begin{definition}
An $R$-mountain pass path is an $\mathbf F$-continuous family $\{\Omega_t\}_{t\in[0,1]}$ in $\mathcal C(M)$ such that $\Omega_0=\emptyset$ and $A^{\zeta_R}(\Omega_1) < 0$. Let $\mathcal P_{R}$ denote the set of all $R$-mountain pass paths. The $R$-min-max value is 
\[
\omega_R = \inf_{\{\Omega_t\}\in \mathcal P_R}\left(\sup_{t\in[0,1]} A^{\zeta_R}(\Omega_t)\right).
\]
\end{definition}

We will now prove some elementary properties of the $R$-min-max values $\omega_R$. 

\begin{proposition}
\label{mono}
The function $R\mapsto \omega_R$ is decreasing. 
\end{proposition} 

\begin{proof}
Fix some $R_1 < R_2$ and let $\eta > 0$.  Choose an $R_1$-mountain pass path $\{\Omega_t\}_{t\in[0,1]}$ such that 
\[
\sup_{t\in[0,1]} A^{\zeta_{R_1}}(\Omega_t) \le \omega_{R_1}+\eta. 
\]
Note that $\zeta_{R_2} \ge \zeta_{R_1}$ and therefore $A^{\zeta_{R_2}}(\Omega) \le A^{\zeta_{R_1}}(\Omega)$ for all $\Omega\in \mathcal C(M)$.  It follows that 
\[
\sup_{t\in[0,1]} A^{\zeta_{R_2}}(\Omega_t) \le \sup_{t\in[0,1]} A^{\zeta_{R_1}}(\Omega_t)\le  \omega_{R_1} + \eta.
\]
Moreover, $\Omega_0 = \emptyset$ and $A^{\zeta_{R_2}}(\Omega_1) \le A^{\zeta_{R_1}}(\Omega_1) < 0$. Thus $\{\Omega_t\}_{t\in[0,1]}$ is also an $R_2$-mountain pass path and so 
\[
\omega_{R_2} \le \sup_{t\in[0,1]} A^{\zeta_{R_2}}(\Omega_t) \le \omega_{R_1} + \eta.
\] 
Since $\eta$ was arbitrary, it follows that $\omega_{R_2}\le \omega_{R_1}$, as required.
\end{proof}

\begin{proposition}
\label{limit}
We have $\omega_R \to \omega_c(M)$ as $R\to \infty$. 
\end{proposition}

\begin{proof}
Note that $c \ge \zeta_R$ for all $R$. Therefore, the same argument as in the previous proof shows that $\omega_c(M) \le \omega_R$ for all $R$.  Since $\omega_R$ is decreasing, it therefore suffices to prove the following: for every $\eta > 0$ there is an $R$ such that $\omega_R \le \omega_c(M) + 2\eta$.  

Fix some $\eta > 0$. By definition, there exists a path $\{\Omega_t\}_{t\in [0,1]} \in \mathcal P$ such that 
\[
\sup_{t\in[0,1]} A^c(\Omega_t) \le \omega_c(M) + \eta. 
\]
We claim that there is a very large number $r > 0$ such that $\vol(\Omega_t\setminus B_r) < \eta c\inv$ for all $t\in [0,1]$. Suppose to the contrary that this is not the case. Then there is a sequence $r_k\to \infty$ and a sequence $t_k\in [0,1]$ such that $\vol(\Omega_{t_k}\setminus B_{r_k}) \ge \eta c\inv$ for all $k$. After passing to a subsequence, we can suppose that $t_k \to t_\infty \in [0,1]$. Then by continuity, we have $\Omega_{t_k} \to \Omega_{t_\infty}$. But $\Omega_{t_\infty}$ is contained in $B_{r_\infty}$ for some $r_\infty > 0$. It follows that $$\vol(\Omega_{t_\infty} \operatorname{\Delta} \Omega_{t_k}) \ge \vol(\Omega_{t_k} \setminus B_{r_k}) \ge \eta c\inv$$ for $k$ sufficiently large, and this is a contradiction. 

Now select some $R \ge r$.  Observe that 
\begin{align*}
A^{\zeta_R}(\Omega_t) &=  A^c(\Omega_t)-\int_{\Omega_t}(\zeta_R - c) \\ &= A^c(\Omega_t) - \int_{\Omega_t\cap B_R} (\zeta_R-c) - \int_{\Omega_t\setminus B_R} (\zeta_R - c) \\
&\le A^c(\Omega_t) + c\vol(\Omega_t\setminus B_R)
\end{align*}
and so $A^{\zeta_R}(\Omega_t) \le A^c(\Omega_t) + \eta$ for all $t\in [0,1]$. By choosing $R$ larger if necessary, we can ensure that $\Omega_1 \subset B_R$ and hence that $A^{\zeta_R}(\Omega_1) = A^c(\Omega_1) < 0$.  For this $R$, the path $\{\Omega_t\}_{t\in[0,1]}$ is an $R$-mountain pass path and so 
\[
\omega_R \le \sup_{t\in[0,1]} A^{\zeta_R}(\Omega_t) \le \sup_{t\in[0,1]}A^c(\Omega_t)+\eta \le \omega_c(M) + 2\eta.
\]
This proves the proposition. 
\end{proof}

\begin{definition}
Let $\mathcal Q_R$ be the set of all $R$-mountain pass paths $\{\Omega_t\}_{t\in[0,1]}$ such that $\Omega_t$ is supported in $B_{R+1-5\eps_R}$ for all $t\in[0,1]$. 
\end{definition}

\begin{proposition}
\label{confine} 
The $R$-min-max value does not change if we restrict to $R$-mountain pass paths with support in $B_{R+1-5\eps_R}$. In other words,
\[
\omega_R = \inf_{\{\Omega_t\}\in \mathcal Q_R} \left(\sup_{t\in[0,1]} A^{\zeta_R}(\Omega_t)\right).
\]
\end{proposition}

\begin{proof} Since $\mathcal Q_R\subset \mathcal P_R$, it follows that 
\[
\omega_R \le \inf_{\{\Omega_t\}\in \mathcal Q_R} \left(\sup_{t\in[0,1]} A^{\zeta_R}(\Omega_t)\right).
\]
To prove the reverse inequality, fix a small $\eta > 0$. By definition, there is an $R$-mountain pass path $\{\Omega_t\}_{t\in[0,1]} \in \mathcal P_R$ such that 
\[
\sup_{t\in[0,1]} A^{\zeta_R}(\Omega_t) \le \omega_R + \eta. 
\]
Define $\Theta_t = \Omega_t \cap B_{R+1-6\eps_R}$ for $t\in[0,1]$.  By Proposition \ref{barrier}, we have $A^{\zeta_R}(\Theta_t)\le A^{\zeta_R}(\Omega_t)$ for all $t\in[0,1]$.  

According to \cite{dey2023existence}, the family $\{\Theta_t\}_{t\in[0,1]}$ is flat continuous with no concentration of mass.  Therefore, by an interpolation theorem of Zhou \cite{zhou2020multiplicity}, there exists an $\mathbf F$-continuous family such that 
\begin{itemize}
\item[(i)] $\Xi_0 = \Theta_0$ and $\Xi_1 = \Theta_1$,
\item[(ii)] $\sup_{t\in[0,1]} A^{\zeta_R}(\Xi_t) \le \sup_{t\in[0,1]} A^{\zeta_R}(\Theta_t)+\eta$,
\item[(iii)] $\Xi_t$ is supported in $B_{R+1-5\eps_R}$ for all $t\in[0,1]$. 
\end{itemize}
Thus $\{\Xi_t\}_{t\in[0,1]}$ belongs to $\mathcal Q_R$ and satisfies 
\[
\sup_{t\in[0,1]} A^{\zeta_R}(\Xi_t) \le \omega_R + 2\eta.
\]
This proves that 
\[
 \inf_{\{\Omega_t\}\in \mathcal Q_R} \left(\sup_{t\in[0,1]} A^{\zeta_R}(\Omega_t)\right) \le \omega_R + 2\eta.
\]
The result follows since $\eta$ was arbitrary. 
\end{proof}

\subsubsection{Solutions to the Approximate Problems}

We are now set up to apply the min-max theory developed in \cite{sun2024multiplicity}. The goal is to prove the following proposition.

\begin{proposition}
Fix a large number $R > 0$. There exists a smooth, compact, almost-embedded, free boundary surface $\Sigma_R = \bd \Omega_R$ in $M$ with mean curvature prescribed by $\zeta_R$. Moreover, $\Sigma_R$ is contained in $B_{R+1}$, and $\bd \Sigma_R \cap \bd B_{R+1} = \emptyset$, and $A^{\zeta_R}(\Omega_R) = \omega_R$, and the index of $\Omega_R$ as a critical point of $A^{\zeta_R}$ is at most one. 
\end{proposition}

\begin{proof}
This essentially follows from Theorem 3.11 in \cite{sun2024multiplicity}. We now discuss the minor modifications that are required in the proof.  
We let $(N,\bd N,g) = (B_{R+1-\eps_R},\bd M\cup S_{R+1-\eps_R},g)$.  Let $h = \zeta_R$. Note that although $h$ does not belong to the class $\mathcal S(g)$ defined in \cite{sun2024multiplicity}, all the constructions in \cite{sun2024multiplicity} still work with $h$ since $h$ is positive.  We can suppose that $(N,\bd N,g)$ is isometrically embedded in a closed manifold $(\widetilde N^3,\widetilde g)$. Let $(\widetilde g_k,\widetilde h)$ be a sequence of good pairs on $\widetilde N$ such that $\widetilde g_k$ converges smoothly to $\widetilde g$ and $\widetilde h$ is positive. Let $g_k = \widetilde g_k|_N$.

Choose a sequence of mountain pass paths $\{\Omega^k_t\}$ which are supported in $B_{R+1-5\eps_R}$ for which 
\[
\sup_{t\in [0,1]} A^h(\Omega^k_t,g) \to \omega_R
\]
as $k\to \infty$.  Now for each $k\in \N$, let $\Pi_k$ be the collection of all $\mathbf F$-continuous maps $\Phi\colon [0,1]\to \mathcal C(N)$ such that 
\begin{itemize}
\item[(i)] $\Phi(0) = \emptyset$,
\item[(ii)] $\Phi(1) = \Omega^k_1$, 
\item[(iii)] $\supp(\Phi(t)) \cap S_{R+1-\eps_R} = \emptyset$ for all $t\in[0,1]$. 
\end{itemize}
Observe that the class $\Pi_k$ is non-empty since it contains $\{\Omega^k_t\}$. 
For each $k\in \N$, consider 
\[
\alpha_{k} = \inf_{\Phi\in \Pi_k} \left[\sup_{t\in[0,1]} A^{h}(\Phi(t),g_k)\right]. 
\]
We claim that $\lim_{k\to \infty} \alpha_k = \omega_R$. Note that there are constants $\eta_k > 1$ such that $\eta_k \to 1$ as $k\to \infty$ and 
\begin{gather*}
\eta_k^{-2} \area(\bd \Omega,g) \le \area(\bd \Omega, g_k) \le \eta_k^2 \area(\bd \Omega,g),\\
\eta_k^{-3} \vol(\Omega,g) \le \vol(\Omega,g_k) \le \eta_k^3 \vol(\Omega,g).  
\end{gather*}
Now observe that 
\[
\alpha_k \le \sup_{t\in [0,1]} A^h(\Omega^k_t,g_k).
\]
We know that $\sup_{t\in [0,1]} A^h(\Omega^k_t,g)\to \omega_R$ as $k\to \infty$. It follows that 
\[
\area(\bd \Omega^k_t,g) \le \omega_R + c \vol(B_{R+1}) + 1
\]
for all $t\in [0,1]$ provided $k$ is sufficiently large. Since the area and volume of $\Omega^k_t$ with respect to $g$ are uniformly bounded for all $t\in [0,1]$ and all $k$ sufficiently large, it follows that $\sup_{t\in [0,1]} A^h(\Omega^k_t,g_k) \to \omega_R$ as $k\to \infty$. Hence we have $\lim_{k\to \infty} \alpha_k \le \omega_R$. 

Now suppose for contradiction that $\lim_{k\to\infty} \alpha_k < \omega_R - 3\eps$ for some $\eps > 0$. Then for each large enough $k\in \N$ there exists an $\mathbf F$-continuous map $\Phi_k\colon [0,1]\to \mathcal C(N)$ such that $\Phi(0) = \emptyset$ and $\Phi(1) = \Omega^k_1$ and $\supp(\Phi_k(t)) \cap S_{R+1-\eps_R}=\emptyset$ for all $t\in [0,1]$ and 
\[
\sup_{t\in[0,1]} A^h(\Phi_k(t),g_k) \le \omega_R - 2\eps. 
\]
Note that $\area(\bd \Phi_k(t),g_k)$ and $\vol(\Phi_k(t),g_k)$ are uniformly bounded. Therefore, we have 
\[
\sup_{t\in [0,1]} A^h(\Phi_k(t),g) \le \omega_R - \eps
\]
provided $k$ is large enough. However, $\Phi_k$ is a mountain pass path in $M$, and so this is a contradiction.

We will apply min-max to find a compact surface $\Sigma_k = \bd \Omega_k$ in $(N,\bd N,g_k)$ with prescribed mean curvature $h$, index at most one, $A^{h}(\Omega_k,g_k) = \alpha_k$, and $\supp(\Omega_k) \cap S_{R+1-\eps_R} = \emptyset$. 
Assume $k\in \N$ is very large so that $\tilde g_k$ is very close to $\tilde g$. First, note that there is a min-max sequence $\{\Phi_j^*\}_{j\in \N}$ for $\Pi_k$ with the property that $\supp(\Phi_j^*(t)) \subset B_{R+1-4\eps_R}$ for all $t\in [0,1]$ and all $j\in \N$. Indeed, this follows exactly as in the proof of Proposition \ref{confine}, using the fact that the spheres $S_r$ with $r \ge R+1-5\eps_R$ serve as barriers for $A^h$ with respect to $g_k$ provided $k$ is large enough. 

Next we follow Step A in \cite[Theorem 3.11]{sun2024multiplicity} to construct a pulled-tight min-max sequence $\{\Phi_j\}_{j\in \N}$. Observe that since $\Phi_j$ is obtained from $\Phi_j^*$ by following the flow of $C^1$ isotopies, we can further ensure that $\supp(\Phi_j(t))\subset B_{R+1-3\eps_R}$ for all $t\in [0,1]$ and $j\in \N$. 
We then follow the combinatorial argument as in Part B with no changes. Here, note that if no element of the critical set is almost-minimizing in annuli, then the new maps produced by the combinatorial argument can be chosen to have support in $B_{R+1-2\eps_R}$. Hence the new maps still belong to the class $\Pi_k$, and it follows that some element of the critical set must be almost minimizing in annuli. The regularity argument in Step C goes through with no changes, noting as above that $h$ is positive. The argument in Step D also goes through essentially unchanged, noting that because all the maps $\Phi_j$ are supported in $B_{R+1-3\eps_R}$, the maps produced by the deformation theorem can be taken to have support in $B_{R+1-2\eps_R}$.  

Combining steps A-D, it follows that there exists a compact, almost-embedded surface $\Sigma_k = \bd \Omega_k$ in $(N,\bd N,g_k)$ with prescribed mean curvature $h$. The surface $\Sigma_k$ may have free boundary on $\bd M$ but is disjoint from $S_{R+1-\eps_R}$. Moreover, we have $A^h(\Omega_k,g_k) = \alpha_k$ and $\ind_{A^h}(\Omega_k,g_k) \le 1$. Finally, by the curvature estimates for prescribed mean curvature surfaces, we know that $\Sigma_k = \bd \Omega_k$ converges locally smoothly with multiplicity one away from at most one point to a limit $\Sigma_R=\bd \Omega_R$ in $(N,\bd N,g)$. This limit satisfies $A^h(\Omega_R) = \omega_R$ and $\ind_{A^h}(\Omega_R) \le 1$ and $\supp(\Omega_R)\subset B_{R+1-\eps_R}$.  Therefore, $\Omega_R$ is as required. 
\end{proof}

\subsubsection{Behavior at Infinity}

The goal of this section is to prove that some connected component of $\Sigma_R = \bd \Omega_R$ must be contained in $B_{R}$ when $R$ is sufficiently large. This will complete the proof of Theorem \ref{min-max-hyp}.  For the rest of this section, suppose to the contrary that for every large $R > 0$,  there is no component of $\Sigma_R = \bd \Omega_R$ contained in $B_{R}$. 

\begin{proposition}
There is a constant $d > 0$ which does not depend on $R$ such that $\Sigma_R \subset B_{R+1}\setminus B_{R-d}$.  
\end{proposition}

\begin{proof}
Let $(\Gamma,\bd \Gamma)$ be a stable constant mean curvature surface with mean curvature $c > 2$ in hyperbolic space. De Lima \cite[Corollary 2]{de2003surfaces} proved that 
\[
\dist(p,\bd \Gamma) \le \frac{4}{3}\frac{\pi}{\sqrt{c^2 - 4}}
\]
for all $p\in \Gamma$. 
 Since $M$ is asymptotically hyperbolic, it follows that there is a constant $\rho = \rho(c) > 0$ such that any stable constant mean curvature $c > 2$ surface $(\Gamma,\bd \Gamma)$ contained far out in the end of $M$ satisfies $\dist(p,\bd \Gamma) \le \rho$ for all $p\in \Gamma$. 
 
Let $d = 8\rho$.  Let $\Gamma_R$ be a connected component of $\Sigma_R \cap B_R$. Then $\Gamma_R$ has constant mean curvature $c$ and boundary contained in $S_R$. Moreover, the index of $\Gamma_R$ for the $A^c$ functional is at most one.  Suppose for contradiction that $\Gamma_R$ intersects $B_{R-d}$. Then $\Gamma_R$ intersects every sphere $S_r$ with $R-d\le r \le R$. Let $\Gamma_R^1 = \Gamma_R \cap (B_R \setminus B_{R-4\rho})$ and let $\Gamma_R^2 = \Gamma_R \cap (B_{R-4\rho}\setminus B_{R-8\rho})$. Since $\Gamma_R$ has index at most one, it follows that $\Gamma_R^i$ must be stable for some choice of $i=1,2$. This implies that $\dist(p,\bd \Gamma_R^i) \le \rho$ for all $p\in \Gamma_R^i$. However, it is clear that there is a point $p\in \Gamma_R^i$ with $\dist(p,\bd \Gamma_R^i) \ge 2\rho$. This is a contradiction, and the result follows. 
\end{proof}

Next we analyze the pointed limits of the surfaces $\Sigma_R$ as $R\to \infty$. Choose a sequence $R_j\to \infty$ and let $p_j$ be any point in $S_{R_j}$.  Let $x,y$ be geodesic normal coordinates for $S_{R_j}$ in a neighborhood of $p_j$.  Let $\nu$ be the {\it inward} pointing normal to $S_{R_j}$. For any fixed $r > 0$ and $\eps > 0$ we can define a map 
\begin{gather*}
\phi_j\colon \{(x,y): x^2+y^2 < r\} \times (\eps,\eps\inv) \to M,\\
\phi_j(x,y,z) = \exp_{(x,y)}\left(\left[\int_{1}^z s\inv \, ds\right] \nu\right).
\end{gather*}
Then $\phi_j$ is a diffeomorphism onto a neighborhood of $p_j$ for $j$ large enough. Moreover, 
\[
\phi_j^* g \to \frac{dx^2 + dy^2 + dz^2}{z^2}
\]
smoothly as $j\to \infty$. The image of $S_{R_j}$ is the slice $\{z=1\}$ and the image of $S_{R_j+1}$ is the slice $\{z=e\inv\}$. 

Note that $M$ has uniform geometry at infinity, and that the functions $\zeta_R$ have uniformly bounded $C^3$ norm. Hence there is a uniform curvature estimate for stable surfaces with prescribed mean curvature $\zeta_R$.  Since the ambient dimension is 3, this curvature estimate does not depend on an area upper bound for the surfaces.   As the surfaces $\Sigma_{R_j}$ have index at most 1, after passing to a subsequence, we can use the above construction to take a pointed limit of $(\Sigma_{R_j},p_j)$. More precisely, by a standard argument, we can find points $q_j \in M$ so that 
\[
\vert A_{\Sigma_{R_j}}(x)\vert \min\{1, \dist(x,q_j)\} \le C
\]
for a universal constant $C$. Using the maps $\phi_j$ as above on a sequence of larger and larger domains, in view of the curvature estimates, we can pass the surfaces $\phi_j^{-1}(\Sigma_{R_j})$ to a limit in $\mathbb H^3 \setminus \{q\}$. Here, by $\mathbb H^3$ we mean the upper half space model and $q$ is the limit of $\phi_j^{-1}(q_j)$.

This limit is a weak lamination $\mathcal L$ of $\mathbb H^3\setminus\{q\}$.  Each leaf $\Lambda\in \mathcal L$ is almost-embedded with mean curvature prescribed by the function 
\[
h(x,y,z) = \zeta\left(-\int_1^z s\inv\, ds\right). 
\]
Moreover, $\mathcal L$ is contained in a strip $\{e\inv \le z \le e^d\}$. The surfaces $\Sigma_{R_j}$ converge to $\mathcal L$ locally smoothly away from $q$. 

\begin{proposition} Each leaf in $\mathcal L$ is complete and proper away from $\{q\}$. 
\end{proposition} 

\begin{proof}
Let $\Lambda$ be a leaf in $\mathcal L$. It is a standard fact that $\Lambda$ is complete away from $q$. Suppose for contradiction that $\Lambda$ is not proper in $\mathbb H^3\setminus \{q\}$. Then there is a sequence of points $x_k\in \Lambda$ such that $x_k$ diverges to infinity in the intrinsic Riemannian topology on $\Lambda$ but $x_k \to x\in \mathbb H^3\setminus \{q\}$. Note that $\Lambda$ has uniform curvature estimates away from $q$. Hence there is an $r > 0$ such that the connected component $\Gamma_k$ of $\Lambda \cap B_r(x_k)$ passing through $x_k$ is uniformly graphical over $T_{x_k}\Lambda$. After passing to a subsequence, the graphs $\Gamma_k$ converge to a limiting graph $\Gamma$ with prescribed mean curvature $h$ passing through $x$. 

However, because each $\Sigma_{R_j}$ is a boundary, between any two sheets $\Gamma_k$ and $\Gamma_{k+1}$, there must be another sheet $\Gamma_k'$ in $\mathcal L$
whose mean curvature points in the opposite direction.  Passing to a further subsequence, we can suppose that $\Gamma_k'$ converges to a limiting graph $\Gamma'$ which has prescribed mean curvature $h$ with the opposite orientation. Since $\Gamma'$ must lie to the mean convex side of $\Gamma$, this contradicts the maximum principle. 
\end{proof} 

\begin{proposition}
Each leaf in $\mathcal L$ extends smoothly across $q$. 
\end{proposition}

\begin{proof}
Consider a leaf $\Lambda$ in $\mathcal L$. Since $\Lambda$ is complete and proper away from $\{q\}$, it follows that $\Lambda$ defines a locally rectifiable current in $\mathbb H^3\setminus \{q\}$. Since the mean curvature of $\Lambda$ is bounded, a result of Harvey and Lawson \cite[Theorem 3.1]{harvey1975extending} implies that $\Lambda$ has finite mass in a neighborhood of $q$.  It is now standard that the singularity at $q$ can be removed. 
\end{proof}

We now know that all the leaves in $\mathcal L$ are complete, smooth, and proper, even across $q$. Next, we classify the possible leaves that can appear. 

\begin{proposition}
Every compact leaf in $\mathcal L$ is a sphere with constant mean curvature $c$. 
\end{proposition}

\begin{proof}
Let $\Lambda$ be a compact leaf of $\mathcal L$.  Observe that $\Lambda$ bounds a compact region $\Theta$ and that $\Theta$ is a critical point of the $A^h$ functional.  The result now follows from the monotonicity of $h$ in the $z$ direction. More precisely, the vector field $X = (x,y,z)$ is a Killing field in the upper half-space model. Hence, if $\psi_t$ denotes the flow of $X$, it follows that 
\begin{align*}
0 = \frac{d}{dt}\eval_{t=0} A^h(\phi_t(\Theta)) &= \frac{d}{dt}\eval_{t=0} \operatorname{Area}(\phi_t(\Lambda)) - \frac{d}{dt}\eval_{t=0} \int_{\phi_t(\Theta)} h\\
&=\int_{\Lambda} h \la X,\nu\ra = \int_\Theta h\operatorname{div}(X) + \int_\Theta \la \grad h, X\ra = \int_\Theta \la \grad h,X\ra.
\end{align*}
Since $\la \grad h,X\ra \ge 0$, it follows that $\Omega$ is entirely contained in the region where $h\equiv c$. Thus $\Lambda$ has constant mean curvature $c$. Finally, it is well-known that a compact, embedded surface in hyperbolic space with constant mean curvature $c$ must be a sphere. 
\end{proof}

Do Carmo and Lawson \cite[Theorem A]{do1983alexandrov} proved that if $\Lambda$ is a complete surface in $\mathbb H^3$ with constant mean curvature and $\Lambda$ lies above a horosphere $\{z=z_0\}$ in the upper half-space model then $\Lambda$ must also be a horosphere.  In the following proposition, we note that their argument still applies to surfaces with prescribed mean curvature $h$, since $h$ is monotone in the $z$ direction. 

\begin{proposition}
If $\Lambda$ is a non-compact leaf in $\mathcal L$, then $\Lambda$ must be the horosphere $\{z = e\inv\}$. 
\end{proposition}

\begin{proof}
Assume $\Lambda$ is a non-compact leaf in $\mathcal L$. Let $\Theta$ be the unique region in $\mathbb H^3$ such that $\bd \Theta = \Lambda$ and $\inf_\Theta z > 0$. Given $q = (x_0,y_0)\in \R^2$ and $t > 0$, let 
\[
s(q,t) = \{(x,y,z): (x-x_0)^2+(y-y_0)^2+z^2 = t^2,\, z > 0\}.
\]
Then $s(q,t)$ is a totally geodesic surface in the upper half-space model.  

Fix a point $q\in \R^2$. Then $s(q,t) \cap \Lambda = \emptyset$ for sufficiently small $t$. Let $t_0(q)$ be the first time that $s(q,t) \cap \Lambda \neq \emptyset$. Then for $t > t_0(q)$, let $\Lambda_-(q,t)$ be the portion of $\Lambda$ lying inside $s(q,t)$, and let $\Lambda_+(q,t)$ be the portion of $\Lambda$ lying outside $s(q,t)$.  Finally, let $r_{q,t}$ be the hyperbolic reflection across $s(q,t)$, and let $\Lambda_-'(q,t)$ be the image of $\Lambda_-(q,t)$ under the reflection $r_{q,t}$. Suppose that there is a time $t > t_0(q)$ for which $\Lambda_-(q,t)'$ is not contained in the closure of $\Theta$. Then there is a first time $t_1 > t_0(q)$ at which the surfaces $\Lambda_-'(q,t_1)$ and $\Lambda_+(q,t_1)$ share a common point of tangency (possibly at the boundary), and moreover at this time $\Lambda_-'(q,t_1)$ lies within the closure of $\Theta$.  Now observe that $h(x,y,z) \ge h(r_{q,t_1}(x,y,z))$ for all $(x,y,z)$ lying outside $s(q,t_1)$. This is because $h$ is monotone in the $z$ direction and the hyperbolic reflection decreases the $z$ coordinate of points $(x,y,z)$ lying outside $s(q,t_1)$. Since  $\Lambda_-'(q,t_1)$ has prescribed mean curvature $h(r_{q,t_1}(x,y,z))$, and $\Lambda_+(q,t_1)$ has prescribed mean curvature $h(x,y,z)$, and $\Lambda_-'(q,t_1)$ lies to the mean convex side of $\Lambda_+(q,t_1)$, and $h(r_{q,t_1}(x,y,z)) \le h(x,y,z)$, it now follows by the maximum principle that $\Lambda_-'(q,t_1)$ coincides with $\Lambda_+(q,t_1)$.  But this implies that $\Lambda$ is compact, contrary to assumption. 

Hence $\Lambda_-'(q,t)$ is contained within the closure $\Theta$ for all $t > t_0(q)$. Since this holds for all $q$, it now follows as in the proof of \cite[Theorem A]{do1983alexandrov} that $\Lambda$ must be a horosphere $\{z = z_0\}$ for some choice of $z_0$. As $\Lambda$ has mean curvature prescribed by $h$ and $\Lambda$ lies above $\{z = e\inv\}$ the only possibility is that $\Lambda = \{z=e\inv\}$. This proves the proposition. 
\end{proof}

Next we show that if the lamination $\mathcal L$ contains the horosphere $\{z = e\inv\}$, then this is the only leaf of $\mathcal L$. 

\begin{proposition}
If $\{z=e\inv\}$ is a leaf in $\mathcal L$, then this is the only leaf in $\mathcal L$. 
\end{proposition}

\begin{proof}
Suppose $\{z=e\inv\}$ is a leaf in $\mathcal L$. Recall that the pointed surfaces $(\Sigma_{R_j},p_j)$ converge locally smoothly to $\mathcal L$ away from at most one point. Since $\Sigma_{R_j}$ bounds a compact region $\Omega_{R_j}$ in $M$, it follows that $\mathcal L$ is the boundary of a region $\Theta$ in the upper half-space with $\inf_\Theta z > 0$.  Moreover, the mean curvature of every leaf in $\mathcal L$ points into $\Theta$. 

Note that any other leaf of $\mathcal L$ must be a sphere of mean curvature $c$ contained in the region $\{1\le z \le e^d\}$.  Suppose for contradiction that $\mathcal L$ contains such a sphere $S$. Since the mean curvature of this sphere points into $\Theta$, it follows that $\Theta$ contains the ball $B$ enclosed by $S$. However, the intersection number of a generic path $\gamma$ connecting the center of $B$ to the point $(0,0,(2e)\inv)$ with the lamination $\mathcal L$ is even. Therefore $\Theta$ must contain $\{z < e\inv\}$, and this is a contradiction. 
\end{proof}

Finally, we use the above results to analyze the behavior of $\Sigma_R = \bd \Omega_R$ as $R\to \infty$. Choose a sequence $R_j\to \infty$ and sequence of points $p_j\in S_{R_j}$ so that the surfaces $(\Sigma_{R_j},p_j)$ converge in the pointed sense to a limiting lamination $\mathcal L(\{R_j\},\{p_j\})$. 

\begin{proposition}
Assume that the limiting lamination $\mathcal L(\{R_j\},\{p_j\})$ is the horosphere $\{z=e\inv\}$. Then for every other choice of points $x_j\in S_{R_j}$, the surfaces $(\Sigma_{R_j},x_j)$ converge in the pointed sense to a limiting lamination $\mathcal L(\{R_j\},\{x_j\})$ which is also the horosphere $\{z=e\inv\}$. 
\end{proposition}

\begin{proof}
Choose a sequence $x_j\in S_{R_j}$. After passing to a subsequence $x_{j_k}$, we obtain a limiting lamination $\mathcal L(\{R_{j_k}\},\{x_{j_k}\})$. The fact that $\mathcal L(\{R_{j}\},\{p_{j}\})$ is a horosphere implies $\Omega_{R_{j}}$ contains $B_{R_{j}-d}$. Thus the lamination $\mathcal L(\{R_{j_k}\},\{x_{j_k}\})$ must separate $\{z= (2e)\inv\}$ from $\{z = 2e^d\}$. This is only possible if $\mathcal L(\{R_{j_k}\},\{x_{j_k}\})$ is the horosphere $\{z=e\inv\}$.  Since this is true no matter which subsequence we pick, it follows that $\mathcal L(\{R_j\},\{x_j\})$ exists and is also the horosphere $\{z=e\inv\}$. 
\end{proof}

Finally, we can derive a contradiction. 
First, we rule out the possibility of seeing a horosphere.  Indeed, suppose that for $R_j$ and $p_j$ as above, the limiting lamination $\mathcal L(\{R_j\},\{p_j\})$ is a horosphere. Then we know that for all other choices of points $q_j\in S_{R_j}$ the limiting lamination is also a horosphere. This implies that $\Sigma_{R_j}$ is the graph of a function $u_j$ over $S_{R_j+1}$ with small $C^1$ norm. It follows that, for any $\eps > 0$, we have $\area(\Sigma_{R_j}) \le (1+\eps)\area(S_{R_j+1})$ for sufficiently large $j$. Moreover, $\Omega_{R_j}$ contains  $B_{R_j}$. Since $M$ is asymptotically hyperbolic and $c > 2$, this implies that 
$
A^c(\Omega_{R_j}) < 0,
$
for sufficiently large $j$. This contradicts that $A^c(\Omega_{R_j}) = \omega_R \to \omega_c(M) > 0$. 

Hence, for every choice of $R_j$ and $p_j$ for which the limiting lamination $\mathcal L(\{R_j\},\{p_j\})$ exists, the lamination $\mathcal L(\{R_j\},\{p_j\})$ is either empty or a collection of spheres with mean curvature $c$. It is straightforward to see that this implies $\liminf_{R\to \infty} A^c(\Omega_R) \ge \omega_c(\mathbb H^3)$, which is a contradiction.  Hence, for some large $R$, the surface $\Sigma_R$ is contained in $B_R$. It follows that $\Sigma_R$ is a compact, almost-embedded, free boundary surface in $M$ with constant mean curvature $c$. This completes the proof of Theorem \ref{min-max-hyp}. 

\subsubsection{Proof of Main Theorem}
Finally, we can give the proof of our second main theorem. 

\begin{theorem}
    Let $(M^3,g)$ be a complete, asymptotically hyperbolic manifold, possibly with non-empty boundary. Assume that $M$ has scalar curvature at least $-6$. Then, for every constant $c > 2$, there exists a compact, almost-embedded, free boundary, constant mean curvature surface $\Sigma$ in $M$ with mean curvature $c$. 
\end{theorem}

\begin{proof}
    Let $M$ be as in the statement of the theorem and fix a constant $c > 2$. If $M$ has a hyperbolic end, then we can choose $\Sigma$ be a hyperbolic sphere of constant mean curvature $c$ in the end.  Hence we can assume that some end of $M$ is not identically hyperbolic. Then according to Proposition \ref{Prop: sweep-out AH}, for any $v > 0$ there is a continuous family of open sets $\{\Omega_t\}_{t\in [0,1]}$ with $\Omega_0 = \emptyset$ and $\vol(\Omega_1) = v$ and 
    \[
    \operatorname{Area}(\bd \Omega_t) < \mathcal I_{hyp}(\vol(\Omega_t)). 
    \]
    This implies that 
    \[
    \sup_{t\in [0,1]} A^c(\Omega_t) < \omega_c(\mathbb H^3). 
    \]
    Hence we can apply Theorem \ref{min-max-hyp} and the result follows. 
\end{proof}

\appendix
\section{Evolution equations}
\subsection{Evolution equations along inverse mean curvature flow}
Let $(\mathbb R^n,\bar g)$ be a Riemannian manifold and $\Phi:N\times [0,T)\to (\mathbb R^n,\bar g)$ be a smooth inverse mean curvature flow satisfying $\partial_t\Phi=H^{-1}\nu$, where $H$ is the mean curvature of $\Phi(N\times\{t\})$ with respect to the unit normal vector field $\nu$.
\begin{lemma}\label{Lem: evolution equation IMCF}
    Denote $g$ and $h$ to be the induced metric and the second fundamental form with respect to $\nu$ of $\Phi(N\times\{t\})$ respectively. Then we have the following evolution equations
    \begin{itemize}
        \item[(i)] $\partial_tg_{ij}=2H^{-1}h_{ij}$;
        \item[(ii)] $\partial_t\nu=-\nabla H^{-1}$;
        \item[(iii)] $\partial_t\mu_t=\mu_t$;
        \item[(iv)] $\partial_t{H^{-1}}=H^{-2}\Delta H^{-1}+H^{-3}(|h|^2+\overline\Ric(\nu))$; 
    \end{itemize}
\end{lemma}
\begin{proof}
  This follows from a direct computation.
\end{proof}

\subsection{Evolution equations along mean curvature flow}
Let $(\mathbb R^n,\bar g)$ be a Riemannian manifold and $\Phi:N\times [0,T)\to (\mathbb R^n,\bar g)$ be a smooth mean curvature flow satisfying $\partial_t\Phi=-H\nu$, where $H$ is the mean curvature of $\Phi(N\times\{t\})$ with respect to the unit normal vector field $\nu$.
 
\begin{lemma}\label{Lem: evolution equation MCF}
Denote $g$ and $h$ to be the induced metric and the second fundamental form with respect to $\nu$ of $\Phi(N\times\{t\})$ respectively.  Then we have
    \begin{itemize}
        \item[(i)] $\partial _tg_{ij}=-2Hh_{ij}$;
        \item[(ii)] $\partial_t\nu=\nabla H$;
        \item[(iii)] $\partial_t\mu_t=-H^2\mu_t$;
        \item[(iv)] $$\partial_t h_{ij}=\Delta h_{ij}-2Hh_{il}h^l_j+|h|^2h_{ij}+h\ast \overline{\Rm}+\bar\nabla\overline{\Rm};$$
        \item[(v)] $\partial_t H=\Delta H+H(|h|^2+\overline\Ric(\nu))$;
        \item[(vi)] \[
\begin{split}\partial_t|h|^2&=\Delta|h|^2-2|\nabla h|^2+2|h|^2(|h|^2+\overline\Ric(\nu))\\
        &\qquad+\overline\Rm\ast h\ast h+\bar\nabla\overline\Rm\ast h.\end{split}\]
    \end{itemize}
\end{lemma}
\begin{proof}
    See \cite{huisken1986MCF}.
\end{proof}

\end{document}